\documentclass[a4paper]{amsart}
\usepackage{comment}
\excludecomment{AKT}
\newenvironment{ARXIV}{}{}
\newcommand{\ARXIVONLY}[1]{#1}
\newcommand{\AKTONLY}[1]{}
\usepackage[top=2.5cm,bottom=2.5cm,left=4cm,right=4cm]{geometry}
\usepackage[utf8]{inputenc}
\usepackage[T1]{fontenc}
\usepackage{lmodern,microtype,etoolbox}
\usepackage[british]{babel}
\newcommand{\stacktag}[1]{\href{https://stacks.math.columbia.edu/tag/#1}{\texttt{#1}}}
\usepackage[all]{xy}
\usepackage{mathtools}
\usepackage{amssymb}
\usepackage{amsthm}
\usepackage{mathrsfs}
\usepackage{changepage} 
\usepackage{paralist}
\usepackage{varioref}
\usepackage{prettyref}
\newrefformat{fig}{Figure~\ref{#1} \vpageref[]{#1}}
\newrefformat{eg}{Example~\ref{#1}}
\usepackage{hyperref}
\hypersetup{breaklinks=true,colorlinks=true,linkcolor=black,anchorcolor=black,citecolor=black,filecolor=black,menucolor=black,urlcolor=black}
\usepackage{color}

\swapnumbers
\newtheorem{theorem}{Theorem}[section]
\newrefformat{thm}{\hyperref[{#1}]{Theorem~\ref*{#1}}}
\newtheorem{definition}[theorem]{Definition}
\newrefformat{def}{\hyperref[{#1}]{Definition~\ref*{#1}}}

\newrefformat{note}{\hyperref[{#1}]{Note~\ref*{#1}}}

\newrefformat{setup}{\hyperref[{#1}]{Setup~\ref*{#1}}}
\newtheorem{lemma}[theorem]{Lemma}
\newrefformat{lem}{\hyperref[{#1}]{Lemma~\ref*{#1}}}
\newtheorem{proposition}[theorem]{Proposition}
\newrefformat{prop}{\hyperref[{#1}]{Proposition~\ref*{#1}}}
\newtheorem{corollary}[theorem]{Corollary}
\newrefformat{cor}{\hyperref[{#1}]{Corollary~\ref*{#1}}}

\theoremstyle{remark}
\newtheorem{remark}[theorem]{Remark}
\newrefformat{rem}{\hyperref[{#1}]{Remark~\ref*{#1}}}
\newtheorem{example}[theorem]{Example}
\newrefformat{ex}{\hyperref[{#1}]{Example~\ref*{#1}}}

\newrefformat{not}{\hyperref[{#1}]{Notation~\ref*{#1}}}

\newcommand{\op}{\operatorname}
\newcommand{\colim}{\op{colim}}         
\newcommand{\id}{\op{id}}               
\newcommand{\sing}{\mathrm{sing}}       
\newcommand{\et}{\mathrm{\acute{e}t}}   
\renewcommand{\r}{{\op{r}}}             
\newcommand{\ret}{\mathrm{r\acute{e}t}} 
\newcommand{\fl}{\mathrm{fl}}           
\renewcommand{\P}{\mathbb{P}}
\newcommand{\Z}{\mathbb{Z}}
\newcommand{\A}{\mathbb{A}}
\newcommand{\R}{\mathbb{R}}
\newcommand{\C}{\mathbb{C}}

\newcommand{\Gm}{\mathbb{G}_{\mathrm{m}}}
\newcommand{\OO}{\mathcal{O}}          
\DeclareMathOperator{\Hom}{Hom}

\DeclareMathOperator{\support}{supp}

\DeclareMathOperator{\vcd}{vcd}
\DeclareMathOperator{\Spec}{Spec}
\newcommand{\noloc}{:\!}
\newcommand*{\cat}[1]{\mathcal #1}
\newcommand*{\sheaf}[1]{\mathcal #1}

\newcommand{\pfist}[1]{\langle\!\langle #1 \rangle\!\rangle} 
\newcommand{\sign}{\op{sign}} 
\newcommand{\realcycle}{\op{rcl}} 
\newcommand{\Sper}{\op{sper}} 
\newcommand{\cont}{\mathcal C}        
\newcommand{\god}{\mathbb G} 
\newcommand{\res}{\sheaf R}  

\usepackage{graphicx,pbox}
\newcommand{\cbox}[2][]{\pbox[#1]{\textwidth}{\relax\ifvmode\centering\fi#2}}
\newcommand*{\downcong}{\rotatebox[origin=c]{-90}{$\cong$}}

\newcommand{\gersten}[5][]{\op{C}^{#2}(#3,#4\ifstrempty{#5}{}{(#5)})_{#1}}     
\newcommand{\Icohom}[5][]{\op{H}^{#2}_{#1}(#3,#4(#5))} 
\newcommand{\I}{\mathbf{I}}                   
\newcommand{\W}{\mathbf{W}} 
\newcommand{\Lb}{\mathcal{L}} 
\newcommand{\Nb}{\mathcal{N}} 

\title{The real cycle class map}
\author{Jens Hornbostel, Matthias Wendt, Heng Xie and Marcus Zibrowius}
\begin{document}

\address{Jens Hornbostel, Fachgruppe Mathematik/Informatik, Bergische Universit\"at Wuppertal, Gau\ss{}stra\ss{}e~20, 42119 Wuppertal, Germany}
\email{hornbostel@math.uni-wuppertal.de}

\address{Matthias Wendt, Fachgruppe Mathematik/Informatik, Bergische Universit\"at Wuppertal, Gau\ss{}stra\ss{}e~20, 42119 Wuppertal, Germany}
\email{wendt@math.uni-wuppertal.de}

\address{Heng Xie, Fachgruppe Mathematik/Informatik, Bergische Universit\"at Wuppertal, Gau\ss{}stra\ss{}e~20, 42119 Wuppertal, Germany}
\email{heng.xie@math.uni-wuppertal.de}

\address{Marcus Zibrowius, Mathematisches Institut, Heinrich-Heine-Universit\"at Düsseldorf, Universitäts{}stra\ss{}e~1, 40225 D\"usseldorf, Germany}
\email{marcus.zibrowius@cantab.net}

\begin{abstract}
The classical cycle class map for a smooth complex variety sends cycles in the Chow ring to cycles in the singular cohomology ring.  We study two cycle class maps for smooth real varieties: the map  from the \(\I\)-cohomology ring to singular cohomology induced by the signature, and a new cycle class map defined on the Chow-Witt ring. For both maps, we establish basic compatibility results like compatibility with pullbacks, pushforwards and cup products.  As a first application of these general results, we show that both cycle class maps are isomorphisms for cellular varieties.
\end{abstract}

\maketitle
\tableofcontents

\section{Introduction}
Classically, for smooth varieties $X$ over a field $F$ with a complex embedding $\sigma\colon F\hookrightarrow \mathbb{C}$, there are cycle class maps $\op{CH}^k(X)\to \op{H}^{2k}(X(\mathbb{C}),\Z)$ from Chow groups to singular cohomology, and these are compatible with pullbacks, pushforwards and intersection products.

In the present article, we consider real cycle class maps in the analogous situation for real varieties.
Note that integral Chow groups alone are not the right object to look at when studying real varieties.  For example, it is known that the rational motivic Eilenberg--Mac Lane spectrum maps to zero under real realization.  Rather, the general philosophy due to Morel, Fasel and others is that we should study \(\I\)-cohomology: sheaf cohomology with coefficients in the sheafified powers \(\I^k\) of the fundamental ideal of the Witt ring (see Definitions~\ref{def:W-and-I-sheaves} and \ref{def:I-cohomology} below).  It is this theory that is to singular cohomology of real varieties what Chow groups are to singular cohomology for complex varieties.

Concretely, for varieties $X$ over a field $F$ with a real embedding $\sigma\colon F\hookrightarrow \R$ (or more generally, an embedding into a real closed field) there are {\em real cycle class maps} $\op{H}^k(X,\I^k)\to \op{H}^k(X(\R),\Z)$. These were discussed in the work of Jacobson \cite{jacobson}, and they lift the more classical maps constructed by Borel--Haefliger and Scheiderer.

The main goal of this article is to study this real cycle class map in detail. In particular, we will show that
it is compatible with all the structure we are interested in, notably pushforwards, pullbacks
and ring structures.
We also study the cycle class map with twisted coefficients $\I^k(\Lb)$, whose target is singular cohomology with coefficients in local coefficient systems arising from real realization of the twisting line bundle.

One important application of the compatibility results for the real cycle class maps is that for a smooth cellular variety $X$ over $\R$, the real cycle class map induces a ring isomorphism
\[
  \bigoplus_{k,\mathcal{L}}\op{H}^k(X,\mathbf{I}^k(\mathcal{L}))\to \bigoplus_{n,\mathcal{L}}\op{H}^k_\sing(X(\R),\Z({\mathcal{L}}))
\]
where $\Z({\mathcal{L}})$ denotes coefficients in the local system associated to the line bundle $\mathcal{L}$. More generally, the result holds over real closed fields, using semi-algebraic cohomology instead of singular cohomology.

Another important goal of the present article is to extend the real cycle class map to a {\em full cycle class map} on the twisted Chow--Witt groups $\smash{\widetilde{\op{CH}}}^k(X,\mathcal{L})$ and establish the compatibility of the cycle class maps with pushforwards, pullbacks and intersection products in Chow--Witt theory. The target of this full cycle class map combines both the singular cohomology of $X(\R)$ with integral coefficients (possibly twisted by a local system) and the $\op{C}_2$-equivariant cohomology of $X(\mathbb{C})$, in the same way that Milnor--Witt K-theory combines the sheaves $\I^k$ and $\mathbf{K}^{\op{M}}_k$,
and the full cycle class map refines the recently defined
$\op{C}_2$-equivariant cycle class map of Benoist--Wittenberg \cite{benoist:wittenberg}.

The results concerning real realization and real cycle class maps discussed here were motivated by recent computations of Chow--Witt rings and \(\I\)-cohomology rings of classifying spaces \cite{hornbostelwendt}, Grassmannians \cite{real-grassmannian} and split quadrics \cite{HXZ:quadrics}. In particular, in the present preprint we discuss the compatibility of real realization with the characteristic classes left open  both in \cite{real-grassmannian} and in \cite{hornbostelwendt}, cf.\ in particular \cite[Remark 9.2]{hornbostelwendt}.

\medskip

For most of our considerations we deal only with ``geometric'' bidegrees, that is, we concentrate on $\op{H}^k(X,\I^k)$ or $\smash{\widetilde{\op{CH}}}^k(X)=\op{H}^k(X,\mathbf{K}^{\op{MW}}_k)$ and ignore the full bigraded theories (Milnor--Witt motivic cohomology $\op{H}^i(X,\widetilde{\Z}(j))$, etc.) that have recently been developed and studied by Fasel and others.

\begin{remark} The focus of this article is on schemes over $\R$ and realization to classical topological manifolds. However, several results refine to the real \'etale site
and the real topology, see Sections \ref{sec:prelim-ret} and \ref{compare:retsing}, Remarks \ref{retvariant} and \ref{rem:base-fields-for-cellular-applications} and Appendix \ref{sec:app}. To start with,
 there are evident extensions of the results to the case of real closed fields, using the theory of semi-algebraic cohomology as discussed in \cite{delfs}. It is also possible to go beyond the real closed case for parts of the results. In full generality, one has to replace singular cohomology by real-\'etale cohomology as in \cite{scheiderer}. In this setting, it is possible to define cycle class maps $\op{H}^i(X,\I^j)\to \op{H}^i_{\ret}(X,\mathbf{Z})$ and prove all the relevant compatibilities. For a formally real field $F$, the kernel of the natural homomorphism
  \[
    \op{W}(F) \to \prod_P \op{W}(F_P) =\op{H}^0_{\ret}(F,\Z)
  \]
  is equal to the torsion subgroup of $\op{W}(F)$, where the product runs over all orderings $P$ of the field $F$ and $F_P$ denotes the real closure for $F$ with the ordering $P$. In particular, to determine an element up to the torsion, it suffices to consider all the signatures. Consequently, claims concerning the image of an element in real cohomology over a formally real field may then be deduced from the corresponding claims for real closed fields (whose proofs can be obtained from the case $\R$ discussed in the present article). Lacking significant applications for this extension, we won't go into further details here.
\end{remark}

\begin{remark}
  At least for smooth varieties over $\R$, taking real
  points is expected to induce a morphism of spectral sequences from the Gersten--Witt spectral sequence defined by Balmer and Walter \cite{BW02} to the Atiyah--Hirzebruch spectral sequence for topological KO-theory, each potentially twisted by a line bundle in an appropriate sense. 
  On $E_2$-terms in degree zero, this morphism is the (twisted) signature map studied here, and on $E_{\infty}$-terms it is expected to be (a twisted version of) the ``signature map'' that appears in  Brumfiel's Theorem \cite{brumfiel,karoubi-schlichting-weibel}. It thus seems plausible that variants of Brumfiel's Theorem could be obtained by such a comparison of spectral sequences.  However, a careful implementation of such arguments requires higher signature maps, which are beyond the scope of this paper.  We refer to \cite{jacobson-pre2} for some work in this direction, and for related questions.
\end{remark}

\medskip

\emph{Structure of the paper:} We provide a recollection on the various cohomology theories involved in \prettyref{sec:prelims}. \prettyref{sec:cyclesheaf} discusses sheaf-theoretic and geometric constructions of the real cycle class maps and shows that they agree. The compatibility results with pullbacks, pushforwards and intersection products are proved in \prettyref{sec:compatibilities}. The special case of cellular varieties is discussed in \prettyref{sec:cellular}. The comparison of algebraic and topological characteristic classes is established in Section~\ref{sec:characteristic}.

\emph{Acknowledgements:} The first, third and fourth author
were supported by the DFG-priority program 1786 grants HO~4729/3-1 and ZI~1460/1-1. This
research was conducted in the framework of the DFG-funded research training group
\emph{GRK 2240: Algebro-Geometric Methods in Algebra, Arithmetic and Topology}. The third author was supported by EPSRC Grant EP/M001113/1, and he would also like to
thank Max Planck Institute for Mathematics in Bonn for its hospitality. 
We thank Marco Schlichting for helpful discussions related to the twisted signature,
Jeremy Jacobsen for answering a question of us,%
\begin{AKT}
  and the referee for their careful reading and very useful suggestions.
  Following the advice of the managing committee, section~2 of the published version is considerably shorter than it was in previous variants.
\end{AKT}
\begin{ARXIV}
  and the referee of \textit{Annals of K-Theory} for their careful reading and very useful suggestions.
  Following the advice of the managing committee of \textit{Annals of K-Theory}, section~2 of the published version is considerably shorter than it is in this preprint.
\end{ARXIV}

\section{Preliminaries and notation}
\label{sec:prelims}

In this article, all schemes are assumed to be separated of finite type over a field $F$ which has a real embedding $\sigma\colon F\hookrightarrow \R$, although many results obviously hold in greater generality. In particular, the conditions of being perfect, infinite and of characteristic $\neq 2$ which frequently occur in the literature on $\mathbb{A}^1$-homotopy and Chow--Witt theory are automatically satisfied. Occasionally, we might refer to reduced schemes as varieties.

We write \(\cat Ab(X)\) for the category of sheaves of abelian groups on a topological space \(X\).  Given a continuous map \(f\colon X\to Y\), we write \(f_*\) and \(f^*\) for the components of the canonical adjunction \(\cat Ab(X)\leftrightarrows \cat Ab(Y)\).  (Note that some authors prefer to write \(f^{-1}\) for the left adjoint and reserve the notation \(f^*\) for \(\sheaf O\)-modules over ringed spaces.)

\subsection{Recollections on sheaf cohomology}
The \textbf{sheaf cohomology} of a topological space \(X\) with coefficients in a sheaf of abelian groups \(\sheaf F\) on \(X\) is given by the right derived functors of the global sections functor \(\Gamma\):
\begin{ARXIV}
\[
  \op{H}^i(X,\sheaf F):= (\op{R}^i\Gamma)(\sheaf F)
\]
\end{ARXIV}
\begin{AKT}
  \(\op{H}^i(X,\sheaf F):= (\op{R}^i\Gamma)(\sheaf F)\).
\end{AKT}
Thus, sheaf cohomology depends on a pair \((X,\sheaf F)\) consisting of a topological space \(X\) and a sheaf of abelian groups \(\sheaf F\), and we refer to such a pair \((X,\sheaf F)\) as \textbf{coefficient datum}.  We define a \textbf{morphism of coefficient data} \((f,\phi)\colon (X,\sheaf F_X) \to (Y,\sheaf F_Y)\) to be a pair consisting of a morphism of spaces \(f\colon X\to Y\) and a morphism of sheaves \(\phi\colon f_*\sheaf F_X\leftarrow \sheaf F_Y\).  There's an obvious notion of composition of such morphisms:
\begin{equation}\label{eq:composition-of-morphisms-of-coefficient-data}
  (g,\gamma)\circ(f,\phi) = (gf,g_*\phi\circ\gamma).
\end{equation}
\begin{example}
  A morphism of ringed spaces \((X,\sheaf O_X)\to (Y,\sheaf O_Y)\) is, in particular, a morphism of coefficient data in the above sense.
\end{example}
\begin{example}\label{eg:canonical-morphisms-of-coefficient-data}
  For any morphism of spaces \(f\colon X\to Y\) and arbitrary sheaves \(\sheaf F_X\) and \(\sheaf F_Y\) on \(X\) and \(Y\), respectively, we have canonical associated morphisms of coefficient data \((X,\sheaf F_X)\to (Y,f_*\sheaf F_X)\) and \((X,f^*\sheaf F_Y)\to (Y,\sheaf F_Y)\).
\end{example}
\begin{remark}\label{rem:cohomomorphisms}
In \cite{bredon}, the component \(\phi\) of a morphism of coefficient data \((f,\phi)\) is referred to as an ``\(f\)-cohomomorphism''.  We will not use this terminology.
\end{remark}
In order to compute sheaf cohomology with coefficients in \(\sheaf F\) from the definition, one needs to first find an injective resolution of \(\sheaf F\).  However, one gets the same result using resolutions by flasque or more generally by acyclic sheaves \cite[Proposition~III.1.2A]{hartshorne}, and this additional flexibility is often useful.  The Godement resolution at the centre of this section is one particularly well-behaved choice of flasque resolution.

\subsubsection{Godement resolutions}\label{sec:construction-Godement}
\begin{ARXIV}
The classical and established sources for Godement resolutions are \cite{godement}, \cite{iversen} and \cite{bredon}.  We are mainly interested in a good definition of cup and cross products, but we provide a more general overview purely for the reader's convenience.

Let \(X\) be a topological space.  By a \textbf{resolution} \(\res_{\sheaf F}\) of a sheaf \(\sheaf F\) on \(X\) we mean a complex of sheaves \(\res_{\sheaf F}\) equipped with a morphism to \(\sheaf F\to \res_{\sheaf F}\) such that \(0\to\sheaf F\to\res_{\sheaf F}\) is exact.  Following \cite[II.6.4 and Appendice~\S\,3]{godement}, the \textbf{Godement resolution} \(\sheaf F\to\god\sheaf F\) can be constructed as follows:
Let \(X^\delta\) denote \(X\) equipped with the discrete topology.  The canonical continuous map \(d\colon X^\delta \to X\) induces an adjunction between the associated categories of sheaves of abelian groups:
\(
d_*\colon \cat Ab(X^\delta) \leftrightarrows \cat Ab(X) \noloc d^*
\). This induces a monad  \(G := d_*d^*\) which for an abelian sheaf \(\sheaf F\) on \(X\) is concretely given as
\begin{equation}\label{eq:Go0-concrete}
  G\sheaf F = \prod_{x\in X^\delta}\sheaf F_x,
\end{equation}
where \(\sheaf F_x\) denotes the (skyscraper sheaf associated with) the stalk of \(\sheaf F\) at \(x\).
Associated with this monad $G$, we have a canonical simplicial resolution \(\sheaf F \to \god^\star\sheaf F\) of abelian sheaves on \(X\), given in degree~\(n\) by \(\god^n\sheaf F := G^{n+1}(\sheaf F)\).  The Godement resolution \(\sheaf F \to \god\sheaf F\) is the chain complex associated with this simplicial resolution by the construction of \cite[8.6.15]{weibel}.   Its many nice features include:  the resolution is flasque, it is functorial in \(\sheaf F\), and at each point \(x\in X^\delta\), the complex of stalks \(\sheaf F_x \to \god \sheaf F_x\) is not only exact but even chain nullhomotopic.

\begin{remark}\label{rem:G1-versus-G2}
  There is another common resolution, also known as Godement resolution or canonical resolution, and introduced by Godement in the same book \cite[II.4.3]{godement}.  Godement writes \(\mathcal F^*(X;\sheaf F)\) for the resolution we denote as \(\god\sheaf F\) here, and \(\mathcal C^*(X;\sheaf F)\) for the resolution we denote \(\god'\sheaf F\).   This latter resolution \(\god'\sheaf F\) is used in \cite[II.2]{bredon}, for example, and also described in \cite[II.3.6]{iversen}.
  The two resolutions are closely related: there is a canonical injection \(\god'\sheaf F\to \god \sheaf F\), and this injection induces an isomorphism on the associated cohomology groups (see \cite[II.6.4 (d) (I)]{godement}).
\end{remark}

\begin{remark}[sites]
  The Godement resolution can be defined more generally for a site \(X\) with a set of conservative points \(X'\). Indeed, in this situation we have a geometric morphism of topoi \(\cat Sh(X') \to \cat Sh(X)\) (cf.\ \cite[after Def.~7.1.1]{johnstone}) and hence an adjunction \(\cat Ab(X') \leftrightarrows \cat Ab(X)\) that allows us to define a resolution exactly as above (cf.\ \cite[Theorem~8.20]{johnstone}).  Thus, when working with schemes, Godement resolutions can be used to compute not just Zariski sheaf cohomology but also étale or Nisnevich sheaf cohomology. However, in this paper we will have no use for this generality.
  \end{remark}

As the Godement resolution \(\god\sheaf F\) is flasque, we can use it to compute sheaf cohomology with coefficients in \(\sheaf F\).  To be more precise,  let us temporarily denote cohomology computed using Godement resolutions as \(\op{H}^i_{\text{god}}(X,\sheaf F) := \op{H}^i(\Gamma\god\sheaf F)\), and cohomology computed using injective resolutions as \(\op{H}^*_{\text{inj}}(X,-)\).  For both cohomologies, we have canonical identifications \(\op{H}^0_{\text{god}}(X,\sheaf F)\cong \Gamma\sheaf F\) and \(\op{H}^0_{\text{inj}}(X,\sheaf F)\cong \Gamma\sheaf F\), and for both cohomologies we have long exact sequences associated with short exact sequences of coefficient sheaves (see \prettyref{prop:Go-sheaf-exact-sequence} below for Godement cohomology and \cite[Theorem~III.1.1A]{hartshorne} for injective cohomology).
\begin{lemma}[Agreement]
  \label{lem:sheaf-cohomology-unique}
  \begin{enumerate}[(a)]\label{enum:scu:godement-is-injective}
  \item
    There is a canonical family of natural isomorphism \(\op{H}^i_{\mathrm{god}}(X,-)\cong \op{H}^i_{\mathrm{inj}}(X,-)\) compatible with the identifications of degree zero cohomology with the global sections functor and with the boundary maps in long exact cohomology sequences.
  \item \label{enum:scu:canonical-map-to-Go-cohomology}
    More generally, for an \emph{arbitrary} resolution \(\sheaf F\to \res_{\sheaf F}\), we have a canonical homomorphism
    \(
    \op{H}^*(\Gamma\res_{\sheaf F}) \to \op{H}^*_{\mathrm{god}}(X,\sheaf F),
    \)
    and this is an isomorphism whenever the sheaves comprising \(\res_{\sheaf F}\) are acyclic.
  \end{enumerate}
\end{lemma}
\begin{proof}
  See \cite[II.4.7]{godement} and \cite[II.4.8]{godement}.
\end{proof}
Given the lemma, we will again drop the subscripts on cohomology and simply write
\[
  \op{H}^i(X,\sheaf F) := \op{H}^i(\Gamma\god\sheaf F)
\]
from now on.
\begin{lemma}\label{lem:Go-resolution-functorial}
  For a continuous map \(f\colon X\to Y\) there is a natural transformation \(\god f_* \to f_* \god\) such that, if we write \(\mu_X\) for the coaugmentation \(\mu_X\colon \sheaf F\to \god\sheaf F\) of the Godement resolution of a sheaf \(\sheaf F\) on \(X\), the following diagram commutes:
  \[
    \xymatrix{
     f_* \sheaf F\ar[r]_{\mu_Y} \ar@{=}[d] & \god f_*\sheaf F\ar[d] \\
     f_* \sheaf F\ar[r]_{f_*(\mu_X)} & f_*\god \sheaf F
    }
  \]
\end{lemma}
\begin{proof}
  We have a commutative square of continuous maps as follows:
  \[\xymatrix@R=8pt{
      X^\delta\ar[d]^{d_X} \ar[r]^{f^\delta} & Y^\delta \ar[d]^{d_Y} \\
      X \ar[r]^{f} & Y
    }\]
  Pass to abelian sheaves, and let \(G_X\) and \(G_Y\) be the monads on \(\cat Ab(X)\) and \(\cat Ab(Y)\) defined by \(d_X\) and \(d_Y\), respectively. Then \(f_*\) defines a lax map of monads \(G_X\to G_Y\) in the sense of \cite[II.6.1]{leinster}, i.e.\ we have a natural transformation \(\psi\colon f_*G_X\to G_Yf_*\) compatible with the multiplication maps \(G_X^2\to G_X\) and \(G_Y^2\to G_Y\) and the unit maps \(\mu_X\colon \id\to G_X\) and \(\mu_Y\colon \id \to G_Y\). Explicitly, if we denote by \(\eta_Y\) the counit of the adjunction \(d_Y^* \dashv (d_Y)_*\) and by \(\phi\) the natural isomorphism \(f_*(d_X)_*\to (d_Y)_*f_*^\delta\), the lax map of monads can be defined as the composition \(\psi:= \nu\circ G_Yf_*(\mu_X)\), where \(\nu:= (\phi_{d^*})^{-1} \circ d_*((\eta_Y)_{f_*^\delta d^*}) \circ G_Y(\phi_{d^*})\).  The axioms of a lax map of monads can be verified by noting that \(\nu\circ (\mu_Y)_{f_*G} = \id\).  As the simplicial boundary and degeneracy maps on \(\god_X(-)\) and \(\god_Y(-)\) are defined purely in terms of the monad structures on \(G_X(-)\) and \(G_Y(-)\), the lax map of monads \(\psi\) induces a natural transformation \(f_*\god_X\to \god_Y f_*\) as required.
\end{proof}
It follows in particular that a morphism of coefficient data \((f,\phi)\colon (X,\sheaf F)\to (Y,\sheaf G)\) induces a morphism of resolutions, cf.\ \cite[II.8]{bredon} for a direct construction:
\[
  f_*\god \sheaf F \xleftarrow{\text{\prettyref{lem:Go-resolution-functorial}}} \god f_*\sheaf F \xleftarrow{\god \phi} \god \sheaf G
\]
\end{ARXIV}
\begin{AKT}
We refer to the classical and established sources for the Godement resolution \(\sheaf F\to \god\sheaf F\) of a sheaf \(\sheaf F\): \cite{godement}, \cite{iversen} and \cite{bredon}.  Here, we are mainly interested in a good definition of cup and cross products.  First, let us note that a morphism of coefficient data \((f,\phi)\colon (X,\sheaf F)\to (Y,\sheaf G)\) induces a morphism of resolutions, cf.\ \cite[II.8]{bredon} for a direct construction:
\[
  f_*\god \sheaf F  \leftarrow \god \sheaf G
\]
\end{AKT}
\begin{definition}\label{def:sheaf-pullback}
  Given a morphism of coefficient data \((f,\phi)\colon (X,\sheaf F)\to (Y,\sheaf G)\), we define the \textbf{pullback in sheaf cohomology}
  \[
    \op{H}^*(X,\sheaf F) \xleftarrow{(f,\phi)^*} \op{H}^*(Y,\sheaf G)
  \]
  to be the morphism induced by passing to global sections \ARXIVONLY{in the sequence }above.
\end{definition}
\begin{ARXIV}
Before discussing the functoriality and further properties of this construction, we will generalize it slightly to take care of supports.
\end{ARXIV}

\subsubsection{Cohomology with supports}\label{sec:cohomology-supports}
Suppose \(i\colon V \hookrightarrow X\) is a closed embedding. Then the pushforward \(i_*\colon \cat Ab(V)\to \cat Ab(X)\) has a right adjoint \(i^!\) \cite[Exposé~IV, Proposition~14.5]{SGA4.I}.  Both functors \(i_*\) and \(i^!\) and the global sections functor \(\Gamma\) on \(X\), are left exact, hence so is the composition \(\Gamma i_*i^!\).

\begin{definition}\label{def:cohomology-with-support}
  Given coefficient data \((X,\sheaf F)\) and a closed embedding \(i\colon V\hookrightarrow X\), as above, we define the \textbf{global sections functor with support on \(V\)} as the composition \(\Gamma_V := \Gamma i_*i^!\), and the \textbf{cohomology with support in \(V\)} via the right derived functors of \(\Gamma_V\) as
  \(
  \op{H}_V^s(X,\sheaf F) := \op{R}^s\Gamma_V\sheaf F
  \).
\end{definition}
\begin{ARXIV}
Let \(j\colon U\to X\) denote the embedding of the open complement of \(V\). For any sheaf \(\sheaf F\) on \(X\), the counit of the adjunction \(i_* \dashv i^!\) and the unit of the adjunction \(j^* \dashv j_*\) can be composed to obtain the following exact sequences \cite[Exposé~IV, Proposition~14.6]{SGA4.I}:
\begin{equation}\label{eq:i!-defining-sequence}
  \begin{aligned}
    0 \to i_*i^!\sheaf F\to &\sheaf F \to j_*j^*\sheaf F\\
    0 \to \Gamma_V\sheaf F\to \Gamma &\sheaf F \to \Gamma j_* j^*\sheaf F
  \end{aligned}
\end{equation}
\begin{remark}\label{rem:i!-defining-sequence-for-flasque}
The rightmost maps in the sequences~\eqref{eq:i!-defining-sequence} are \emph{not} surjective in general.  They are, however, surjective when \(\sheaf F\) is flasque.
\end{remark}
\begin{remark}\label{rem:cohomology-with-support-via-flasque-resolution}
  For a flasque sheaf \(\sheaf F\), all higher cohomology groups with support vanish, i.e.\ \(\op{H}^s_V(X,\sheaf F)\) for all \(s>0\) \cite[Corollary~II.9.3]{iversen}.  Hence, in general, \(\op{H}^s_V(X,\sheaf F)\) can be computed using a flasque resolution of \(\sheaf F\).  For an \emph{arbitrary} resolution \(\sheaf F\to \res_{\sheaf F}\), we still have a canonical map
  \begin{equation}\label{eq:canonical-map-to-Go-cohomology-with-support}
    \op{H}^*(\Gamma_V\res_{\sheaf F}) \to \op{H}^*_V(X,\sheaf F)
  \end{equation}
  as in \prettyref{lem:sheaf-cohomology-unique}\,\eqref{enum:scu:canonical-map-to-Go-cohomology}.
\end{remark}
It follows in particular that cohomology with support can be computed using the Godement resolutions of \prettyref{sec:construction-Godement}: \(\op{H}^s(\Gamma_V\god\sheaf F) \cong \op{H}^s_V(X,\sheaf F)\).
In degree zero, the complex \(\Gamma_V\god\sheaf F\) is more concretely given by
\begin{equation}\label{eq:Go0-concrete-with-support}
  \Gamma_VG \sheaf F \cong \prod_{x\in X^\delta\cap V} \sheaf F_x,
\end{equation}
as follows from \eqref{eq:Go0-concrete} and \eqref{eq:i!-defining-sequence}.
\begin{remark}
  The definition of cohomology with support given here agrees with \cite[Definition~II.9.1]{iversen}. Indeed, the exact sequences~\eqref{eq:i!-defining-sequence} show that the functor \(\Gamma_V \) considered here can be identified with the functor denoted \(\Gamma_V(X,-)\) in \cite[II.9]{iversen}.
  The definition also agrees with, or rather is a special case of, the definitions of cohomology with support used in \cite{godement} and \cite{bredon}.  Indeed, our functor \(\Gamma_V\) is the functor \(\Gamma_{\Phi_V}\) of \cite[I.6]{bredon} for \(\Phi_V\) the family of closed subsets of \(V\).
\end{remark}
\end{ARXIV}

\begin{lemma}\label{lem:pullback-with-support}
  Consider a continuous map \(f\colon X\to Y\).  Suppose \(i_V\colon V\hookrightarrow X\) and \(i_W\colon W\hookrightarrow Y\) are closed subsets such that \(V\supset f^{-1}W\).
  Then we have a canonical natural transformations of functors
  \(
    (i_W)_*i_W^! f_*  \longrightarrow f_*(i_V)_*i_V^!
  \),
  and, hence, a canonical natural transformation \(\Gamma_Wf_* \to \Gamma_V\).  If \(V=f^{-1}W\), then this natural transformation is an isomorphism.
  \AKTONLY{\qed}
\end{lemma}

\begin{ARXIV}
\begin{proof}
  Let \(j_V\) and \(j_W\) denote the inclusions of the open complements of \(V\) and \(W\), respectively.
  For any sheaf \(\sheaf F\), we have the following two exact sequences:
  \[
    \xymatrix{
      0  \ar[r] &  (i_W)_*i_W^! f_*\sheaf F \ar[r]  \ar@{..>}[d] & f_*\sheaf F \ar[r] \ar@{=}[d] & (j_W)_*(j_W)^*f_*\sheaf F  \ar@{..>}[d] \\
      0 \ar[r] & f_*(i_V)_*i_V^!\sheaf F \ar[r] & f_*\sheaf F \ar[r] &  f_*(j_V)_*(j_V)^*\sheaf F
    }
  \]
  The upper row is sequence~\eqref{eq:i!-defining-sequence} for \(f_*\sheaf F\) and \(W\subset Y\), while the lower row is the image of sequence~\eqref{eq:i!-defining-sequence} for \(\sheaf F\) and \(V\subset X\) under the left-exact functor \(f_*\).  It follows from our assumptions that, for any open subset \(O\subset Y\), \((f^{-1}O)\cap (X\setminus V)\) is contained in \(f^{-1}(O \cap (Y\setminus W))\). This allows us to define the vertical dotted morphism of sheaves on the right of the diagram, hence we obtain a vertical morphism indicated on the left, as desired.  If \(V=f^{-1}W\), then \((f^{-1}O)\cap (X\setminus V)=f^{-1}(O \cap (Y\setminus W))\), hence the right vertical morphism is an isomorphism.  It follows from the five lemma that so is the vertical morphism on the left in this case.
  \end{proof}
\end{ARXIV}

\begin{definition}[Pullback with support]
  \label{def:Go-pullback-with-support}
  Suppose we have a morphism of coefficient data \((f,\phi)\colon (X,\sheaf F) \to (Y,\sheaf G)\), and suppose we have closed subsets \(V\subset X\) and \(W\subset Y\) such that \(f^{-1}W\subset V\).  Then the \textbf{pullback with support along \((f,\phi)\)} is the morphism
  \[
    \op{H}^s_{V}(X,\sheaf F) \xleftarrow[(f,\phi)^*]{} \op{H}^s_{W}(Y,\sheaf G)
  \]
  induced by the following composition:
\begin{ARXIV}
    \[
      \Gamma_V\god \sheaf F
      \xleftarrow{\text{\prettyref{lem:pullback-with-support}}}
      \Gamma_W f_*\god \sheaf F
      \xleftarrow{\text{\prettyref{lem:Go-resolution-functorial}}}
      \Gamma_W\god f_*\sheaf F
      \xleftarrow{\quad\phi\quad}
      \Gamma_W \god \sheaf G
    \]
    Reading from right to left, the first two arrows are essentially the same morphisms as used in the definition of pullback without support (\prettyref{def:sheaf-pullback}).
\end{ARXIV}
\begin{AKT}
     \[
      \Gamma_V\god \sheaf F
      \xleftarrow{\text{\prettyref{lem:pullback-with-support}}}
      \Gamma_W f_*\god \sheaf F
      \leftarrow
      \Gamma_W \god \sheaf G
    \]
    The first arrow, on the right, is essentially the same morphisms as used in the definition of pullback without support (\prettyref{def:sheaf-pullback}).
\end{AKT}
\end{definition}
  
\begin{proposition}[Universality of the pullback]\label{prop:Go-pullback-universal}
  In the situation of \prettyref{def:Go-pullback-with-support}, suppose we have resolutions \(\sheaf F\to \res_{\sheaf F}\) and \(\sheaf G\to \res_{\sheaf G}\) on \(X\) and \(Y\), respectively, together with a morphism of complexes of sheaves \(\hat \phi\colon f_*\res_{\sheaf F}\leftarrow \res_{\sheaf G}\) such that
  \[
    \xymatrix{
      f_*\res_{\sheaf F}
      &
      \ar[l]^{\hat \phi}
      \res_{\sheaf G}
      \\
      \ar[u]
      f_*\sheaf F
      &
      \ar[u]
      \ar[l]^{\phi}
      \sheaf G
    }
  \]
  commutes. Then also the following square commutes:
  \[
    \xymatrix{
      \op{H}^*(\Gamma_V\res_{\sheaf F})
      \ar[d]
      &
      \ar[l]
      \op{H}^*(\Gamma_W\res_{\sheaf G})
      \ar[d]
      \\
      \op{H}^*_V(X,\sheaf F)
      &
      \ar[l]^{(f,\phi)^*}
      \op{H}^*_W(Y,\sheaf G)
    }
  \]
  Here, the vertical arrows are the canonical maps\ARXIVONLY{ \eqref{eq:canonical-map-to-Go-cohomology-with-support}},
  the upper horizontal arrow is induced by \(\hat \phi\) and the canonical transformation of \prettyref{lem:pullback-with-support},  and the lower horizontal arrow is the pullback of \prettyref{def:Go-pullback-with-support}.
  \AKTONLY{\qed}
\end{proposition}
\begin{ARXIV}
\begin{proof}
  Factor the diagram into two parts:
  \[
    \xymatrix{
      \op{H}^*(\Gamma_V\res_{\sheaf F})
      \ar[d]
      &
      \ar[l]
      \op{H}^*(\Gamma_Wf_*\res_{\sheaf F})
      \ar[d]
      &
      \ar[l]_-{\hat \phi}
      \op{H}^*(\Gamma_W\res_{\sheaf G})
      \ar[d]
      \\
      \op{H}^*_V(X,\sheaf F)
      &
      \ar[l]
      \op{H}^*_W(Y,f_*\sheaf F)
      &
      \ar[l]^-{(\id,\phi)^*}
      \op{H}^*_W(Y,\sheaf G)
    }
  \]
  For the commutativity of the square on the right, see \cite[Théorème~4.7.2]{godement}.  For the commutativity of the square on the left, recall from \cite[II.4.7]{godement} that the canonical vertical arrows are induced by  (mono)morphisms of complexes of sheaves \( \res_{\sheaf F} \rightarrow \op{Tot}(\god\res_{\sheaf F})\leftarrow \god\sheaf F\), where \(\op{Tot}\) denotes the associated total complex of a double complex. The upper and lower horizontal maps are induced by the natural transformation of \prettyref{lem:pullback-with-support} for \(\res_{\sheaf F}\) and \(\god\sheaf F\), respectively.  As this natural transformation can also be defined on \(\op{Tot}(\god\res_{\sheaf F})\), we find that this square commutes as well.
\end{proof}
\end{ARXIV}

\begin{proposition}[Functoriality of pullback]\label{prop:sheaf-pullback-functorial}
  Sheaf pullback is functorial with respect to morphisms of coefficient data.  That is, given two composable morphism of coefficient data \((f,\phi)\) and \((g,\gamma)\) and adequate choices of support,
\begin{AKT}
    \((f,\phi)^*\circ(g,\gamma)^* = ((g,\gamma)\circ(f,\phi))^*\). \qed
\end{AKT}
\begin{ARXIV}
  \[
    (f,\phi)^*\circ(g,\gamma)^* = ((g,\gamma)\circ(f,\phi))^*.
  \]
\end{ARXIV}
\end{proposition}
\begin{ARXIV}
\begin{proof}
  This follows from the functoriality of the Godement resolution (\prettyref{lem:Go-resolution-functorial}), the definition of composition in \eqref{eq:composition-of-morphisms-of-coefficient-data} and the universality of the pullback (\prettyref{prop:Go-pullback-universal}). (It is also claimed without proof in \cite[II.8]{bredon}. Note however that the statement there requires some amount of interpretation as Bredon does not even define the composition of ``cohomomorphisms'' (cf.\ \prettyref{rem:cohomomorphisms}).)
\end{proof}
\end{ARXIV}

\begin{proposition}[Localization sequence]\label{prop:localization-sequence}
Let $X$ be a topological space, $i\colon V\hookrightarrow X$ the inclusion of a closed subspace with open complement $j\colon U\hookrightarrow X$. With the definition of cohomology with support from \prettyref{def:cohomology-with-support}, we have a long exact sequence of cohomology groups as follows:
  \[
    \dots \to \op{H}^s_V(X,\sheaf F)\to \op{H}^s(X,\sheaf F) \to \op{H}^s(U,j^*\sheaf F) \to \op{H}^{s+1}_V(X,\sheaf F) \to \dots
  \]
  The maps preserving cohomological degree are the obvious pullback maps, as defined in \prettyref{def:Go-pullback-with-support}.
  \AKTONLY{\qed}
\end{proposition}
\begin{ARXIV}
\begin{proof}
  Take a flasque resolution of \(\sheaf F\), for example the Godement resolution \(\god F\).
  Then \(0\to \Gamma_V\god\sheaf F\to \Gamma\god \sheaf F\to \Gamma j^*\god \sheaf F\to 0\) is a short exact sequence of complexes of abelian groups.  Indeed, it is exact on the right by \prettyref{rem:i!-defining-sequence-for-flasque}.
  As \(j^*\) is simply restriction to an open subset, \(j^*\god\sheaf F\) is a flasque resolution of \(j^*\sheaf F\).  So the long exact cohomology sequence associated with this short exact sequence has the form displayed above.  For the identification of the map that restricts to \(U\), use \prettyref{prop:Go-pullback-universal}.
\end{proof}
\end{ARXIV}
\begin{proposition}
  \label{prop:Go-sheaf-exact-sequence}
  \label{prop:pullback-exact-seq}
  Given a short exact sequence of sheaves \(0\to \sheaf E\to\sheaf F\to \sheaf G\to 0\) on \(X\), we have an associated long exact sequence of cohomology groups:
  \[
    \dots \to \op{H}^s_V(X,\sheaf E) \to \op{H}^s_V(X,\sheaf F) \to \op{H}^s_V(X,\sheaf G)\xrightarrow{\partial}\op{H}^{s+1}_V(X,\sheaf E)\to \dots
  \]
  The maps preserving cohomological degree are the obvious change-of-coefficient morphisms.
\begin{ARXIV}

\end{ARXIV}
  Moreover, given a continuous map \(f\colon X\to Y\) such that \(0\to f_*\sheaf F\to f_*\sheaf G\to f_*\sheaf H\to 0\) is still exact on \(Y\), and given a closed subset \(W\subset Y\) such that \(f^{-1}W\subset V\), the pullback along \(f\) and the two long exact cohomology sequences fit together to
  \AKTONLY{a commutative ladder diagram. \qed}%
\begin{ARXIV}%
  a commutative ladder diagram as follows:
  \[
    \xymatrix@C=1em{
      \cdots \ar[r]
      &
      \op{H}^s_W(Y,f_\ast\mathcal{E})
      \ar[r]
      \ar[d]
      &
      \op{H}^s_W(Y,f_\ast\mathcal{F})
      \ar[r]
      \ar[d]
      &
      \op{H}^s_W(Y,f_\ast\mathcal{G})
      \ar[r]
      \ar[d]
      &
      \op{H}^{s+1}_W(Y,f_\ast\mathcal{E})
      \ar[r]
      \ar[d]
      &
      \cdots
      \\
      \cdots
      \ar[r]
      &
      \op{H}^s_V(X,\mathcal{E})
      \ar[r]
      &
      \op{H}^s_V(X,\mathcal{F})
      \ar[r]
      & \
      \op{H}^s_V(X,\mathcal{G})
      \ar[r]
      &
      \op{H}^{s+1}_V(X,\mathcal{E})
      \ar[r]
      &
      \cdots
      \\
    }
  \]
\end{ARXIV}
\end{proposition}
\begin{ARXIV}
\begin{proof}
  In the category of abelian groups, products of exact sequences are exact \cite[Example~A.4.5 and Ex.~A.4.5.1]{weibel}.  Thus, we see from \eqref{eq:Go0-concrete-with-support} that \(0\to\Gamma_V\god\sheaf F\to \Gamma_V\god\sheaf G\to \Gamma_V\god\sheaf H\to 0\) is exact in each degree.  The sequence displayed is the associated long exact cohomology sequence.
  The commutativity of the ladder diagram follows from the naturality of the transformation of \prettyref{lem:pullback-with-support}.
\end{proof}
\end{ARXIV}

\begin{ARXIV}
The next two lemmas will be needed to define the Godement cup product.
\begin{lemma}\label{lem:support-intersection}
  For closed subsets \(V,W\subset X\), we have \(\Gamma_V\sheaf F \cap \Gamma_W\sheaf F = \Gamma_{V\cap W}\sheaf F\) in \(\Gamma\sheaf F\).
\end{lemma}
\begin{lemma}\label{lem:support-tensor}
  For closed subsets \(V,W\subset X\) and abelian sheaves \(\sheaf M\) and \(\sheaf N\) on \(X\), the natural transformation \(\Gamma \sheaf M\otimes \Gamma\sheaf N\to\Gamma(\sheaf M\otimes\sheaf N)\) restricts to a natural transformation \(\Gamma_V\sheaf M\otimes\Gamma_W\sheaf N\to \Gamma_{V\cap W}(\sheaf M\otimes\sheaf N)\).

  More generally, given a sheaf of commutative rings \(\sheaf A\) on \(X\) and \(\sheaf A\)-modules \(\sheaf M\) and \(\sheaf N\), we have a natural transformation of \(\Gamma_{V\cap W}\sheaf A\)-modules \(\Gamma_V\sheaf M\otimes_{\Gamma_{V\cap W}\sheaf A} \Gamma_W\sheaf N\to \Gamma_{V\cap W}(\sheaf M\otimes\sheaf N)\).
\end{lemma}
\begin{proof}[Proofs]
  \prettyref{lem:support-intersection} follows directly from the definitions. For \prettyref{lem:support-tensor}, consider the following commutative diagram:
  \[
    \xymatrix@R=2ex{
      &
      \Gamma_V\sheaf M\otimes \Gamma_W\sheaf N
      \ar[d]
      \\
      &
      \Gamma_V\sheaf M \otimes \Gamma \sheaf N
      \ar[r]
      \ar@{..>}[dd]
      &
      \Gamma \sheaf M \otimes \Gamma \sheaf N
      \ar[r]
      \ar[dd]
      &
      \sheaf M(X\setminus V) \otimes \Gamma \sheaf N
      \ar[d]
      \\
      &
      &
      &
      \sheaf M(X\setminus V)\otimes \sheaf N(X\setminus V)
      \ar[d]
      \\
      0
      \ar[r]
      &
      \Gamma_V(\sheaf M \otimes \sheaf N)
      \ar[r]
      &
      \Gamma(\sheaf M\otimes \sheaf N)
      \ar[r]
      &
      (\sheaf M\otimes \sheaf N)(X\setminus V)
    }
  \]
  As the lowest row is exact and the composition of the two morphisms in the upper row is zero, we find that the composition \(\Gamma_V\sheaf M\otimes \Gamma_W\sheaf N \to \Gamma \sheaf M\otimes \Gamma\sheaf N \to \Gamma(\sheaf M\otimes\sheaf N)\) factors through \(\Gamma_V(\sheaf M\otimes \sheaf N)\). By symmetry, the same composition also  factors through \(\Gamma_W(\sheaf M\otimes \sheaf N)\). So by applying \prettyref{lem:support-intersection}, we obtain a factorization \(\Gamma_V\sheaf M\otimes \Gamma_W\sheaf N\to \Gamma_{V\cap W}(\sheaf M\otimes\sheaf N)\), as claimed.  The assertion for \(\sheaf A\)-modules follows similarly.
\end{proof}
\end{ARXIV}

Finally, we note that on \emph{Noetherian} spaces cohomology commutes with sequential colimits. Note that all schemes considered in this article are Noetherian, as they are assumed to be of finite type over a field.
\begin{lemma}\label{lem:cohomology-of-colimits}
  Let \(X\) be a Noetherian topological space, \(V\subset X\) a closed subset.  For any sequence of abelian sheaves \(\sheaf F_1\to \sheaf F_2\to\sheaf F_3 \to \dots\) on \(X\), the canonical maps \(\colim_j \op{H}^s_V(X,\sheaf F_j)\to \op{H}^s_V(X,\colim_j \sheaf F_j)\) are isomorphisms.
  \AKTONLY{\qed}
\end{lemma}
\begin{ARXIV}
\begin{proof}
  It suffices to show that the canonical morphisms \(\colim_j (\Gamma_V\sheaf F_j)\to \Gamma_V(\colim_j \sheaf F_j)\) is an isomorphism.  For \(V=X\), this is \cite[\stacktag{009F}\,(4)]{stacks}.  (Here, we use the assumption that \(X\) is Noetherian, or, equivalently, that every open subset of \(X\) is quasi-compact \cite[Ex.~II.2.13\,(a)]{hartshorne}.)  The case of general \(V\) follows by considering the sequences~\eqref{eq:i!-defining-sequence}.
\end{proof}
\end{ARXIV}

\subsubsection{Godement's cup and cross products}
\label{sec:Godement-product}
Of course, the above pullbacks could also be constructed using injective resolutions.  For the following cup product structure on sheaf cohomology, however, we know of no construction using injective resolutions.

\begin{theorem}[Cup product \cite{godement,bredon}]\label{thm:Go-cup-product}
  Let \(X\) be a topological space, or more generally a site with a conservative set of points.  Let \(V, V'\subset X\) be closed subsets.  Consider a sheaf of commutative rings \(\sheaf A\) on \(X\) and the ring of global sections over \(V\cap V'\), which we abbreviate to \(\Gamma'\sheaf A := \Gamma_{V\cap V'}\sheaf A\).  For \(\sheaf A\)-modules \(\sheaf M\) and \(\sheaf N\), there exists a unique cup product
  \begin{align*}
    \op{H}^p_V(X,\sheaf M)\otimes_{\Gamma'\sheaf A} \op{H}_{V'}^q(X,\sheaf N)
    &\rightarrow \op{H}^{p+q}_{V\cap V'}(X,\sheaf M \otimes_{\sheaf A} \sheaf N)\\
    a\otimes b  &\mapsto a\cup b
  \end{align*}
  That is, there exist unique such homomorphisms that are natural in \(\sheaf M\) and \(\sheaf N\) and satisfy the following properties:
  \begin{compactenum}[(i)]
  \item
    For \(p=q=0\), \(\Gamma_V\sheaf M\otimes_{\Gamma'\sheaf A} \Gamma_{V'}\sheaf N \to \Gamma_{V\cap V'}(\sheaf M\otimes_{\sheaf A}\sheaf N)\) is the canonical map\ARXIVONLY{ of \prettyref{lem:support-tensor}}.
  \item
    The cup product is compatible with the boundary maps \(\partial\) in the long exact cohomology sequences associated with short exact sequences of sheaves:  given an \(\sheaf A\)-module \(\sheaf N\) and a short exact sequence of left \(\sheaf A\)-modules \(0\to \sheaf M'\to \sheaf M \to \sheaf M'' \to  0\) that remains exact after tensoring over \(\sheaf A\) with \(\sheaf N\), we have \(\partial(a\cup b) = (\partial a)\cup b\).
  \item
    The cup product is graded commutative: \(a\cup b = (-1)^{p,q}\tau^*(b\cup a)\), where \(\tau\colon \sheaf N \otimes \sheaf M\to \sheaf M\otimes \sheaf N\) denotes the canonical isomorphism.
  \item
    The cup product is associative.
    \AKTONLY{\qed}
  \end{compactenum}
\end{theorem}
\begin{ARXIV}
\begin{proof}[Proof/references]
  In the following exposition, we suppress the module structures over \(\sheaf A\) and \(\Gamma'\sheaf A\) for simplicity.  Starting from the canonical simplicial resolution described at the beginning of \prettyref{sec:construction-Godement}, one construction of the cup product is as follows \cite[II.6.6\,(b)]{godement}:
  Given sheaves \(\sheaf M\) and \(\sheaf N\), we have a canonical natural transformation \(G\sheaf M\otimes G\sheaf N \to G(\sheaf M\otimes \sheaf N)\) inducing a natural transformation of simplicial resolutions \(\god^\star \sheaf M \times \god^\star \sheaf N \to \god^\star(\sheaf M \otimes \sheaf N)\).
  Passing to the associated chain complexes and using the Eilenberg--Zilber Theorem \cite[I.3.9.1]{godement}, we obtain a morphism of complexes of sheaves \(\god \sheaf M \otimes \god\sheaf N \to \god(\sheaf M \otimes \sheaf N)\).  Using \prettyref{lem:support-tensor}, we now obtain a morphism \(\Gamma_V(\god\sheaf M)\otimes \Gamma_{V'}(\god \sheaf N)\to \Gamma_{V\cap V'}(\god\sheaf M\otimes\god\sheaf N)\to \Gamma_{V\cap V'}\god(\sheaf M\otimes \sheaf N)\), from which the cup product can easily be constructed.  Properties (i)--(iv) are established in \cite[II.6.5]{godement}.  Godement gives details only for the closely related cross product, but the cup product can be treated similarly.

  Alternative constructions of the cup product are described in \cite[II.6.1 and II.6.6]{godement} and \cite[II.6 and II.7]{bredon}.  These constructions use the variation \(\god'(-)\) of Godement's resolution mentioned in \prettyref{rem:G1-versus-G2}.

  The uniqueness of a cup product satisfying the properties (i)--(iv) is established by Bredon in \cite[II.7]{bredon}: see Theorem~7.1, Corollary~7.2 and the following two remarks.  In particular, all constructions agree.  The arguments used in Godement and Bredon's proofs all work in the generality of sites with conservative sets of points.
\end{proof}
\end{ARXIV}
For the comparison theorems of this paper, we will need a few further properties of the cup product in sheaf cohomology.

\begin{proposition}[Naturality of the cup product]\label{prop:Go-product-natural}
  The cup product is natural with respect to continuous maps: Given a continuous map \(f\colon X\to Y\) and closed subsets \(V,V'\subset X\) and \(W,W'\subset Y\) such that \(f^{-1}W\subset V\) and \(f^{-1}W'\subset V'\), the following square commutes:
  \[\xymatrix{
      \op{H}^*_W(Y,f_*\sheaf M)\otimes_{\Gamma'\sheaf A} \op{H}^*_{W'}(Y,f_*\sheaf N) \ar[r]^-{\cup}\ar[d]^{f^*\otimes f^*}
      & \op{H}^*_{W\cap W'}(Y,f_*\sheaf M \otimes_{f_*\sheaf A}f_*\sheaf N)\ar[d]^{f^*}
      \\
      \op{H}^*_V(X,\sheaf M)\otimes_{\Gamma'\sheaf A} \op{H}^*_{V'}(X,\sheaf N) \ar[r]^-{\cup}
      & \op{H}^*_{V\cap V'}(X,\sheaf M \otimes_{\sheaf A}\sheaf N)
    }
  \]
  \AKTONLY{\qed}
\end{proposition}
\begin{ARXIV}
\begin{proof}
  The claim can be checked directly using the naturality of the transformation in \prettyref{lem:support-tensor} and the naturality of the other constructions involved.
\end{proof}
\end{ARXIV}

\begin{ARXIV}
\begin{theorem}[Universality of cup product]\label{thm:Go-cup-product-universal}
  In the situation of \prettyref{thm:Go-cup-product}, suppose we have resolutions \(\sheaf M \to \res_{\sheaf M}\), \(\sheaf N\to \res_{\sheaf N}\) and \(\sheaf M\otimes_{\sheaf A}\sheaf N\to \res_{\otimes}\) together with a homomorphism
  \[
    \res_{\sheaf M} \otimes \res_{\sheaf N} \xrightarrow{\quad\mu\quad} \res_{\otimes}
  \]
  that is compatible with the two augmentation maps from \(\sheaf M \otimes \sheaf N\).  Then the following square commutes:
  \[\xymatrix{
      \op{H}^*(\Gamma_V\res_{\sheaf M}) \otimes_{\Gamma'\sheaf A} \op{H}^*(\Gamma_W\res_{\sheaf N}) \ar[d] \ar[r]_-{\mu}
      &
      \op{H}^*(\Gamma_{V\cap W}\res_{\otimes}) \ar[d]\\
      \op{H}^*_V(X,\sheaf M) \otimes_{\Gamma'\sheaf A} \op{H}^*_W(Y,\sheaf N)  \ar[r]_-{\cup}
      &
      \op{H}^*_{V\cap W}(X,\sheaf M\otimes\sheaf N)
    }\]
  Here, the top horizontal map is induced by \(\mu\), while the two vertical maps are the canonical maps\ARXIVONLY{ \eqref{eq:canonical-map-to-Go-cohomology-with-support}}.
\end{theorem}
\begin{proof}
  This is \cite[Théorème~II.6.6.1]{godement}.
\end{proof}
\end{ARXIV}

We also have a cross product in sheaf cohomology.  Given two topological spaces \(X\) and \(Y\), let \(\pi_X\colon X\times Y\to X\) and \(\pi_Y\colon X\times Y\to Y\) denote the canonical projections.
For sheaves of abelian groups \(\sheaf M\) over \(X\) and \(\sheaf N\) over \(Y\), we write \(\sheaf M \boxtimes \sheaf N := \pi_X^*\sheaf M \otimes \pi_Y^*\sheaf N\) for the exterior tensor product sheaf over \(X\times Y\), and similarly for complexes of sheaves.
\begin{definition}[Cross product]
  \label{def:Go-cross-product}
For sheaves \(\sheaf M\) on \(X\) and \(\sheaf N\) on \(Y\), and for closed subsets \(V\subset X\) and \(W\subset Y\), the cross product
\begin{equation}\label{eq:Go-cross-product}
  \times\colon \op{H}^p_V(X, \sheaf M) \otimes \op{H}^q_W(Y, \sheaf N) \rightarrow \op{H}^{p+q}_{V\times W}(X\times Y, \sheaf M \boxtimes \sheaf N)
\end{equation}
is given by \(a\times b := \pi_X^*a \cup \pi_Y^*b\).
\end{definition}
Alternatively, this cross product can be obtained directly by a construction entirely analogous to the construction of the cup product, but we will not need this. If we take \(X = Y\) in \prettyref{def:Go-cross-F-product} and write \(\Delta\colon X\to X\times X\) for the diagonal, we clearly have \(a\cup b = \Delta^*(a\times b)\) in \(\op{H}^*(X,\sheaf M\otimes \sheaf N)\), by the naturality of the cup product, cf.\ \prettyref{prop:Go-product-natural}.

Now suppose \(X\) and \(Y\) are schemes over our base field \(F\).  Then we need to distinguish (the underlying topological space of) the scheme-theoretic product \(X\times_F Y\) from the topological product \(X\times Y\).  Let \(\pi_{F,X}\colon X\times_F Y \to X\) and \(\pi_{F,Y}\colon X\times_F Y \to Y\) denote the respective scheme-theoretic projections. As these morphisms are continuous, we have an induced continuous map
\begin{equation}\label{eq:scheme-to-top-product}
  u \colon X\times_FY \to X\times Y
\end{equation}
such that \(\pi_X u = \pi_{X,F}\) and \(\pi_Y u = \pi_{Y,F}\). Given closed subsets \(V\subset X\) and \(W\subset Y\), we have \(V\times_F W\subset u^{-1}(V\times W)\).  Now consider sheaves \(\sheaf M\) and \(\sheaf N\) on \(X\) and \(Y\), respectively.  The scheme-theoretic exterior tensor product of \(\sheaf M\) and \(\sheaf N\) is given by
\begin{align*}
  \sheaf M\boxtimes_F\sheaf N
  &:= \pi_{F,X}^*\sheaf M\otimes\pi_{F,Y}^*\sheaf N\\
  &\;\cong u^*(\pi_X^*\sheaf M\otimes\pi_Y^*\sheaf N)
\end{align*}

\begin{definition}[Cross product for schemes]\label{def:Go-cross-F-product}
  For \(F\)-schemes \(X\) and \(Y\), closed subsets \(V\subset X\) and \(W\subset Y\), and sheaves \(\sheaf M\) and \(\sheaf N\) on \(X\) and \(Y\), respectively, the \textbf{sheaf-theoretic cross product} \(\times_F\) is the composition of the usual cross product with the pullback along \(u\):
  \[
    \op{H}^p_V(X, \sheaf M) \otimes \op{H}^q_W(Y, \sheaf N) \xrightarrow{\times} \op{H}^{p+q}_{V\times W}(X\times Y, \sheaf M \boxtimes \sheaf N) \xrightarrow{u^*} \op{H}^{p+q}_{V\times_F W}(X\times_F Y, \sheaf M \boxtimes_F \sheaf N)
  \]
\end{definition}
We clearly have the following analogues of the formulas relating the cup and the usual cross product:
\begin{align}
  a\times_F b &= \pi_{X,F}^*a \cup \pi_{Y,F}^*b \label{eq:cross-F-from-cup}\\
   a \cup b &= \Delta_F^*(a\times_F b) \quad \text{ when \(X = Y\)}\label{eq:cup-from-cross-F}
\end{align}
Here, \(\Delta_F\colon X\to X\times_F X\) denotes the scheme-theoretic diagonal.

\begin{theorem}[Universality of the cross product]\label{thm:Go-cross-F-product-universal}
  Suppose we have resolutions
  \[
  \sheaf M\to \res_{\sheaf M}, \sheaf N\to \res_{\sheaf N} \textrm{ and } \sheaf M\boxtimes_F \sheaf N\to \res_{\boxtimes},
  \]together with a homomorphism
  \[
    \res_{\sheaf M} \boxtimes_F \res_{\sheaf N} \xrightarrow{\quad\mu\quad} \res_{\boxtimes}
  \]
  that is compatible with the two augmentation maps from \(\sheaf M \boxtimes_F \sheaf N\).  Then the following diagram commutes:
  \[\xymatrix{
      \op{H}^*(\Gamma_V\res_{\sheaf M}) \otimes \op{H}^*(\Gamma_W\res_{\sheaf N}) \ar[d] \ar[r]^-{\mu}
      &
      \op{H}^*(\Gamma_{V\times_FW}\res_{\boxtimes}) \ar[d]\\
      \op{H}^*_V(X,\sheaf M) \otimes \op{H}^*_W(Y,\sheaf N)  \ar[r]_-{\times}
      &
      \op{H}^*_{V\times_F W}(X\times_F Y,\sheaf M\boxtimes_F\sheaf N)
    }\]
  Here, the top horizontal map is induced by the pairing of complexes, while the two vertical maps are induced by the augmentations.
  \AKTONLY{\qed}
\end{theorem}
\begin{ARXIV}
\begin{proof}
  Using the description \(\times_F\) in terms of the cup product \eqref{eq:cross-F-from-cup}, we can factor the above square in two squares.  In the first square, the horizontal morphisms are given by the tensor product of the pullbacks along \(\pi_{X,F}\) and \(\pi_{Y,F}\).  More precisely, for the lower horizontal arrow we need to consider the morphisms of coefficient data \((\pi_{X,F},\phi_X)\) and \((\pi_{Y,F},\phi_Y)\), where \(\phi_X\colon (\pi_{X,F})_*\pi_{X,F}^*\sheaf M\to M\) and \(\phi_Y\) are the canonical morphisms. This first square commutes by \prettyref{prop:Go-pullback-universal}.  The second square encodes the cup product, and commutes by \prettyref{thm:Go-cup-product-universal}.
\end{proof}
Of course, an analogous theorem also holds for the usual cross product, but we will not need this here.
\end{ARXIV}
From now on, whenever dealing with schemes over \(F\), the product \(X\times Y\) and the cross-product will be understood to be taken over \(F\), even when the subscript \(F\) is suppressed.

\subsubsection{Twisted coefficients}\label{sec:twistedcoefficients}
So far, we have been considering arbitrary sheaves of abelian groups on arbitrary topological spaces (or even sites) in this section.  We now consider an operation on coefficient sheaves that requires slightly more structure:  ringed spaces \((X,\sheaf O_X)\), and coefficient sheaves of abelian groups \(\sheaf F\) that come equipped with a \(\Gm\)-action.  We always assume this action to be \emph{linear} in the sense that sections of \(\Gm\) act on \(\sheaf F\) as automorphisms of abelian groups.

\begin{definition}\label{def:general-twist}
  Consider a ringed space \((X,\sheaf O_X)\) equipped with a line bundle \(\sheaf L\), i.e.\ with a locally free \(\sheaf O_X\)-module of rank one. Given an abelian sheaf \(\sheaf F\) over \(X\) with a (linear) \(\Gm\)-action, the \textbf{twisted sheaf} \(\sheaf F(\sheaf L)\) is defined as \(\sheaf F\otimes_{\Z[\Gm]}\Z[\sheaf L^\times]\), i.e.\ as the sheafification of the presheaf $U\mapsto \sheaf F(U)\otimes_{\Z[\sheaf O_X^\times(U)]}\Z[\mathcal{L}(U)^\times]$.
  Given an equivariant morphism of sheaves with \(\Gm\)-action \(\phi\colon \sheaf F\to\sheaf G\), the \textbf{twisted morphism} \(\phi(\sheaf L)\colon \sheaf F(\sheaf L)\to\sheaf G(\sheaf L)\) is defined as \(\phi\otimes\id\).
\end{definition}

Here, $\Z[\mathbb{G}_{\op{m}}(U)]$ denotes the free abelian group generated by the invertible functions on $U$ and $\Z[\mathcal{L}(U)^\times]$ denotes the free abelian group generated by the nowhere vanishing sections of the line bundle $\mathcal{L}$ over $U$. (See \cite[Definition~2.6]{asok-fasel:euler}, where the notation $\mathcal{L}^\circ$ is used to denote the complement of the zero section of the line bundle $\mathcal{L}$; $\mathcal{L}^\circ(U) = \mathcal{L}(U)^\times$.)

\begin{ARXIV}
\begin{remark}\label{rem:twisted-sheaf-is-presheaf-tensor}
  Over any open subset \(U\subset X\) over which \(\sheaf L\) is trivial, the sheafification in \prettyref{def:general-twist} can be ignored.  Indeed, over such open \(U\), the \emph{presheaf} tensor product \(\sheaf F\otimes_{\Z[\sheaf O^\times]}\Z[\sheaf L^\times]\) is already a sheaf, isomorphic to \(\sheaf F_{|U}\) (though the isomorphism depends on the trivialization of $\sheaf L$).  As \(X\) can be covered by open subsets over which \(\sheaf L\) is trivial, we see in particular, that \(\sheaf F(\sheaf L)\) is always locally isomorphic to~\(\sheaf F\).
\end{remark}
\end{ARXIV}

\begin{lemma}\label{lem:general-twist-of-structure-morphism}
  Consider a morphism of ringed spaces \(f\colon (X,\sheaf O_X)\to (Y,\sheaf O_Y)\).
  For any sheaf \(\sheaf F_X\) over \(X\) with \(\Gm\)-action and any line bundle \(\sheaf L\) over \(Y\), we have a canonical isomorphism of sheaves of abelian groups \(f_*(\sheaf F_X(f^*\sheaf L))\cong (f_*\sheaf F_X)(\sheaf L)\), natural in \(\sheaf F_X\) and \(\sheaf L\).\footnote{
    In the context of line bundles, \(f^*\sheaf L\) denotes the pullback of \(\sheaf L\) as an \(\sheaf O_X\)-module, so that \(f^*\sheaf L\) is again a line bundle.
  }
  \AKTONLY{\qed}
\end{lemma}
\begin{ARXIV}
\begin{proof}
  Let us temporarily write \(f^p\) for the presheaf inverse image functor.  It can easily be checked explicitly that the counit \(\sheaf F_X\leftarrow f^pf_*\sheaf F_X\) is \(\Gm\)-equivariant, and it follows from this that the counit \(\eta\colon \sheaf F_X \leftarrow f^*f_*\sheaf F_X\) is likewise \(\Gm\)-equivariant.  Tensoring with the identity on \(f^*\sheaf L\) and precomposing with the canonical isomorphism \(f^*f_*\sheaf F_X\otimes_{\Z[\sheaf O_X^\times]}\Z[f^*\sheaf L^\times] \cong f^*(f_*\sheaf F_X\otimes_{\Z[\sheaf O_Y^\times]}\Z[\sheaf L^\times])\), we obtain a morphism \(\sheaf F_X(f^*\sheaf L)\leftarrow f^*((f_*\sheaf F_X)(\sheaf L))\).  The adjoint morphism is the morphism we are looking for.  Explicitly, it can be written as a composition
  \begin{equation}\label{eq:canonical-twist-composition}
    f_*(\sheaf F(f^*\sheaf L)) \xleftarrow{f_*(\eta\otimes\id)}
    f_*(f^*f_*\sheaf F(f^*\sheaf L))  \cong
    f_*f^*\left(f_*\sheaf F(\sheaf L)\right) \xleftarrow{\text{unit}}
    (f_*\sheaf F)(\sheaf L)
  \end{equation}
  Evaluate this composition on an open subset \(U\subset Y\) over which \(\sheaf L\) is trivial. Using \prettyref{rem:twisted-sheaf-is-presheaf-tensor} and the identification \((f^*f_*\sheaf F)(f^{-1}U)\cong \sheaf F(f^{-1}U)\), we find that this evaluation is an isomorphism, hence \eqref{eq:canonical-twist-composition} is an isomorphism of sheaves.
\end{proof}
Another way of formulating this lemma is to say that the canonical morphism of coefficient data \((X,\sheaf F_X) \to (Y,f_*\sheaf F_X)\) of \prettyref{eg:canonical-morphisms-of-coefficient-data} can be twisted by an arbitrary line bundle \(\sheaf L\) on \(Y\) to yield a new morphism of coefficient data \((X,\sheaf F_X(f^*\sheaf L))\to (Y,(f_*\sheaf F_X)(\sheaf L))\).
\end{ARXIV}

More generally, we can twist an arbitrary morphism of coefficient data.  To formalize this, let us a call a pair \(((X,\sheaf O_X),\sheaf F)\) consisting of a ringed space \((X,\sheaf O_X)\) and a sheaf of abelian groups \(\sheaf F\) on \(X\) with a linear \(\Gm\)-action a \textbf{ringed coefficient datum}.  Note that \(\sheaf F\) is not assumed to be an \(\sheaf O\)-module; when we write \(f^*\sheaf F\) in the following, we still mean the pullback of \(\sheaf F\) as sheaf of abelian groups.  If we forget the structure sheaf and the \(\Gm\)-action, we obtain an underlying coefficient datum \((X,\sheaf F)\).  A \textbf{morphism of ringed coefficient data}
\[
  (f,\phi)\colon ((X,\sheaf O_X),\sheaf F_X)\to ((Y,\sheaf O_Y),\sheaf F_Y)
\]
is a pair such that \((f,\phi)\colon (X,\sheaf F_X)\to (Y,\sheaf F_Y)\) is morphism of coefficient data in the previous sense, and such that \(f\colon (X,\sheaf O_X)\to (Y,\sheaf O_Y)\) is a morphism of ringed spaces whose structure morphism \(f_*\sheaf O_X\leftarrow \sheaf O_Y\) is compatible with the given actions of \(\sheaf O_X^\times\) on \(\sheaf F_X\) and \(\sheaf O_Y^\times\) on \(\sheaf F_Y\).

\begin{example}\label{eg:morphism-on-big-site-as-morcd}
  Suppose we are given an abelian sheaf \(\sheaf F\) with linear $\mathbb{G}_{\op{m}}$-action on the \emph{big} Zariski site \(\op{Sm}/F\). Then any morphism of smooth \(F\)-schemes \(f\colon X\to Y\) can be viewed as a morphism of ringed coefficient data \(((X,\sheaf O_X),\sheaf F_{|X}) \to ((Y,\sheaf O_Y),\sheaf F_{|Y})\).
\end{example}

\begin{definition}\label{def:twisted-morcd}
  Consider a morphism \((f,\phi)\colon ((X,\sheaf O_X),\sheaf F_X)\to ((Y,\sheaf O_Y),\sheaf F_Y)\) of ringed coefficient data. Given a line bundle \(\sheaf L\) over \(Y\), the \textbf{twisted morphism of ringed coefficient data}
  \[
    (f,\phi)(\sheaf L) := (f,\phi_{\sheaf L})\colon ((X,\sheaf O_X),\sheaf F_X(f^*\sheaf L)) \to ((Y,\sheaf O_Y),\sheaf F_Y(\sheaf L))
  \]
  has the same underlying morphism of ringed spaces \(f\), while the structure morphism \(\phi_{\sheaf L}\) is defined as the composition of the canonical isomorphism of \prettyref{lem:general-twist-of-structure-morphism} with the morphism \(\phi(\sheaf L)\) of \prettyref{def:general-twist}:
  \(
  f_*(\sheaf F_X(f^*\sheaf L))\cong (f_*\sheaf F_X)(\sheaf L) \xleftarrow[\phi(\sheaf L)]{} \sheaf F_Y(\sheaf L)
  \).
\end{definition}

\begin{lemma}\label{lem:twisted-composition}
  Twisting is compatible with composition:
  Given two composable morphisms of ringed coefficient data
  \[
    ((X,\sheaf O_X),\sheaf F_X)\xrightarrow{(f,\phi)} ((Y,\sheaf O_Y),\sheaf F_Y) \xrightarrow{(g,\gamma)}((Z,\sheaf O_Z),\sheaf F_Z)
  \]
  and a line bundle \(\sheaf L\) over \(Z\), \((g,\gamma)(\sheaf L) \circ (f,\phi)(g^*\sheaf L) = \big((g,\gamma)(f,\phi)\big)(\sheaf L)\), up to the canonical identification \(f^*g^*\sheaf L \cong (gf)^*\sheaf L\) on \(X\).
  \AKTONLY{\qed}
\end{lemma}
\begin{ARXIV}
\begin{proof}
  \newcommand{\can}{\mathrm{can}}
  Let \(\can_{f,\sheaf F,\sheaf L}\) denote the natural isomorphism \(f_*(\sheaf F(f^*\sheaf L))\leftarrow (f_*\sheaf F)(\sheaf L)\) from \prettyref{lem:general-twist-of-structure-morphism}.  The diagram whose commutativity we need to check can be divided into two squares: one square that commutes by the naturality of \(\can_{f,\sheaf F,\sheaf L}\) in \(\sheaf F\), and another square that looks as follows:
  \[\xymatrix{
      {g_*((f_*\sheaf F_X)(g^*\sheaf L))}
      \ar[d]^{g_*(\can_{f,\sheaf F_X,g^*\sheaf L})}_{\cong}
      &
      \ar[l]_-{\can_{g,f_*\sheaf F_X,\sheaf L}}^{\cong}
      {(g_*f_*\sheaf F_X)(\sheaf L)}
      \ar[d]^{\can_{gf,\sheaf F_X,\sheaf L}}_{\cong}
      \\
      {g_*f_*(\sheaf F_X(f^*g^*\sheaf L))}
      &\ar[l]^{\cong}
      {g_*f_*(\sheaf F_X((gf)^*\sheaf L))}
    }\]
  Here, the lower arrow is the arrow induced by the canonical identification \(f^*g^*\sheaf L \cong (gf)^*\sheaf L\).  To check the commutativity of such a square, it suffices to  consider sections over open sets \(U\subset Z\) over which \(\sheaf L\) is trivial. Using \prettyref{rem:twisted-sheaf-is-presheaf-tensor}, we thus reduce the commutativity of the above square to the commutativity of the following square:
  \[\xymatrix{
      \ar[d]_{g_*\eta_fg^*}
      {g_*g^*\sheaf L}
      &
      \ar[l]_{\eta_g}
      {\sheaf L}
      \ar[d]^{\eta_{gf}}
      \\
      {g_*f_*f^*g^*\sheaf L}
      &
      \ar[l]^{\cong}
      {g_*f_*(gf)^*\sheaf L}
    }\]
  Here, \(\eta_g\) is the counit of the adjunction \(g^*\dashv g_*\), \(\eta_f\) and \(\eta_{gf}\) are defined similarly, and the lower horizontal isomorphism is again induced by the canonical natural isomorphism \(f^*g^*\cong (gf)^*\).  This isomorphism rests on the uniqueness of adjoints, and the fact that both \(f^*g^*\) and \((gf)^*\) are left adjoints of \(g_*f_* = (gf)_*\).  The commutativity thus follows from abstract nonsense; see the explanation following Corollary~1 in \cite[\S\,IV.1]{maclane}.
\end{proof}
\end{ARXIV}

We will mostly use this lemma in the following form:
\begin{corollary}\label{cor:twist-key}
  Suppose we are given a commutative square of morphisms of ringed coefficient data:
  \[
    \xymatrix@R=4em@C=5em{
      ((X,\sheaf O_X),\sheaf F_X)
      \ar[r]^{(g,\gamma)}
      \ar[d]^{(f,\phi)}
      & ((X',\sheaf O_{X'}),\sheaf F_{X'})
      \ar[d]^{(f',\phi')}
      \\
      ((Y,\sheaf O_Y),\sheaf F_Y)
      \ar[r]^{(h,\eta)}
      &
      ((Y',\sheaf O_{Y'}),\sheaf F_{Y'})
    }
  \]
  Then for any line bundle \(\sheaf L\) on \(Y'\), the following square still commutes, where the isomorphism in the top left corner is induced by the canonical identification \(f^*h^*\sheaf L \cong g^*f'^*\sheaf L\):
  \[
    \xymatrix@R=5em{
      \cbox[t]{
        $((X,\sheaf O_X),\sheaf F_X(g^*f'^*\sheaf L))$\\
        $\downcong$\\
        $((X,\sheaf O_X),\sheaf F_X(f^*h^*\sheaf L))$
      }
      \ar[r]^-{(g,\gamma)}
      \ar[d]^{(f,\phi)}
      &
      ((X',\sheaf O_{X'}),\sheaf F_{X'}(f'^*\sheaf L))
      \ar[d]^{(f',\phi^*)}
      \\
      ((Y,\sheaf O_Y),\sheaf F_Y(h^*\sheaf L))
      \ar[r]^-{(h,\eta)}
      &
      ((Y',\sheaf O_{Y'}),\sheaf F_{Y'}(\sheaf L))
    }
  \]
  \AKTONLY{\qed}
\end{corollary}

\begin{example}
  \label{eg:twist-functorial}
  Let $F$ be a field, let $f\colon X\to Y$ be a morphism of smooth $F$-schemes, and let $\mathcal{L}$ be a line bundle over $Y$. Any equivariant morphism of abelian sheaves with $\mathbb{G}_{\op{m}}$-action $\phi\colon\mathcal{F}\to \mathcal{G}$ on the big Zariski site $\op{Sm}_F$ induces well-defined morphisms of abelian sheaves \(\phi_Y(\sheaf L)\colon \sheaf F|_Y(\sheaf L)\to \sheaf G|_Y(\sheaf L)\) and \(\phi_X(f^*\sheaf L)\colon \sheaf F|_X(f^*\sheaf L)\to \sheaf G|_X(f^*\sheaf L)\) such that the following diagram commutes:
  \[
    \xymatrix@C=7em{
      \mathcal{F}|_Y(\mathcal{L}) \ar[r]^{\phi_Y(\sheaf L)} \ar[d] & \mathcal{G}|_Y(\mathcal{L}) \ar[d] \\
      f_\ast (\mathcal{F}|_X(f^\ast\mathcal{L})) \ar[r]^{f_*\phi_X(f^*\sheaf L)} & f_\ast(\mathcal{G}|_X(f^\ast\mathcal{L}))
    }
  \]
  Here the vertical morphisms are the ones induced from the sheaf restriction morphisms as in \prettyref{eg:morphism-on-big-site-as-morcd}.
\begin{ARXIV}
  Indeed, this is a case of \prettyref{cor:twist-key} corresponding to the following commutative diagram of ringed coefficient data, in which \(\varphi\) and \(\gamma\) denote the structure maps for \(\sheaf F\) and \(\sheaf G\):
  \[
    \xymatrix{
      ((Y,\sheaf O_Y),\sheaf F_Y)             \ar@{<-}[r]^{(\id,\phi_Y)} & ((Y,\sheaf O_Y),\sheaf G_Y)\\
      ((X,\sheaf O_X),\sheaf F_X) \ar[u]^{(f,\varphi)}\ar@{<-}[r]^{(\id,\phi_X)} & ((X,\sheaf O_X),\sheaf G_X) \ar[u]^{(f,\gamma)}
    }
  \]
\end{ARXIV}
\end{example}

\subsection{Recollections on Witt groups and fundamental ideals}
Given a scheme with a line bundle $(X, \mathcal{L})$, let \(\op{VB}(X)\) denote the category of vector bundles over \(X\) and let $\#_\mathcal{L}\colon \op{VB}(X)^{\textnormal{op}} \rightarrow \op{VB}(X) $ denote the functor sending $ \mathcal{F}$ to $ \mathcal{F}^{\vee_\mathcal{L}} : = \mathcal{H}\textnormal{om}_{\mathcal{O}_X}(\mathcal{F}, \mathcal{L}) $.
\begin{definition}
  The Witt group $\op{W}(X, \mathcal{L})$ of the pair $(X, \mathcal{L})$ is the Witt group
  \(
  \op{W}(X, \mathcal{L}) : = \op{W}(\op{VB}(X), \#_\mathcal{L})
  \)
  of the exact category with duality $(\op{VB}(X), \#_\mathcal{L})$, equipped with the usual double-dual identification.
\end{definition}
To simplify notation, we often omit $\mathcal{L}$ when $\mathcal{L}=\sheaf{O}$.
Note that $\op{W}(X,\mathcal{L})$ is a $\op{W}(X)$-module by left multiplication/tensoring with symmetric forms.

There is a well-defined ring homomorphism
\[
  \op{rk}\colon \op{W}(X) \rightarrow \op{H}^0_{\et}(X, \Z/2\Z )
\]
which assigns to an isometry class of a symmetric bilinear form the rank of its underlying vector bundle modulo two. The kernel of the rank homomorphism is denoted $\op{I}(X)$ and is called the fundamental ideal. Let $\op{I}^j(X)$ denote the \(j^{\text{th}}\) power of this fundamental ideal, and let $\op{I}^j(X, \mathcal{L}): = \op{I}^j(X) \cdot \op{W}(X, \mathcal{L})$ denote the \(\op{W}(X)\)-submodule of \(\op{W}(X,\mathcal{L})\) generated by \(\op{I}^j(X)\).

\begin{definition}\label{def:W-and-I-sheaves}
  We write \(\W_{\mathcal L}\) for the Zariski sheafification of the presheaf \((U\subseteq X) \mapsto \op{W}(U, \mathcal{L})\) on the small Zariski site of~$X$, and
  \(\I^j_{\mathcal L}\)  for the Zariski sheafification of the presheaf $(U\subseteq X) \mapsto \op{I}^j(U, \mathcal{L})$.
\end{definition}
Again, \(\mathcal{L}=\sheaf{O}\) will usually be suppressed in the notation. The \(\I^k\)-cohomology of a scheme \(X\) is by definition the sheaf cohomology of \(X\) with coefficients in \(\I^k\),
or with coefficients in the twisted variants \(\I^k_{\mathcal L}\).

\begin{ARXIV}
\begin{remark}
  \label{rem:twotwists}
  As we will see in \prettyref{thm:allIcohomologyagree} below,
  the associated Zariski sheaf cohomology groups coincide with the groups considered by Fasel \cite{fasel:ij} at least when $X$ is smooth over a field $F$. Indeed, the Gersten conjecture for Witt groups is known for local rings containing a field \cite{bgpw}, and we may use \cite{gille} to pass to $\op{I}^j(X, \mathcal{L})$, see \cite[section 2.1]{fasel:ij}.
  This implies in particular that our definition of $\mathbf{I}^j_\mathcal{L}$ agrees with Fasel's definition as a sheafification of the kernel of
  the first map in a Gersten complex in \cite[section 2.1]{fasel:ij}:
  there is a map from our presheaf to
  the one of loc.\ cit., and it is an isomorphism
  on local rings, hence on the associated sheaves.
\end{remark}
\end{ARXIV}

We can also twist the sheaves \(\I^j\) using the construction of \prettyref{def:general-twist}.  Indeed, for any scheme \(U\), we have a morphism \(\Gm(U)\to \op{W}(U)\) sending an invertible function to the associated symmetric form of rank one.  This defines a morphism of sheaves of rings \(\Gm\to \W\) and hence equips each of the sheaves \(\I^j\) with a canonical linear \(\Gm\)-action. Thus, for any line bundle \(\sheaf L\) over \(U\), we have twisted sheaves \(\I^j(\sheaf L)\).  Luckily, twists in this sense agree with twists in the sense considered above:

\begin{proposition}
  \label{prop:identtwist}
  For any smooth variety $X$ over a field and any line bundle $\mathcal{L}$ over $X$ there is an isomorphism $\mathbf{I}^j(\mathcal{L}) \cong \mathbf{I}^j_\mathcal{L}$ of Zariski sheaves which is compatible with the linear \(\Gm\)-action on both sides.
  \AKTONLY{\qed}
\end{proposition}
\begin{ARXIV}
\begin{proof}
  \def\LL{\mathcal{L}}
  \def\VV{\mathcal{V}}
  An element of \(\LL(U)^\times\) defines an isomorphism \(\OO|_U\to\LL|_U\).  Thus, from a symmetric vector bundle \(\phi\colon \VV\to \Hom_\OO(\VV,\OO|_U)\) over \(U\) and an element \(u\in \LL(U)^\times\) we can construct a symmetric vector bundle for the \(\LL\)-twisted duality \(u_*\phi\colon \mathcal V\to \Hom_\OO(\VV,\OO|_U)\to \Hom_\OO(\VV,\LL|_U)\).  We thus obtain a morphism of presheaves
  \[
    \gamma\colon \op{I}^j(U)\otimes_{\Z[\Gm(U)]}\Z[\LL(U)^\times]\to \op{I}^j(U,\LL)
  \]
  that sends a pure tensor \((\VV,\phi)\otimes u\) to \((\VV,u_*\phi)\). Over any open \(U\) over which \(\LL\) is trivial, this morphism \(\gamma\) is an isomorphism. In particular, \(\gamma\) induces an isomorphism of the respective sheafifications.
\end{proof}
\end{ARXIV}
In the following, we sometimes denote the sheaf \(\W\) on \(X\) by \(\W_X\) for clarity.  The functoriality of the Witt sheaves furnishes us with structure morphisms
\(
\phi\colon f_*\W_X \leftarrow \W_Y
\)
for morphisms of schemes \(f\colon X\to Y\) over \(F\). In particular, any such morphism $f$ induces a morphism of ringed coefficient data
\[
  (f,\phi)\colon ((X,\sheaf O_X),\W_X)\to ((Y,\sheaf O_Y),\W_Y).
\]
More generally, we have associated morphisms of coefficient data for twists of the fundamental ideal sheaves as follows:

\begin{definition}\label{def:induced-morphism-of-cd}
  Consider a morphism of schemes \(f\colon X\to Y\) over \(F\). Given an exponent \(j\) and a line bundle \(\sheaf L\) over \(Y\), the \textbf{induced morphism of coefficient data}
  \[
    ((X,\sheaf O_X), \I^j(f^*\sheaf L)) \to ((Y,\sheaf O_Y), \I^j(\sheaf L))
  \]
  is obtained by restricting \((f,\phi)\) to the sheaves \(\I^j\) and by twisting via the construction of \prettyref{def:twisted-morcd}.
\end{definition}

\begin{remark}[big site]
  When we work on the big Zariski site $\op{Sm}_F$ of smooth schemes over \(F\), \(\W\) may occasionally denote the sheafification of the presheaf \((U\mapsto \op{W}(U))\) on this big site.  The restriction of this sheaf to a scheme \(X\) agrees with the sheaf \(\W_X\) considered above, i.e.\ with the sheafification of the restricted Witt presheaf, so no confusing seems likely.
\end{remark}

\begin{definition}[Mah\'{e}]\label{def:signature}
  Let $U$ be a scheme. The global signature is the ring homomorphism
  \begin{align*}
    \sign^U\colon \op{W}(U) &\rightarrow \op{C}(U_\r,\Z)=\op{H}^0(U_\r,\Z)  \\
    [\phi] &\mapsto \left(\sign([\phi])\colon (x, P) \mapsto \sign_P(i^*_x\phi)\right)
  \end{align*}
  where $\sign_P \colon \op{W}(\kappa(x)) \rightarrow \Z$ sends a form $[\varphi]$ to its signature with respect to the ordering $P$ on the residue field $\kappa(x)$ and
  $U_\r$ is the real spectrum of $U$ discussed in Subsection \ref{sec:prelim-ret}.

  If U is a scheme over $\R$, we may replace $sign_P$ by the classical signature for the closed point $x$, and we denote the resulting variant $\op{W}(U) \rightarrow \op{C}(U(\R),\Z)$ by $\sign^U$ as well.
\end{definition}

\begin{ARXIV}
There is an obvious relation between the rank homomorphism $\op{rk}$ and
the global signature mod 2,
see e.g.\ \cite[(8.3)]{jacobson}.
\end{ARXIV}

\subsection{Recollections on \texorpdfstring{$\I$}{I}-cohomology}\label{sec:I-recollections}
\begin{definition}\label{def:I-cohomology}
  The \textbf{\(\I^k\)-cohomology} of a scheme \(X\) equipped with a line bundle \(\Lb\) is the Zariski sheaf cohomology of \(X\)  with coefficients in the sheaf of abelian groups $\I^k(\Lb)$ defined in \prettyref{def:W-and-I-sheaves}\slash\prettyref{prop:identtwist}.  We will denote it as \(\op{H}^i_{\rm Zar}(X, \I^k(\Lb))\), or simply as \(\Icohom{i}{X}{\I^k}{\Lb}\).  In general statements that hold for all exponents~\(k\), we simply say \textbf{\(\I\)-cohomology} in place of \(\I^k\)-cohomology.
\end{definition}
In principle, \(\I^k\)-cohomology can be computed using an injective resolution of $\I^k$. However, apart from the existence of pullback maps along morphisms of schemes, most of the fundamental properties of \(\I\)-cohomology are difficult to establish using injective resolutions.  Instead, Fasel uses explicit Gersten resolutions to establish the fundamental properties in his foundational text \cite{fasel:memoir}.  Pushforwards for proper morphisms and pullbacks for flat morphisms are constructed in section 9.3 of loc.\ cit.  The latter extend to arbitrary maps between smooth schemes, see \prettyref{def:sheaf-pullback} and \prettyref{prop:GoGe-pullback-compatible} below for more details. We have a long exact localization sequence \cite[9.3.4 and 9.3.5]{fasel:memoir}, at least if the closed complement is smooth.  We also have homotopy invariance for vector bundles over regular schemes, see \cite [Th\'eor\`eme 11.2.9]{fasel:memoir}. Moreover, the \(\I\)-cohomology groups form a graded ring \cite{fasel:chowwittring}.

For the comparison results in later sections, in particular those regarding the ring structure, the sheaf theoretic constructions are more suitable than the constructions using Gersten complexes.  We therefore discuss and compare both.  We begin by briefly recalling the approach using Gersten complexes in \prettyref{sec:construction-Gersten}, noting in \prettyref{thm:allIcohomologyagree} that the \(\I^k\)-cohomology groups defined in this way indeed coincide with the sheaf cohomology groups with coefficients in \(\I^k\) for smooth \(X\).   In \prettyref{sec:I-pullbacks}, we verify, following \cite{asok-fasel:euler}, that pullbacks constructed using Gersten complexes coincide with the sheaf theoretic pullbacks of \prettyref{sec:construction-Godement}.  In \prettyref{sec:I-products}, we verify that, likewise, products constructed using Gersten complexes coincide with the products constructed via Godement resolutions in \prettyref{sec:Godement-product}.  Lastly, in \prettyref{sec:thomclass}, we recall the construction of Thom classes and pushforwards.  

Everything in this section applies verbatim to $\mathbf{J}$-cohomology, where $\mathbf{J}^k \cong \mathbf{K}_k^{\rm MW}$, with one notable exception concerning the commutativity of the ring structure (see \prettyref{rem:J-graded-commutativity}). In fact, all arguments and almost all references below apply to Chow--Witt groups or their bigraded extension.

\subsubsection{Gersten resolutions}
\label{sec:construction-Gersten}

\newcommand{\II}{\mathbb I} We recall the basic construction of Gersten resolutions from \cite{BW02}, see also \cite{gille} and \cite{fasel:memoir}.  Other constructions of the Gersten complex for $\mathbf{W}$ and $\mathbf{GW}$ are known to agree with this one: the construction in \cite{MField} agrees with the one of Schmid, and the latter one agrees with the ones of Fasel et.\ al.\ discussed here by \cite[section 7]{fasel:memoir}.

Let $X$ be a regular scheme of dimension $n$. Let $\Lb$ be a line bundle. We denote by $(D^{\rm b}(X), \vee_\Lb, \varpi_\Lb)$ the derived category with duality where $ \vee_\Lb\colon D^{\rm b}(X)^\textnormal{op} \rightarrow D^{\rm b}(X):  V \mapsto \Hom_{\OO_X}(V, \Lb) $ and $\varpi_\Lb$ is the double dual identification. Let $D^{(i)} \subset D^{\rm b}(X)$ be the triangulated subcategory of objects whose homology is supported in codimension $\geq i$, and let $D^i$ denote the triangulated category $ D^{(i)}/D^{(i+1)}$.
The augmented Gersten--Witt complex on \(X\) has the form
\[
  0 \rightarrow \op{W}(X, \Lb) \stackrel{d}\rightarrow \op{W}(D^0, \Lb) \stackrel{d}\rightarrow \op{W}^1(D^1, \Lb) \stackrel{d}\rightarrow \cdots \stackrel{d}\rightarrow \op{W}^n(D^n, \Lb) \rightarrow 0
\]
where $\op{W}^i$ denotes Balmer's derived Witt groups. The terms of the complex can be identified via isomorphisms
\begin{equation}\label{eq:identification-of-terms-in-Gersten-complex}
  \Sigma\colon \op{W}^i(D^i, \Lb) \cong \bigoplus_{x \in X^{(i)}} \op{W}(k(x), \omega_{x}^\Lb)
\end{equation}
obtained from localizations and dévissage where $\omega_{x}^\Lb : = \op{Ext}^i_{\OO_x}(k(x), \Lb_x)$ is a one-dimensional vector space over $k(x)$ for each point $x \in X^{(i)}$.

Denote by $\op{I}^{i,j}(X, \Lb) \subset \op{W}^i(D^i, \Lb)$ the subgroup corresponding to the sum of the $j$-th powers of fundamental ideals in the Witt groups \(\op{W}(k(x), \omega_{x}^\Lb)\) under the isomorphism \(\Sigma\), and by $\I^{i,j}_\Lb$ the sheafification of $U\mapsto \op{I}^{i,j}(U, \Lb|_U)$ on $X$. Note that $\I^{i,j}_\Lb$ is a flasque sheaf.
The Gersten--Witt complex is compatible with the fundamental ideal by \cite{gille}, or more precisely \cite[Th\'{e}or\`{e}me 9.2.4]{fasel:memoir} and \cite{fasel:ij}, i.e.,\ the differential $d$ in the Gersten--Witt complex satisfies $d(\op{I}^m(k(x), \omega_x^\Lb)) \subset \op{I}^{m-1}(k(y), \omega_y^\Lb)$ for any $m \in \Z$, $x \in X^{(i)}$ and $y \in X^{(i+1)}$.  We thus obtain, by restriction, a complex of abelian groups
\begin{align*}
  \gersten{\bullet}{X}{\I^j}{\Lb}&\colon  \op{I}^{0,j}(X, \Lb) \rightarrow \op{I}^{1,j-1}(X, \Lb) \rightarrow \cdots \rightarrow \op{I}^{n,j-n}(X, \Lb) \to 0,
                                   \intertext{%
                                   and a complex of flasque sheaves on \(X\)
                                   }
                                   \gersten{\bullet}{-}{\I^j}{\Lb}&\colon  \I^{0,j}_\Lb \rightarrow \I^{1,j-1}_\Lb \rightarrow \cdots \rightarrow \I^{n,j-n}_\Lb \to 0,
\end{align*}
which we also refer to as Gersten--Witt complexes.
\begin{theorem}
  \label{thm:allIcohomologyagree}
  Let $X$ be a smooth $F$-scheme of dimension $n$. Then, for any $i$ and any $j$, we have canonical isomorphisms
  \[
    \op{H}^i(\gersten{\bullet}{X}{\I^j}{\Lb}) \cong \op{H}^i_{\rm Zar}(X, \I^j(\Lb)) \cong   \op{H}^i_{\rm Nis}(X, \I^j(\Lb))
  \]
  We therefore simply write \(\Icohom{i}{X}{\I^j}{\Lb}\) for the above cohomology groups from now onwards.
\end{theorem}

\begin{proof}
  For the first isomorphism, the augmented Gersten complex provides a flasque resolution of the sheaf \(\I^j_\Lb\), i.e.,\ the complex
  \[
  0 \to \I^j_\Lb \rightarrow \I^{0,j}_\Lb \rightarrow \I^{1,j-1}_\Lb \rightarrow \cdots \rightarrow \I^{n,j-n}_\Lb \rightarrow 0
  \]
  is exact.  This is shown for \(j=0\), i.e.,\ for the sheaf \(\W\), in \cite[Lemma~4.2 and Theorem~6.1]{bgpw}; for the general case of fundamental ideals and twisted line bundles see \cite[Proof of Theorem~9.2.4]{fasel:memoir} and \cite[Cor.~7.2]{gille}.  So the first isomorphism is the canonical isomorphism coming from the uniqueness of sheaf cohomology\ARXIVONLY{ (\prettyref{lem:sheaf-cohomology-unique})}.
  The second isomorphism is \cite[Corollary~5.43]{MField}, using that the sheaves $\mathbf{I}^j$ satisfy the conditions of Morel's framework by \cite[Example 3.34 and Lemma~3.35]{MField}.
\end{proof}

\subsubsection{Supports}
For a closed subset \(i\colon Z\hookrightarrow X\) of a smooth $F$-scheme $X$, we have a definition of \(\I\)-cohomology with support in \(Z\) as the derived functors of \(\Gamma_Z=\Gamma(Z,-)\circ i^!\), cf.~\prettyref{def:cohomology-with-support}. This can also be expressed in terms of Gersten complexes, as follows.
\begin{proposition}\label{prop:GoGe-supports-compatible}
  Let \(j\colon U\hookrightarrow X\) denote the open complement of \(i\colon Z\hookrightarrow X\).  Consider the natural transformation \(\Gamma(X,-)\to \Gamma(U,-)\circ j^*\).  Following \cite[Définition~9.3.3]{fasel:memoir}, define the Gersten complex with support on \(Z\) as the kernel of this natural transformation applied to the Gersten complex on \(X\):
  \[
    \gersten[Z]{\bullet}{X}{\I^j}{\Lb} :=
    \ker\left(
      \gersten{\bullet}{X}{\I^j}{\Lb} \to \gersten{\bullet}{U}{\I^j}{\Lb}
    \right)
  \]
  There is a natural isomorphism \(\Icohom[Z]{s}{X}{\I^j}{\Lb} \cong \op{H}^s(\gersten[Z]{\bullet}{X}{\I^j}{\Lb})\).
\end{proposition}
\begin{proof}
  We need to derive the functor \(\Gamma_Z=\Gamma(Z,-)\circ i^!\).
  \ARXIVONLY{As noted in \prettyref{rem:cohomology-with-support-via-flasque-resolution}, we }%
  \AKTONLY{We }%
  can compute \((\op{R}^s\Gamma_Z)(\I^j)\) by evaluating \(\Gamma_Z\) on a flasque resolution.  Choosing this resolution to be the flasque resolution \(\gersten{\bullet}{-}{\I^j}{\Lb}\) of \(\I^j(\Lb)\),
  \ARXIVONLY{and using sequence \eqref{eq:i!-defining-sequence}, }%
  we can identify \(\Gamma_Z\) with the functor \(\ker \left(\Gamma(X,-)\to \Gamma(U,-)\right)\).
\end{proof}
In the above situation, we have a short exact sequence of Gersten complexes \(0\to \gersten[Z]{\bullet}{X}{\I^j}{\Lb} \to \gersten{\bullet}{X}{\I^j}{\Lb} \to \gersten{\bullet}{U}{\I^j}{\Lb}\to 0\). The associated long exact cohomology sequence is the localization sequence of \cite[Théorème~9.3.4]{fasel:memoir}, and coincides via the isomorphism above with the localization sequence in \prettyref{prop:localization-sequence}.
\begin{remark}\label{rem:boundary-map-on-complexes}
The boundary map in this localization sequence is induced by a morphism of complexes \(\partial\colon \gersten{\bullet}{U}{\I^j}{\Lb} \to \gersten[Z]{\bullet+1}{X}{\I^j}{\Lb}\), at least up to a sign.  To see this, one notes that, in each degree, \(\gersten{i}{U}{\I^j}{\Lb}\) is a direct summand of \(\gersten{i}{X}{\I^j}{\Lb}\).  The differential on \(\gersten{i}{X}{\I^j}{\Lb}\) thus has a component that maps the direct summand \(\gersten{i}{U}{\I^j}{\Lb}\) to the direct summand \(\gersten[Z]{i+1}{X}{\I^j}{\Lb}\). See \cite[Lemma~2.10]{asok-fasel:euler} for more details.
\end{remark}

\begin{theorem}[Dévissage]\label{thm:I-devissage}
  Suppose \(i\colon Z\hookrightarrow X\) is a closed embedding of codimension~\(r\) of regular schemes. Let \(\omega_Z\) denote the determinant line bundle of the normal bundle of this embedding, and let \(\Lb\) be an arbitrary line bundle over \(X\).  We have an isomorphism of complexes \(\gersten{\bullet-r}{Z}{\I^{j-r}}{i^*\Lb\otimes\omega_Z} \cong \gersten[Z]{\bullet}{X}{\I^{j}}{\Lb}\)
  inducing isomorphisms in cohomology
  \[
    \Icohom{s-r}{Z}{\I^{j-r}}{i^*\Lb\otimes\omega_Z} \cong \Icohom[Z]{s}{X}{\I^{j}}{\Lb}.
  \]
\end{theorem}
\begin{proof}
  This is stated in \cite[Remarque 9.3.5]{fasel:memoir}. We add some details for the reader's convenience. As $Z$ is regular and thus lci, it follows that the index sets $Z^{(s-r)}$ and $Z \cap X^{(s)}$ are in bijection. Using the notation as in 9.2.1/9.2.2 of loc.\ cit., the direct sums in degree $s$ of the relevant Gersten complexes have summands $\op{I}^{j-s}(\OO_{Z,z})$ and $\op{I}^{j-s}(\OO_{X,x})$. Under the above bijection, for corresponding points $z$ and $x$ there is a chain of isomorphisms
  \[
  \op{W}^\fl(\OO_{Z,z})\cong \op{W}(\kappa(z),\omega_{\kappa(z)/F})\cong \op{W}(\kappa(x),\omega_{\kappa(x)/F})\cong \op{W}^\fl(\OO_{X,x})
  \]
  given by applying \cite[Corollaire 6.2.8]{fasel:memoir}, and this isomorphism restricts to an isomorphism of the respective powers of fundamental ideals. To see that this is indeed an isomorphism of complexes, one has to check that these degreewise isomorphisms respect the differentials. For Witt groups, this follows using the two commutative squares in the diagram \cite[page 339]{Gil07b} in which the horizontal maps compose to the boundary of the Gersten complex. For the fact that the differentials respect $\I$-powers (up to a shift of one) see \cite[Corollary 7.3]{gille}.  The proofs in \cite{gille} and in \cite{fasel:memoir} rely on the identification of each term of the Gersten--Witt complex as a direct sum of Witt groups indexed on points of codimension $p$, see \eqref{eq:identification-of-terms-in-Gersten-complex} above and \cite[diagram (14) in Section~5.8]{gille}.  We observe that the vertical maps ${\rm Tr}$ in the diagram of \cite{Gil07b}  yield the (zig-zag) pushforward maps we described above. This can be seen combining the commutative diagram in the proof of \cite[Lemma~3.4]{Gil07b} with the commutative diagram in \cite[Proposition 3.6]{gille}. This also shows that for $f=i$ as above, the map $f_*$ of \cite[5.3.6, 6.3.10]{fasel:memoir} is precisely the one we study above composed with forgetting the support.
\end{proof}

\subsubsection{Pullbacks}
\label{sec:I-pullbacks}
\begin{definition}\label{def:I-pullback}
Given a morphism of schemes \(f\colon X\to Y\), closed subsets \(V\subset X\) and \(W\subset Y\) such that \(f^{-1}W\subset V\) and a line bundle \(\mathcal L\) over \(Y\), the \textbf{pullback in \(\I\)-cohomology}
\[
  \op{H}^*_{V}(X,\I^k(f^*\mathcal L)) \xleftarrow{(f,\phi)^*} \op{H}^*_{W}(Y,\I^k(\mathcal L)),
\]
is defined as in \prettyref{def:Go-pullback-with-support} from the morphism of coefficient data induced by \(f\) according to \prettyref{def:induced-morphism-of-cd}.
We will denote this pullback simply by~\(f^*\).
\end{definition}

The pullback along a morphism of smooth schemes \(f\colon X\to Y\) can also be constructed in terms of Gersten complexes.  To this end, one factors \(f\) as $X \to X \times Y \to Y$, where the first morphism is the graph of $f$ and the second is the projection onto $Y$. The graph of \(f\) is a regular embedding and the projection to \(Y\) is flat, so it suffices to treat these two classes of morphisms separately.

For flat $f$, it is possible to define a pullback directly on the level of Gersten complexes \cite[Section~3]{fasel:memoir}:
\[
\gersten{\bullet}{X}{\I^k}{f^*\Lb} \xleftarrow{f^*} \gersten{\bullet}{Y}{\I^k}{\Lb})
\]
To define a pullback with support, note that this construction is functorial and, in particular, compatible with restriction to open subsets.  Thus we obtain a morphism of complexes of sheaves,
\(
  f_*\gersten{\bullet}{-}{\I^k_X}{f^*\Lb} \leftarrow \gersten{\bullet}{-}{\I^k_Y}{\Lb})
  \).
For any closed subset \(W\subset Y\), we have an induced morphism of complexes of abelian groups $\gersten[f^{-1}W]{\bullet}{X}{\I^k}{\Lb} \leftarrow \gersten[W]{\bullet}{Y}{\I^k}{\Lb}$.  Passing to cohomology, we obtain a pullback with support as in \prettyref{def:I-pullback}.

For a regular embedding of smooth schemes \(i\colon Z\to X\), and a line bundle \(\Lb\) over \(X\), a pullback via Gersten complexes is constructed in \cite[Section~5]{fasel:chowwittring}. This is more delicate: it does not seem possible to define \(i^*\) directly on the level of Gersten complexes.  However, it is possible to define on the level of Gersten complexes a morphism \(\Theta\) that induces the pullback \((iq)^*\) along the composition \(iq\colon \Nb_Z X \to Z \to X\), where \(q\colon \Nb_Z X \to Z\) is the bundle projection of the normal bundle of the embedding.  As \(q\) is flat, we can then consider the following zig-zag of morphisms of complexes:
\[
  \xymatrix{
    \gersten{\bullet}{\Nb_Z X}{\I^k}{q^*i^*\Lb} 
    &
    \gersten{\bullet}{Z}{\I^k}{i^*\Lb} \ar[l]^-{q^*}
    &
    \gersten{\bullet}{X}{\I^k}{\Lb} \ar@/_2em/[ll]_{\Theta}
  }
\]
Once we pass to cohomology, \(q^*\) becomes an isomorphism by homotopy invariance \cite[Th\'eor\`eme 11.2.9]{fasel:memoir} and can be inverted, so we can define the pullback along \(i\) in \(\I\)-cohomology as \(i^* = (q^*)^{-1}\op{H}^\bullet(\Theta)\).

The construction of \(\Theta\) is based on the deformation to the normal bundle, for which we summarize our notation in \prettyref{fig:deformation-to-the-normal-bundle}.
\begin{figure}
\[
  \xymatrix{
    Z \ar[rr] \ar[dd]_s
    &&
    Z\times \A^1 \ar[dd]^{\tilde{i}}
    &&
    Z\times \Gm \ar[dd]^{i\times\id} \ar[ll]
    &&
    Z \ar[dd]^i \ar[ll] \ar@/_2em/[llll]
    \\\\
    \Nb_Z X \ar[rr]^{i_0} \ar[dd] \ar@{..>}[dr]^{q}
    &&
    D(Z,X) \ar[dd]^>(0.8){\pi} \ar@{..>}[dr]^{b}
    &&
    X\times \Gm \ar[ll]^{\iota'} \ar[dd] \ar@{..>}[dl]^{p_X}
    &&
    X \ar[ll] \ar@/_2em/[llll]_-<(0.2){i_1}  \ar[dd] 
    \\
    & Z \ar@{..>}[rr]^>(0.2){i} && X
    \\
    \ast \ar[rr]_{0}
    &&
    \A^1 
    &&
    \Gm \ar[ll]
    &&
    \ast \ar[ll]^{1} 
  }
\]
\caption{Deformation to the normal bundle}
\label{fig:deformation-to-the-normal-bundle}
\end{figure}
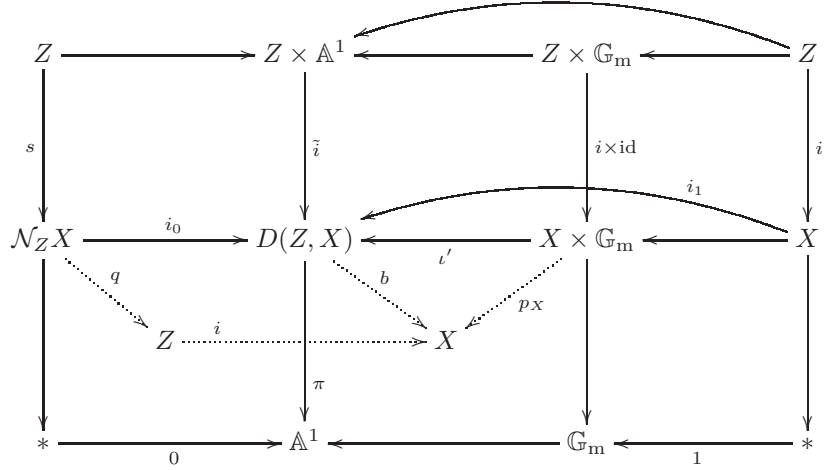
The deformation space $D(Z,X)$ is obtained by blowing up $X\times \mathbb{A}^1$ in $Z\times\{0\}$ and then removing the blowup of $X\times\{0\}$ along $Z\times\{0\}$.  By construction, it comes with a natural morphism to \(X\times \A^1\).  Composing with the projections, we obtain morphisms  \(\pi\colon D(Z,X) \to\A^1\) and \(b\colon D(Z,X)\to X\). The deformation space also comes with a natural embedding $\tilde{i}$ of \(Z\times \A^1\), and the composition of the central vertical arrows $\pi$ and $\tilde{i}$ is the projection onto $\mathbb{A}^1$. The fiber of \(\tilde i\) over $0$ is the zero section \(s\) of the normal bundle \(\Nb_Z X\), and the fibre over all $t\neq 0$ is the given embedding $i\colon Z\to X$.  In other words, the first, third and fourth vertical columns are obtained from the second column by base change along the inclusions of \(\{0\}\), \(\Gm\) or \(\{1\}\) into \(\A^1\), respectively.  In particular, the restriction of the normal bundle $\Nb_{Z\times\mathbb{A}^1}(D(Z,X))$ to any fiber is isomorphic to $\Nb_ZX$.   

The morphism of complexes \(\Theta\) is constructed as a composition \(\Theta = ({i_0}_*)^{-1}\partial m_t p_X^*\), the factors of which we now describe.  The first factor is simply the pullback along the flat projection \(p_X\colon X\times \Gm\to X\).  For the second factor, let $t$ denote the coordinate in $\Gm = \Spec(k[t,t^{-1}])$. Left multiplication with the form $\langle t, -1  \rangle \in \op{W}^0(\Gm)$ induces a morphism of complexes
\begin{equation}\label{eq:form-mt}
  \gersten{\bullet}{X\times \Gm}{\I^{k+1}}{p_X^*\Lb} \xleftarrow{\quad \langle t, -1  \rangle \cup - \quad}  \gersten{\bullet}{X\times \Gm}{\I^k}{p_X^*\Lb}.
\end{equation}
The required morphism $m_t$ of complexes up to a sign is given in degree $j$ by $(-1)^j \langle t, -1  \rangle \cup - $, see \cite[Lemma~5.1]{fasel:chowwittring}.
For the third factor of \(\Theta\), consider the boundary map in the localization sequence associated with \(\Nb_Z X \hookrightarrow D(Z,X) \hookleftarrow X\times \Gm\). As mentioned in \prettyref{rem:boundary-map-on-complexes}, this map can be realized as a morphism of complexes
\[
  \gersten[\Nb_Z X]{\bullet+1}{D(Z,X)}{\I^{k+1}}{b^*\Lb} \xleftarrow{\quad\partial\quad} \gersten{\bullet}{X\times \Gm}{\I^{k+1}}{p_X^*\Lb}  
\]
up to a sign. Hence, the composition $\partial m_t$ becomes a morphism of complexes.
Finally, since $\Nb_Z X$ is a Cartier divisor of $D(Z,X)$, we have the dévissage isomorphism described in \prettyref{thm:I-devissage}:
\[
  \gersten{\bullet}{\Nb_Z X}{\I^k}{q^*i^*\Lb}  \xrightarrow[\cong]{\quad{i_0}_*\quad} \gersten[\Nb_Z X]{\bullet+1}{D(X,Y)}{\I^{k+1}}{b^*\Lb}
\]

\begin{lemma}\label{lem:Fasels-Theta-is-morphism-of-sheaves}
Fasel's morphism \(\Theta\) defines a morphism of complexes of sheaves on \(X\):
\[
  i_*q_* \gersten{\bullet}{-}{\I^k_{\Nb_Z X}}{q^*i^*\Lb}  \xleftarrow{\quad\Theta\quad}  \gersten{\bullet}{-}{\I^k_X}{\Lb}
\]
\end{lemma}
\begin{proof}
  Consider an open subscheme \(U\subset X\).  Pull back the spaces and morphisms in the upper half of \prettyref{fig:deformation-to-the-normal-bundle} along the inclusion of \(U\) into the copy of \(X\) that is the target of \(p_X\) and \(b\).  This yields another such diagram in which \(X\) is replaced by \(U\) and \(Z\) by \(Z\cap U\);  the restriction of \(D(Z,X)\) to \(U\) can be identified with $D(Z\cap U,U)$ since blowing up commutes with flat base change \cite[\stacktag{0805}]{stacks}, and the restriction of $\Nb_Z X$ to \(U\cap Z\) can be identified with $\Nb_{Z\cap U} U$.  We need to check that each factor of \(\Theta\) is compatible with the restriction to \(U\).  
  For the first factor \(p_X^*\), this has already been discussed.  For the second factor \(m_t\), see \cite[Lemma~5.4]{fasel:chowwittring}. The boundary map \(\partial\) commutes with flat pullback, hence in particular with the two inclusions \(U\times \Gm \subset X\times \Gm\) and \(D(Z\cap U,U)\subset D(Z,X)\).  Finally, \({i_0}_*\) commutes with the inclusions \(D(Z\cap U,U)\subset D(Z,X)\) and \(\Nb_{Z\cap U}U\subset \Nb_Z X\) by the base change formula \cite[Corollaire 12.2.8]{fasel:ij}.
\end{proof}
The lemma implies that, for any closed subset \(V\subset X\), we have an induced morphism of complexes \(\Theta_V\colon \gersten[\Nb_Z X|_V]{\bullet}{\Nb_Z X}{\I^k}{\Lb_\Nb} \leftarrow \gersten[V]{\bullet}{X}{\I^k}{\Lb}\).  We can thus define the pullback along \(i\) with support in \(V\) as we did for support in \(X\), as \(i^* = (q^*)^{-1}\op{H}^\bullet(\Theta_V)\).

\begin{proposition}\label{prop:GoGe-pullback-compatible}
Let \(f\colon X\to Y\) be a morphism between smooth schemes over~\(F\).  Suppose we have closed subsets \(V\subset X\) and \(W\subset Y\) such that \(f^{-1}W \subset V\), and a line bundle \(\Lb\) on \(Y\). The induced pullback on \(\I\)-cohomology in terms of Gersten complexes described above agrees with the sheaf-theoretic pullback with support of \prettyref{def:I-pullback} under the identifications of \prettyref{prop:GoGe-supports-compatible}.
\end{proposition}
\begin{proof}
  As noted, a morphism between smooth schemes can be factored into a regular embedding followed by a  flat morphism. Compatibility of the pullbacks can be checked separately for each factor.
%

  For a flat morphism \(f\), it suffices by \prettyref{prop:Go-pullback-universal} to check that the degree zero part of the pullback morphism constructed on the level of Gersten complexes of sheaves is compatible with the usual pullback on \(\I^k\)-sheaves.  This is indeed the case, as in degree zero this pullback is the usual pullback on function fields.

  For a regular embedding \(i\colon Z\to X\), a proof of the desired compatibility in the context of Milnor--Witt sheaves is given in \cite[Section 2]{asok-fasel:euler}.  The argument for \(\I\)-cohomology that we need here is analogous.  Let us provide some details, concentrating for simplicity of notation on the case \(\I^k = \W\).  In the notation above, it suffices to check that \(\op{H}^\bullet(\Theta_V)\) is the usual pullback along the composition \(iq\).  By \prettyref{prop:Go-pullback-universal}, it suffices to verify that the degree zero part of the morphism of complexes of sheaves \(\Theta\) constructed in \prettyref{lem:Fasels-Theta-is-morphism-of-sheaves} is compatible with the usual pullback \(q^*i^*\).
  \newcommand{\pre}{\mathrm{pre}}
  To verify this, we construct another, auxiliary, morphism of presheaves \(\Theta^\pre\).  We then show firstly that \(\Theta^\pre\) is indeed compatible with the degree zero part of \(\Theta\), and secondly that  \(\Theta^\pre\) agrees with the usual pullback \(q^*i^*\) on Witt presheaves.  The compatibility of \(\Theta\) with the sheaf pullback \(q^*i^*\) then follows.
 
  The auxiliary morphism \(\Theta^\pre\colon i_*q_*\op{W}(-;q^*i^*\Lb) \leftarrow \op{W}(-;\Lb)\) of presheaves on \(X\) is defined over an open subset \(U\subset X\) as the composition along the top row of the following diagram:
  \begin{equation}\label{eq:Theta-pre}
    \begin{aligned}
      \tiny
      \xymatrix@C=1.5em{
        \op{W}(\Nb_{Z_U} U)
        \ar[d]
        \ar@{}[dr]|{(a)}
        &
        \op{W}^1_{\Nb_{Z_U} U}(D(Z_U,U))
        \ar[l]_-{{i_0}_*^{-1}}
        \ar@{..>}[d]
        \ar@{}[dr]|{(b)}
        &
        \op{W}(U\times\Gm)
        \ar[l]_-{\partial}
        \ar[d]
        &
        \op{W}(U\times\Gm)
        \ar[l]_-{m_t}
        \ar[dd]
        &
        \op{W}(U)
        \ar[l]_-{p_U^*}
        \ar[dd]
        \\
        \op{W}(\Nb_0)
        \ar[d]^{(\id)}
        \ar@{}[drr]|{(c)}
        &
        \op{W}^1_{\Nb_0}(D_0)
        \ar@{..>}[l]_-{{i_0}_*^{-1}}
        &
        \op{W}(U_0)
        \ar@{..>}[l]_-\partial
        \ar[d]^{(\id)}
        \\
        \op{W}(k(\Nb_Z X))
        &
        \op{W}(k(\Nb_Z X))
        \ar[l]_-{{i_0}_*^{-1}}
        &
        \op{W}(k(X\times \Gm))
        \ar[l]_-{\partial}
        &
        \op{W}(k(X\times \Gm))
        \ar[l]_-{m_t}
        &
        \op{W}(k(X))
        \ar[l]_-{p_X^*}
      }
    \end{aligned}
  \end{equation}
  We have suppressed all twists in this diagram, but it should be clear what the correct twists are by comparison with the definition of \(\Theta\). The notation \(Z_U\) is short for \(Z\cap U\), the morphism \(p_U\) is the obvious projection, $m_t$ again denotes multiplication with the form \(\langle 1 ,-t\rangle\), the boundary map \(\partial\) in the top row is the boundary map of Balmer \cite[Theorem~6.2]{Ba00} in the localization sequence of Witt groups associated with \(\Nb_{Z\cap U} U \hookrightarrow D(Z\cap U,U) \hookleftarrow U\times \Gm\), and \({i_0}_*\) in the top row is the dévissage isomorphism of \cite[Theorem~4.4]{Gil07b}.  The composition of \({i_0}_*\)  with the morphism that forgets support agrees with the pushforward considered in \cite{calmeshornbostel:pushforward} (compare \cite[Remark~4.5]{Gil07b} and \cite[\S\,7.2]{calmeshornbostel:pushforward}), so we can use the base change and projection formulas of \cite{calmeshornbostel:pushforward} in the following.
    
  The composition along the lower row of diagram~\eqref{eq:Theta-pre} is the degree zero part of \(\Theta(U)\), while the vertical morphisms\slash compositions are the natural morphisms from the Witt presheaves to the Gersten complexes.  We need to show that this diagram commutes.  
  This is clear for the two squares on the right. For the big square on the left, we introduce the auxiliary central row.
  We denote by \(D_0\) the local ring of \(D(Z,X)\) at the generic point of \(\Nb_Z X\), by \(\Nb_0\) the local ring of \(\Nb_Z X\) at its generic point, and by \(U_0\) the field of fractions of \(D_0\).  Clearly \(\Nb_0 = k(\Nb_Z X)\) and \(U_0 =  k(X\times\Gm)\).  On the other hand, as \(\Nb_Z X\) is a Cartier divisor, \(D_0\) is a discrete valuation ring with \(\Nb_0\) as residue field, and the spectrum of \(U_0\) is the open complement of the spectrum of \(\Nb_0\) in the spectrum of \(D_0\).
  Thus, square~(b) commutes by naturality of the localization sequence, while square~(a) commutes by the base change formula  \cite[Theorem~6.9]{calmeshornbostel:pushforward}.
  The composition \({i_0}_*^{-1}\partial\) along the central row can be identified with the residue map using \cite[Lemma 8.1]{BW02}, and hence agrees with the composition \({i_0}_*^{-1}\partial\) in the lower row that is used in the construction of \(\Theta\). 
  
  To prove that \(\Theta^\pre\) agrees with the usual pullback \(q^*i^*\) on Witt presheaves, we need to evaluate over each open subset \(U\subset X\).  For simplicity of notation, we spell this out only for \(U=X\).  We first check that
  \begin{equation}\label{eq:pushforward-multiplication-mt}
    {i_0}_*(1) = \partial(m_t(1))
  \end{equation} 
  when $\Lb = \OO_X$.  To see this, we can pass to coherent Witt groups using \cite[Lemma~3.2]{gille:support}. We leave it as an exercise to chase through Balmer's cone construction of $\partial$ and the construction of \({i_0}_*\) in \cite[Theorem~3.2]{Gil07b}; the main additional input is the short exact sequence
  \(
  0 \to \OO_{D(Z,X)} \xrightarrow[t]{} \OO_{D(Z,X)} \to  {i_0}_*\OO_{\Nb_Z X} \to 0
  \)
  on \(D(Z,X)\).

  Using \eqref{eq:pushforward-multiplication-mt}, we now find for arbitrary \(\alpha\in\op{W}^0(X,\Lb)\):
  \begin{alignat*}{7}
    {i_0}_* q^*i^*\alpha  & = {i_0}_* (1\cup i_0^*b^*\alpha) \\
    & = {i_0}_* 1 \cup b^* \alpha  &\quad& \text{(by the projection formula \cite[Theorem~6.5]{calmeshornbostel:pushforward})}\\
    & = \partial(m_t(1)) \cup b^* \alpha   && \text{(by equation~\eqref{eq:pushforward-multiplication-mt})} \\
    & = \partial(m_t(1) \cup  \iota'^* b^* \alpha) &&\text{(by the Leibnize rule \cite[Theorem~2.10]{GN:pairings})}\\
    & = \partial(m_t(p_X^* \alpha))\\
    & = {i_0}_*\Theta^\pre(\alpha)
  \end{alignat*}
  This shows that \(\Theta^\pre\) agrees with \(q^*i^*\), for arbitrary $\Lb$, as claimed.
\end{proof}

\subsubsection{Products}
\label{sec:I-products}
Taking \(X\) to be a scheme, \(\sheaf A := \W\) to be the Witt sheaf and \(\sheaf M\) and \(\sheaf N\) to be powers of \(\I\), possibly twisted by line bundles \(\mathcal L\) and \(\mathcal L'\) over \(X\), Godement's construction described in \prettyref{thm:Go-cup-product} furnishes us with a product
\begin{multline*}
  \op{H}^p(X,\I^k(\mathcal L))\otimes_{\Gamma\W} \op{H}^q(X,\I^l(\mathcal L')) \to \op{H}^{p+q}(X,\I^k(\mathcal L)\otimes_{\W}\I^l(\mathcal L')) \to\\\to \op{H}^{p+q}(X,\I^{k+l}(\mathcal L\otimes\mathcal L'))
\end{multline*}
In particular, for trivial twists we obtain:
\begin{corollary}
  \label{cor:godement-cup-product}
  Godement's cup product
  \[
    \op{H}^p(X,\I^k)\otimes \op{H}^q(X,\I^l) \to \op{H}^{p+q}(X,\I^{k+l})
  \]
  turns \(\bigoplus_{p,k} \op{H}^p(X,\I^k)\) into a bigraded ring.  It is graded commutative with respect to the \(p\)-grading and commutative with respect to the \(k\)-grading.
\end{corollary}

We can also construct a cross product for \(\I\)-cohomology in this way, as follows.  For schemes \(X\) and \(Y\) over \(F\), Godement's cross product (\prettyref{def:Go-cross-F-product}) gives us a map
\(
\op{H}^*(X,\I^k) \times \op{H}^*(Y,\I^l) \to \op{H}^*(X\times Y, \I^k \boxtimes \I^l).
\)
We suppress the subscript \(F\) here to lighten notation, but of course all products considered are scheme-theoretic and taken over \(F\).
On the other hand, on \(X\times Y\), we have a morphism of sheaves
\begin{equation*}
  \I^k_X \boxtimes \I^l_Y \to \I^k_{X\times Y}\otimes \I^l_{X\times Y} \xrightarrow{\cdot} \I^{k+l}_{X\times Y}.
\end{equation*}
So by changing coefficients along this morphism, we obtain a product
\begin{equation*}
  \op{H}^p(X,\I^k) \times \op{H}^q(Y,\I^l) \to \op{H}^{p+q}(X\times Y, \I^{k+l}).
\end{equation*}
More generally, given line bundles \(\sheaf L\) on \(X\) and \(\sheaf L'\) on \(Y\), this product takes the form
\begin{equation}\label{eq:ext-product-on-I-cohomology}
  \op{H}^p(X,\I^k(\mathcal L)) \times \op{H}^q(Y,\I^l(\mathcal L')) \to \op{H}^{p+q}(X\times Y, \I^{k+l}(\mathcal L\boxtimes \mathcal L')).
\end{equation}

We now compare these products to the products defined by Fasel via Gersten complexes.  
Fasel first defines a cross product, and then pulls back along the diagonal.

\begin{theorem}\label{thm:Ge-ext-product}
  Consider a product of smooth \(F\)-schemes \(X\times Y\). Let $\Lb_X$ and $\Lb_Y$ be line bundles on $X$ and $Y$ respectively. Define $\Lb_{XY}:= \Lb_X \boxtimes \Lb_Y$.
  For open subsets \(U\subset X\) and \(V\subset Y\), the exterior product on Witt groups induces a morphism of complexes of abelian groups
  \[
    \star\colon \gersten{\bullet}{U}{\I^i}{\Lb_X} \otimes \gersten{\bullet}{V}{\I^j}{\Lb_Y}\rightarrow \gersten{\bullet}{U\times V}{\I^{i+j}}{\Lb_{XY}}
  \]
  as well as a morphism of the corresponding complexes of flasque sheaves on \(X\times Y\). Then, the following square (in which the vertical arrows are the canonical augmentations) commutes:
  \[\xymatrix@R=12pt{
      \I^i_{\Lb_X} \boxtimes \I^j_{\Lb_Y} \ar[d] \ar[r]^-{\times} & \I^{i+j}_{\Lb_{XY}} \ar[d] \\
      \gersten{\bullet}{-}{\I^i}{\Lb_X} \boxtimes \gersten{\bullet}{-}{\I^j}{\Lb_Y} \ar[r]^-{\star} & \gersten{\bullet}{-}{\I^{i+j}}{\Lb_{XY}}
    }\]
\end{theorem}
\begin{proof}
  The product on the Gersten complexes is constructed in \cite[Theorem~4.16]{fasel:chowwittring} by replacing  $G$-cohomology with \(\I\)-cohomology. This can be done by following the argument in \cite[Lemma~4.1 to Prop.~4.7]{fasel:chowwittring}. By construction, the product is compatible with the augmentation \(\op{W}(-,\Lb_X) \to \gersten{\bullet}{-}{\I^0}{\Lb_X}\), so the same is true after restriction to \(\I^j\).
\end{proof}

\begin{definition}
  The product of Gersten complexes in \prettyref{thm:Ge-ext-product} induces a product on cohomology of the form
  \[
    \star\colon \op{H}^p(\gersten{\bullet}{X}{\I^i}{\Lb_X}) \otimes \op{H}^q(\gersten{\bullet}{Y}{\I^j}{\Lb_Y}) \rightarrow \op{H}^{p+q}(\gersten{\bullet}{X\times Y}{\I^{i+j}}{\Lb_{XY}}). \]
\end{definition}

\begin{proposition}\label{prop:GoGe-products-compatible}
  Fasel's cross product on \(\I\)-cohomology is compatible with Godement's cross product.
  Likewise, Fasel's cup product on \(\I\)-cohomology is compatible with Godement's cup product.
\end{proposition}
\begin{proof}
  The compatibility of the cross products follows from \prettyref{thm:Go-cross-F-product-universal}: the necessary assumption, namely the compatibility of the cross product on Gersten complexes with the cross product on the \(\I^j\)-sheaves, has already been verified in \prettyref{thm:Ge-ext-product}.  The compatibility of the cup products follows from the compatibility of the cross products and \prettyref{prop:GoGe-pullback-compatible}, since both cup products can be obtained from the corresponding cross product by pulling back along the diagonal, cf.\ equation~\eqref{eq:cup-from-cross-F} before \prettyref{thm:Go-cross-F-product-universal}.
\end{proof}

\begin{remark}\label{rem:J-graded-commutativity}
  As the Witt ring is commutative, the only signs appearing in the ring structure on \(\I\)-cohomology are those arising from point (iii) of \prettyref{thm:Go-cup-product}. For $\mathbf{J}$-cohomology we have additional signs coming from the graded commutativity of $\bigoplus_k \mathbf{J}^k \cong \bigoplus_k \mathbf{K}_k^{\rm MW}$: for the cup product
  \[
    \op{H}^p(X,\mathbf{J}^k)\otimes \op{H}^q(X,\mathbf{J}^l) \to \op{H}^{p+q}(X,\mathbf{J}^{k+l})
  \]
  one finds that \(a\cup b = (-1)^{pq}\epsilon^{kl} b\cup a\) with \(\epsilon = -\langle -1\rangle\).
\end{remark}

\subsubsection{Thom class and pushforward}\label{sec:thomclass}
Let $i\colon Z \hookrightarrow X$ be a regular embedding of regular schemes of codimension $r$, with normal bundle \(p\colon \Nb \to Z\). Denote the dévissage isomorphism of \prettyref{thm:I-devissage} as
\[
  i_*\colon
  \op{H}^s(Z, \I^j(i^*\Lb \otimes \det \Nb))
  \xrightarrow{\cong}
  \op{H}^{s+r}_Z(X, \I^{j+r}(\Lb)).
\]
The composition of \(i_*\) with extension of support agrees with the usual pushforward map in \(\I\)-cohomology, as pointed out at the end of the proof of \prettyref{thm:I-devissage}.  However, a more indirect definition of \(i_*\) will be more useful for the comparison with topology.  For this alternative definition, we first consider the case when \(i\) is the zero section of a vector bundle:
\begin{definition}
  \label{def:algebraic-thom-class}
  Let $p \colon V \rightarrow X$ be a rank $d$ vector bundle with zero section $s\colon X \hookrightarrow V$. The \textbf{Thom class} $\op{th}(V)\in \op{H}^d_X(V,\I^d(p^*\det V^\vee))$ is
  the image of the unit element \(1\in\op{H}^0(X,\W)\) under the isomorphism of \prettyref{thm:I-devissage} and the canonical trivialization of \(s^*p^*\det V^\vee \otimes \det V\).
\end{definition}
We now define the pushforward in \(\I\)-cohomology as follows:
\begin{definition}\label{def:pushforwardWittgroups}
  Let $i\colon Z \hookrightarrow X$ be a regular embedding of codimension $r$ between regular $F$-schemes, with normal bundle $p\colon \Nb\rightarrow Z$. The pushforward in \(\I\)-cohomology is defined as the composition
  \[
    i_*\colon
    \op{H}^s(Z, \I^j(i^*\Lb \otimes \det \Nb))
    \xrightarrow{\op{th}(\Nb) \cup p^*(-)}
    \op{H}^{s+r}_Z(\Nb,\I^{j+r}(p^* i^*\Lb))
    \xrightarrow[\cong]{d(Z,X)}
    \op{H}^{s+r}_Z(X, \I^{j+r}(\Lb)),
  \]
  where $d(Z,X)$ is the deformation isomorphism induced by deformation to the normal bundle.  The same name and notation are used for the composition of this map with the extension of support \(\op{H}^{s+r}_Z(X, \I^{j+r}(\Lb)) \rightarrow \op{H}^{s+r}(X, \I^{j+r}(\Lb))\).
\end{definition}
For the deformation to the normal bundle, see \prettyref{sec:I-pullbacks} above, in particular \prettyref{fig:deformation-to-the-normal-bundle}. 
The deformation isomorphism \(d(Z,X)\) is defined as the zig-zag $d(Z,X)=i_1^\ast\circ(i_0^\ast)^{-1}$, in view of the fact that both pullbacks with support $i_0^\ast$ and $i^\ast_1$ are isomorphisms, cf.\ \cite{fasel:excess} or \cite[Theorem~2.12]{asok-fasel:euler}.  In the context of \(\A^1\)-homotopy theory, the existence of \(d(X,Z)\) is a manifestation of homotopy purity.  See \cite[Section~3.5.2]{asok-fasel:euler} for a more detailed discussion.

In any case, the definition of pushforward using this deformation isomorphism and the Thom class is known to agree with the usual one:
\begin{lemma}\label{lem:pushforwards-agree}
  The dévissage isomorphism of \prettyref{thm:I-devissage} agrees with the composition in \prettyref{def:pushforwardWittgroups}.
\end{lemma}
\begin{proof}
  In this proof, we reserve the notation \(s_*\) and \(i_*\) for the d\'evissage isomorphisms of \prettyref{thm:I-devissage}.
  Let \(s\colon Z \to \Nb\) denote the zero section of the normal bundle \(p\colon \Nb\to Z\). The projection formula \cite[\S\,2.2]{fasel:ij} shows that \(s_* = \op{th}(V)\cup p^*(-)\).  It remains to show that \(d(Z,X)s_* = i_*\).
  This follows from the excess intersection formula of loc.\ cit., applied to the top two squares in the description of deformation to the normal bundle in \prettyref{fig:deformation-to-the-normal-bundle} above \prettyref{prop:comparedefconetubnbhd} below.  (Note that the normal bundles of the inclusions \(Z\to Z\times \A^1\) at \(0\) and \(1\) are trivial, so the Euler classes appearing in the excess intersection formula are also trivial.)
\end{proof}

\subsection{Recollections on real-\'etale cohomology}
\label{sec:prelim-ret}

Recall from \cite{scheiderer} that the real spectrum of a scheme $X$ is given as follows. For a commutative ring $A$, the \emph{real spectrum} $\op{sper} A$ is the topological space of pairs $(x,<_x)$ where $x\in \op{Spec} A$ is a point and $<_x$ is an ordering of the residue field $\kappa(x)$. A subbasis for the topology is given by domains of positivity $\op{D}(a)=\{(x,>_x)\in \op{sper}A\mid a>_x0\textrm{ in }\kappa(x)\}$, $a\in A$. These data can be glued giving rise to the \emph{real spectrum} $X_{\op{r}}$ of a scheme $X$. Forgetting the orderings of the residue fields yields a continuous map $\op{supp}\colon X_{\op{r}} \to X$ of topological spaces, called the \emph{support map}.

The real-\'etale topology for a scheme $X$ is then obtained by defining \emph{real-\'etale coverings} to be those families $\{f_i\colon U_i\to X\}_{i\in I}$ of morphisms where each $f_i$ is \'etale and the induced family $(f_i)_{\op{r}}\colon (U_i)_{\op{r}}\to X_{\op{r}}$ of morphisms on real spectra is jointly surjective. For a scheme $X$, the real-\'etale site $X_{\ret}$ associated to $X$ is the small \'etale site of $X$ equipped with the real-\'etale coverings.

By \cite[Theorem~1.3]{scheiderer}, the topos of sheaves on the real spectrum $X_{\op{r}}$ is equivalent to the topos of sheaves on the real-\'etale site $X_{\ret}$. The equivalence is induced by sending a sheaf $\sheaf F$ on $X_{\op{r}}$ to the real-\'etale sheaf
\[
(f\colon Y\to X)\in X_{\ret}\mapsto \op{H}^0(Y_{\op{r}},f_{\op{r}}^\ast \sheaf F).
\]
In particular, sheaf cohomology can be computed either using the real spectrum or the real-\'etale site.

For a scheme $X$, the support map $\op{supp}\colon X_{\op{r}}\to X$ induces a direct image functor
\[
\op{supp}_{\ast}\colon {\mathcal{A}}b(X_{\op{r}})\to {\mathcal{A}}b(X_{\op{Zar}})\colon \sheaf F\mapsto \left((U\subseteq X)\mapsto \sheaf F(\op{supp}^{-1}(U))\right).
\]
By \cite[Theorem~19.2]{scheiderer}, this functor is exact and faithful. For any scheme $X$ and any sheaf $\sheaf F$ on $X_{\op{r}}$, \cite[Lemma~4.6]{jacobson} provides  natural isomorphisms
\[
\op{H}^p(X_{\op{r}}, \sheaf F)\xleftarrow{\cong} \op{H}^p_{\op{Zar}}(X,\op{supp}_\ast \sheaf F), p\geq 0
\]
which arise as edge maps for the Grothendieck spectral sequence for the composition of $\op{supp}_\ast$ and the global sections functor. (In loc.\ cit.\ the above arrow is drawn in the wrong direction.) It can be shown that this arrow is given by $\op{supp}^*$ constructed as in \prettyref{def:sheaf-pullback} above.

\subsection{Recollections on singular cohomology}\label{sec:singcohomology}

We write \(\op{H}^*(-)\) for singular cohomology with integral coefficients and \(\op{h}^*(-)\) for singular cohomology with \(\Z/2\)-coefficients.  Similarly, \(\op{CH}^*(-)\) denotes Chow groups, and \(\op{Ch}^*(-)\) denotes Chow groups modulo two.

There are several well-known related definitions of cohomology for topological spaces, using singular cochains, sheaf cohomology, \v Cech cohomology, Eilenberg--Mac~Lane spectra, etc.  Even better, each of these comes with at least one and often several different constructions of fundamental structures such as pullbacks, pushforwards along closed embeddings, or (cup) products. In this article, we will always use the sheaf cohomology variant and usually restrict to locally constant sheaves. We will not discuss the comparison with all the other wonderful constructions in topology in detail\ARXIVONLY{, but see Remark \ref{folklore}}.
Standard textbook references for sheaf cohomology include \cite{bredon}, Godement \cite{godement} and Iversen \cite{iversen}.

\begin{ARXIV}
\begin{remark}\label{folklore}
  Assume that $X$ is a locally contractible paracompact space, e.g.\ a topological manifold, and $A$ an abelian group with $\underline{A}$ the associated constant sheaf of abelian groups. Then by \cite[Theorem~III.1.1]{bredon} or \cite[5.32]{warner}, there is an isomorphism $\op{H}^*(X,\underline{A}) \cong \op{H}^*(X,A)$ between sheaf cohomology and singular cohomology. As \cite{bredon} shows, this isomorphism is natural, that is compatible with pullbacks, and respects the ring structure. It holds more generally for {\em locally} constant sheaves $\mathcal{A}$, involving    singular cohomology with coefficients as described in section I.7 of loc.\ cit. The isomorphism also refines to cohomology with supports. This easily implies a compatibility of Thom classes in both settings, which in turn leads to a compatibility of pushforwards with respect to closed embeddings (when defining them as in  \cite[VIII.2.4]{iversen}, compare \prettyref{def:pushforwardWittgroups} and Section~\ref{sec:comp:push} below). Using Brown representability, see e.g.\ \cite[Theorem~9.12]{switzer}, these definitions (for constant $A$) are compatible with the cohomology theory represented by the Eilenberg--Mac~Lane spectrum $\mathbb{H}A$, including the compatibility with pullbacks. The ring structure on sheaf or singular cohomology theory may then be compared to the one relying on the monoid structure of $\mathbb{H}A$. More precisely, the ring structures for sheaf cohomology (see \prettyref{thm:Go-cup-product}) or  singular cohomology (using cochains or using $\mathbb{H}A$) all come from compatible exterior products in cohomology. These are then composed with the respective pullbacks along the diagonal, which are compatible by the discussion above, and with the product map $A \otimes A \to A$ or its refinement $\mathbb{H}A \wedge \mathbb{H}A \to \mathbb{H}A$. Further details concerning the $\mathbb{H}A$-case may be extracted from \cite[Chapter~III.9]{adams} and \cite[Chapter~13]{switzer}. A reader interested in \v Cech cohomology may consult \cite[Th\'eor\`eme II.5.10.1]{godement} or \cite[section 3.4]{bredon}.
\end{remark}
\end{ARXIV}

We briefly recall the definition and key properties of Thom classes, notably the Thom isomorphism, cf.\ \cite[sections VIII.2+3]{iversen} for sheaf cohomology and \cite[sections  8--10]{milnorstasheff} for singular cohomology. Other textbook references on Thom classes in singular cohomology include \cite{spanierbook}, \cite{switzer}, \cite{hatcher} and \cite{tomdieck}.

\begin{definition}
  \label{def:topological-thom-class}
  For any real rank $d$ topological vector bundle $V \to X$, a {\em Thom class} $\op{th}(V) \in \op{H}^d_{X}(V,\Z(p^*\det V^\vee))$ is a cohomology class $\op{th}(V)$ such that for any point $x \in X$ its restriction to the fiber $V_x$ is a generator of $\op{H}^d_{x}(V_x,\Z)$.
\end{definition}

Most references only consider the case of an oriented vector bundle. In this case, or when passing to $\Z/2\Z$-coefficients, the local coefficient system $\Z(p^*\det V^\vee)$ in \prettyref{def:topological-thom-class} is constant. For topological Thom classes with twisted coefficients for non-oriented vector bundles, we refer to \cite[Chapter~5, Exercise J.4]{spanierbook}.
In the present article, twisted coefficients in classical topology are described by twisting the locally constant sheaf $\Z$ with a line bundle $\Lb$ as in \prettyref{def:general-twist}.
As explained in section~\ref{localcoeffsection}, these $\Z(\Lb)$ may be translated to other descriptions -- e.g.\ back to local coefficient systems -- of twisted coefficients as used e.g.\ in \cite{spanierbook} and \cite{milnorstasheff} whenever necessary.

Thom classes and local generators -- if they exist -- are unique up to sign, and this ambiguity disappears if we choose compatible orientations, see e.g.\ \cite[Theorem~9.1]{milnorstasheff}. Finally, observe that some topology textbooks  -- but not \cite[Chapter~III.10] {adams} and \cite[Chapter~14]{switzer} -- work with disk and sphere bundles, rather than vector bundles and complements of zero sections. By excision and homotopy invariance, this doesn't change the relevant cohomology groups with support: we have
\[
\op{H}^t_{\{0\}}(\R^d,\Z) = \op{H}^t(\R^d,\R^d-\{0\},\Z) \cong \op{H}^t(D^d,S^{d-1},\Z),
\]
and the ``topological Thom space'' of the free real rank $d$ bundle over a point is $D^d/S^{d-1}=S^d$.

The existence of a Thom class for a rank $d$ vector bundle $V\to X$ leads to a {\em Thom isomorphism}
\begin{align*}
  \op{H}^{\ast}(X,\Z)&\to \op{H}^{\ast + d}_{X}(V,\Z(p^*\det V^\vee))\\
  u &\mapsto \op{th}(V)\cup p^*(u),
\end{align*}
as explained in any of the references above, see e.g.\ \cite[Theorem~5.7.10]{spanierbook}, \cite[Corollary 4D.9]{hatcher} or \cite[Theorem~17.9.1]{tomdieck}. In particular, the image of $1 \in  \op{H}^0(X,\Z)$ under the Thom isomorphism is $\op{th}(V)$.

Composing the Thom isomorphism $\op{H}^{*}(Z) \to \op{H}^{*+r}_Z(\Nb)$ for a closed embedding $i\colon Z \to X$ of codimension $r$ with an oriented normal bundle $\mathcal N$ with a ``deformation to the normal bundle'' or ``tubular neighborhood''-isomorphism and then forgetting the support $\op{H}^*_Z(X) \to \op{H}^*(X)$ is one of the standard definitions of the pushforward both for cohomology of abelian sheaves and for singular cohomology, which in this setting is sometimes called ``Gysin map''. Some further details are provided in the proofs of \prettyref{thm:pushforwards-compatible} and \prettyref{prop:comparedefconetubnbhd} below.

\subsubsection{Recollections on line bundles and local systems}\label{localcoeffsection}

Recall that a {\em local system} on a topological space $X$ is a functor from the fundamental groupoid of $X$ to the category of abelian groups, cf.~\cite[Exercises~1.F]{spanierbook}. For path-connected spaces, local systems with values in a fixed abelian group \(A\) are equivalent to group homomorphisms \(f\colon \pi_1(X,x)\to {\rm Aut}(A)\).
\begin{ARXIV}
  (To recover the local system from $f$, choose an inverse to the skeletal inclusion of $\pi_1(X,x)$ into the fundamental groupoid. Two such local systems are naturally isomorphic if and only if the corresponding maps \(\pi_1(X,x)\to {\rm Aut}(A)\) agree up to conjugation in \({\rm Aut}(A)\).)
\end{ARXIV}
For reasonable spaces like topological manifolds, local systems are moreover equivalent to locally constant sheaves of abelian groups, cf.~\cite[Exercise~6.F]{spanierbook} or \cite[Theorem~IV.9.7]{iversen}.

We will be mainly interested in locally constant sheaves of rank one, i.e., sheaves whose stalks at all points are isomorphic to $\Z$.  These correspond to rank one local systems (or, for path-connected $X$, to homomorphisms $f\colon \pi_1(X,x)\to{\rm Aut}(\Z)\cong\Z/2\Z$).

Any real line bundle \(\sheaf L\) over \(X\) determines a rank one locally constant sheaf \(\Z(\sheaf L)\), as follows.  Let \(\cont_X\) denote the sheaf of continuous real-valued functions on \(X\).  For any nowhere-vanishing function \(f\) on an open subset \(U\subset X\), the sign map defines a locally constant map \(\sign(f)\colon U\to \{\pm 1\}\).  We thus obtain a morphism of abelian sheaves \(\cont_X^\times \to \{\pm 1\}\).  As \(\{\pm 1\}\) acts on any sheaf of abelian groups \(\sheaf A\) in a canonical way, we obtain an induced action of \(\cont_X^\times\) on \(\sheaf A\).  So any sheaf of abelian groups \(\sheaf A\) on \(X\) can be viewed as a \(\Z[\cont_X^\times]\)-module. We can therefore apply the construction of \prettyref{def:general-twist} to \(\sheaf A\) and any line bundle \(\sheaf L\) over \(X\) to obtain a twisted sheaf \(\sheaf A(\sheaf L) := \sheaf A\otimes_{\Z[\cont(X)^\times]}\Z[\sheaf L^\times]\). In particular, taking \(\sheaf A\) to be the constant sheaf \(\Z\), we obtain a locally constant sheaf \(\Z(\sheaf L)\), as promised, and we write $\Z({\sheaf L})$ for the associated local system.

\begin{proposition}\label{prop:line-bundles-versus-local-systems}
  For any paracompact Hausdorff space \(X\), the assignment that associates with a line bundle \(\sheaf L\) the locally constant sheaf \(\Z(\sheaf L)\) induces a bijection between the set of isomorphism classes of real line bundles over \(X\) and the set of isomorphism classes of rank one locally constant sheaves on \(X\).  Both sets can moreover be identified with the (sheaf) cohomology group \(\op{H}^1(X,\Z/2\Z)\) in a natural way, such that the tensor product of line bundles\slash sheaves corresponds to the addition of cohomology classes.
  \AKTONLY{\qed}
\end{proposition}
\begin{ARXIV}
\begin{proof}
  The assumption that \(X\) is a paracompact Hausdorff space is made for simplicity.  It allows us to assume that all open covers of \(X\) are numerable \cite[Theorem~13.1.3]{tomdieck}. We then have natural bijections between the following sets:
  \begin{compactenum}[$(a)$]
  \item homotopy classes of continuous maps \(X\to B(\R^\times)\), where \(B(\R^\times)\) denotes the classifying space of the topological group \(\R^\times\)
  \item isomorphism classes of principal \(\R^\times\)-bundles over \(X\)
  \item equivalence classes of cocycles \((\{U_i\}_i,\{f_{ij}\}_{i,j})\), where \(\{U_i\}_i\) is an open cover of \(X\) and the \(f_{ij}\) are continuous maps \(U_i\cap U_j\to \R^\times\), up to refinement of the cover
  \item isomorphism classes of rank one locally free sheaves on the ringed space \((X,\cont(X))\)
  \item isomorphism classes of real line bundles over \(X\)
  \end{compactenum}
  Moreover, the tensor product in $(a)$, $(b)$, $(d)$, and $(e)$ corresponds to pointwise multiplication of transition functions in $(c)$.
  See \cite[Theorem~14.4.1]{tomdieck} for $(a)$ versus $(b)$ (here we use that all covers are numerable),
  \cite[Chapter~5, \S\S 2 and 3]{husemoller} for $(b)$ versus $(c)$,
  \cite[\stacktag{00AL} and \stacktag{00AM}]{stacks} for $(c)$ versus $(d)$,
  and \cite[Theorem~14.2.7]{tomdieck} for $(b)$ versus $(e)$.

  Similarly, we have natural bijections between the following sets:
  \begin{compactenum}[$(a')$]
  \item homotopy classes of maps \(X\to B(\{\pm 1\})\)
  \item isomorphism classes of principal \(\{\pm 1\}\)-bundles over \(X\)
  \item equivalence classes of cocycles \((\{U_i\}_i,\{f_{ij}\}_{i,j})\), where \(\{U_i\}_i\) is an open cover of \(X\) and the \(f_{ij}\) are continuous maps \(U_i\cap U_j\to \{\pm 1\}\), up to refinement of the cover
  \item isomorphism classes of rank one locally constant sheaves on \(X\)
  \end{compactenum}
  The inclusion \(i\colon \{\pm 1\}\hookrightarrow \R^\times\) is a group homomorphism and a homotopy equivalence, with homotopy inverse the sign map \(\sign\colon \R^\times \to \{\pm 1\}\).  In fact, the composition \(i\circ \sign\) is homotopic to the identity through maps of groups, so the induced maps \(Bi\) and \(B(\sign)\) are again homotopy equivalences.  This shows that the sets described by $(a)$ and $(a')$ are in natural bijection.

  It is not difficult to describe this bijection in terms of the sets $(b)$ and $(b')$, or $(c)$ and $(c')$: in either case, the map in one direction is some form of ``extension of scalars'', while the map in the other direction is induced by the sign map.  By tracing through the different bijections, we moreover see that the corresponding bijection from $(d)$ to $(d')$ can be described by the assignment \(\sheaf L\mapsto \Z(\sheaf L)\).

  Finally, we can also identify the set $(d)$ with the first \v{C}ech cohomology group \(\check{\op{H}}^1(X,\Z/2\Z)\), such that the pointwise product of cocycles considered above corresponds to the sum of cohomology classes, and \v{C}ech cohomology agrees with sheaf cohomology in degree one \cite[II.5.9]{godement}.

  In fact, under our assumptions on \(X\), \v{C}ech cohomology agrees with sheaf cohomology in all degrees \cite[II.5.10]{godement}.  Under mild additional hypothesis, e.g.\ if \(X\) has the homotopy type of a CW complex or is locally contractible, we could equivalently work with the first singular cohomology group with \(\Z/2\Z\)-coefficients; cf.\ Remark~\ref{folklore}. The bijection between the set of isomorphism classes of real line bundles and the cohomology group would then be given by the first Stiefel--Whitney class, which indeed takes the tensor product of two line bundles to the sum of the corresponding cohomology classes.
\end{proof}
\end{ARXIV}
Over a scheme $X/\R$, an algebraic line bundle \(\sheaf L\) has an underlying real line bundle \(\sheaf L_\R\) over \(X(\R)\). Let $\gamma\colon{\rm Pic}(X)/2\to \op{H}^1(X(\R),\Z/2\Z)$ denote the natural homomorphism that sends the isomorphism class of an algebraic line bundle $\mathcal{L}$ (up to squares) to the isomorphism class of the real topological line bundle $\sheaf L_\R$.

\begin{proposition}
  For any scheme $X/\R$, the map $\gamma\colon \op{Pic}(X)/2\to \op{H}^1(X(\R),\Z/2\Z)$ described above is a group homomorphism. The map is an isomorphism for smooth projective cellular varieties.
\end{proposition}

\begin{proof}
  The assignment that maps an algebraic line bundle $\sheaf L$ over $X/\R$ to the real line bundle $\sheaf L_\R$ over $X(\R)$ is functorial and hence gives rise to a well-defined morphism on sets of isomorphism classes. The tensor product of algebraic line bundles is induced from the multiplication in $\mathbb{G}_{\op{m}}$, and the tensor product of real line bundles is induced from the multiplication in $\R^\times$. The above map from algebraic to real line bundles takes the algebraic transition functions to the corresponding continuous functions on real points. This is obviously compatible with the multiplication, and consequently with the tensor product of line bundles. This proves the first statement.

  The second statement is a special case of \prettyref{prop:cellular-cycle-isos} below.
\end{proof}

\begin{remark}\label{example:picfornoncellular}
  The natural homomorphism $\gamma\colon \op{Pic}(X)/2\to \op{H}^1(X(\R),\Z/2\Z)$ is in general not surjective. For example, when $X=\mathbb{A}^2\setminus\{0\}$ we have $\op{Pic}(\mathbb{A}^2\setminus\{0\})/2\cong\op{Ch}^1(\mathbb{A}^2\setminus\{0\})=0$ by the localization sequence for Chow groups. On the other hand, $\mathbb{A}^2\setminus\{0\}(\R)\simeq S^1$ and $\op{H}^1(S^1,\Z/2\Z) \cong \Z/2\Z$. Thus there exists a non-trivial local system on the real points of $\mathbb{A}^2\setminus\{0\}$. In particular, the isomorphism for smooth projective cellular varieties does not extend to ``linear'' varieties.

  To see that in general $\gamma$ is not injective either, consider an elliptic curve $E/\R$ such that $E(\R)$ is isomorphic to $S^1\times{\Z/2\Z}$, e.g.\ the completion of the affine curve given by $Y^2=X^3-X$. Then there are rational points on $E$ which are not 2-divisible, hence $\op{Pic}^0(E)/2$ will be nontrivial. Any corresponding non-trivial line bundle will map to $0$ under $\gamma$.
\end{remark}

\subsubsection{Comparing real-\'etale and singular cohomology}
        \label{compare:retsing}

        For $F=\R$, we wish to identify real-\'etale cohomology with singular cohomology over $\R$.
        A general reference is \cite{scheiderer}, some further discussion can be found in \cite[Remark 4.4]{jacobson}.

   The inclusion $\iota\colon X(\R) \rightarrow X_\r$ sending $x$ to $(x, \R_{\geq 0})$ is continuous.   Let $\mathcal{F}$ be a locally constant sheaf on $X_\r$. Consider the sheaf $\iota^*\mathcal{F}$ on $X(\R)$. This is a locally constant sheaf.
If $X$ is separated of finite type over $\R$ and $M$ is an abelian group, then we have a canonical isomorphism $\op{H}^i(X_\r,M) \cong \op{H}^i_\sing(X(\R),M)$, see e.g.\ \cite[Remark 4.4]{jacobson}. This is generalized by the following Lemma.

        \begin{lemma}
          \label{lem:iota-on-locally-constant-sheaves}
          \label{lem:iota-on-constant-sheaves}
           If $X$ is a scheme over $F=\R$, then
          the map $\iota^*\colon \op{H}^i(X_\r,\iota_*\mathcal{F}) \rightarrow \op{H}^i(X(\R), \mathcal{F})$  is an isomorphism for any \emph{locally} constant sheaf \(\sheaf F\) on \(X(\R)\).  Moreover, \(\iota_*\) defines an equivalence of categories between the category of constant sheaves on \(X_\r\) and the category of constant sheaves on \(X(\R)\).  For an abelian group \(M\) and associated constant sheaves \(M_\r\) and \(M_\R\) on \(X_\r\) and \(X(\R)\), respectively, we have \(\iota_*M_\R = M_\r\).
        \end{lemma}
        \begin{proof}
          For the first statement, see \cite[Chapter~II, Theorem~5.7]{delfs}.  For the second statement, see \cite[Remark~4.4]{jacobson}.
        \end{proof}

To establish the desired compatibility results of the above isomorphism $\iota^*$ with various functors
        (see Remark \ref{retvariant}), one employs sheaf-theoretic arguments. In particular, given a map $f:X \to Y$ of schemes, we have a commutative
        square of continuous maps
        $f_\r \circ \iota_X = \iota_Y \circ f(\R)$
        which still commutes when applying $(-)^\ast$.

\section{Cycle class maps for real varieties}
\label{sec:cyclesheaf}

In the following section, we will discuss real cycle class maps for real varieties and schemes. We fix a scheme $X$ over \(\R\), and we write \(X(\R)\) for the set of real points of \(X\) equipped with the analytic topology, and \(\rho\colon X(\R)\hookrightarrow X\) for the inclusion of the real points into the underlying topological space of the variety \(X\).

To define a cycle class map with target the singular cohomology of $X(\R)$, there are three approaches one could pursue:
\begin{enumerate}[(A)]
\item A sheaf-theoretic approach to the definition of cycle class maps has been developed, for mod 2 coefficients by Scheiderer \cite{scheiderer:purity} and for integral coefficients by Jacobson \cite{jacobson}. This factors through real or  real-\'etale cohomology as an intermediate step.
\item A more ``geometric'' approach would be through a definition of fundamental classes in Borel--Moore homology of the real points. This is similar to the cycle class map for complex varieties \cite{fultonbook} and was done for mod 2 homology by Borel--Haefliger \cite{borel-haefliger}.
\item On a formal homotopy-theoretic level, one can also define a cycle class map as follows. The cohomology theory $\bigoplus_{i,j}\op{H}^i(-,\mathbf{I}^j)$ (here with a full bigrading) is representable in the stable homotopy category $\mathcal{SH}(F)$ by the Eilenberg--Mac~Lane spectrum for the homotopy module $\I^\bullet$. Associated with the real embedding $\sigma:F\hookrightarrow \R$ there is a real realization functor $\op{Real}:\mathcal{SH}(F)\to \mathcal{SH}$ which maps a smooth $F$-scheme $X$ to $X(\R)$ equipped with the usual euclidean topology, cf.\ \cite[Section 10]{bachmann}. Formally, the real realization functor will induce a cycle class map
  \[
    \op{H}^i(X,\mathbf{I}^j)=[X,\mathbb{H}\mathbf{I}^j[i]]_{\mathcal{SH}(F)}\to [X(\R),\op{Real}(\mathbb{H}\mathbf{I}^j)[i]]_{\mathcal{SH}}.
  \]
\end{enumerate}

In this section, we will consider the first two of these possible definitions of real cycle class maps and show that they agree. (See Remark \ref{homotopyoutline} for the third approach.) We explain the sheaf-theoretic construction of the cycle class map in \prettyref{sec:sheaf-cycle}, along with modifications involving twists and cohomology with supports in \prettyref{sec:twist-support}. Then we will discuss a more geometric description of the cycle class map (related to fundamental classes in Borel--Moore homology) in Subsection~\ref{sec:cyclegeom}. Finally, we will discuss an extension of the cycle class map to the Chow--Witt ring in Subsection~\ref{sec:chowwittcycle}. Compatibilities of these identifications and the basic operations in sheaf cohomology (pullbacks, pushforwards, products etc) will be established in \prettyref{sec:compatibilities}.

\subsection{Sheaf-theoretic description of cycle class map}
\label{sec:sheaf-cycle}

We first consider the sheaf-theoretic approach where the real cycle class map is obtained as the following composition:
\begin{equation}\label{big-diagram}
  \begin{aligned}
    \xymatrix@C=6em{
      &\op{H}^n(X,\mathbf{I}^t)\ar[ld]_{\textrm{can}} \ar[d] \ar[rd]^{\textrm{real cycle class}}\\
      \op{H}^n(X,\colim_j\mathbf{I}^j) \ar[r]^-\cong_-{\sign_\infty}
      & \op{H}^n(X,\rho_\ast\Z) \ar[r]_-{\rho^*}^-\cong  &  \op{H}^n(X(\R),\Z)
    }
  \end{aligned}
\end{equation}
We first discuss the map $\op{can}$ in the diagram which is induced by the canonical map of Zariski sheaves $\op{can}\colon \mathbf{I}^t \to \colim_j\mathbf{I}^j$.

\begin{definition}
  Multiplication with $\pfist{-1}$ induces morphisms $\pfist{-1}\colon \mathbf{I}^j\to \mathbf{I}^{j+1}$. These maps give rise to a diagram
  \[
    \cdots \to \mathbf{I}^j \xrightarrow{\pfist{-1}} \mathbf{I}^{j+1}\xrightarrow{\pfist{-1}} \mathbf{I}^{j+2}\to  \cdots
  \]
  We denote the colimit of this system by $\colim_j\mathbf{I}^j$.
\end{definition}

\begin{proposition}
  \label{prop:can-sheaf-map}
  The morphisms $\pfist{-1}\colon \mathbf{I}^j\to \mathbf{I}^{j+1}$ as well as the canonical map $\op{can}\colon \mathbf{I}^t\to \colim_j\mathbf{I}^j$ are $\Gm$-equivariant morphisms of sheaves with $\Gm$-action on the big site $\op{Sm}_{F,\op{Zar}}$. In particular, the results of subsection \ref{sec:twistedcoefficients} apply.
\end{proposition}

\begin{proof}
  Multiplication by $\pfist{-1}$ is compatible with residue maps for $\mathbf{I}^j$ because it is a global unramified symbol. This implies that $\pfist{-1}\colon\mathbf{I}^j\to\mathbf{I}^{j+1}$ induces a morphism of unramified Zariski sheaves. It is also compatible with the $\mathbb{G}_{\op{m}}$-action, because both $\pfist{-1}$ and multiplication by $\langle u\rangle$ arise from the $\mathbf{K}^{\op{MW}}_\bullet$-module structure of $\I^\bullet$. Consequently, the colimit $\colim_j\mathbf{I}^j$ is a sheaf of abelian groups with $\mathbb{G}_{\op{m}}$-action on the big site $\op{Sm}_{F,\op{Zar}}$, and the canonical map $\op{can} \colon \mathbf{I}^i\to\colim_j\mathbf{I}^j$ is compatible with these structures.
\end{proof}

\begin{remark}
  There is a more sophisticated interpretation of this map. Namely, the collection of $\Gm$-sheaves $\mathbf{I}^t$ may be interpreted as a homotopy module, as studied in detail in \cite{MField}. See Examples 3.33 and 3.34 as well as Lemma~3.35 in \cite{MField}, the relevant contractions are computed in \cite[Lemma~2.7, Proposition~2.8]{AsokFaselSpheres}.
  Then the map $\mathbf{I}^t \to \colim_j\mathbf{I}^j$ is nothing but the localization  of the homotopy module $\mathbf{I}^t$ with respect to $-[-1]\in \op{K}^{\op{MW}}_1(F)$. Using results and techniques from \cite{MField} and \cite{bachmann}, it is possible to obtain many of the results in
  the following sections using the theory of homotopy modules. We do not pursue this approach here.
\end{remark}

\begin{remark}\label{rem:can-multiplicative}
  The product \(\I^k\otimes \I^l\to \I^{k+l}\) induces a product on \(\colim_j \I^j\) such that the following diagram is commutative:
  \[
    \xymatrix{ \I^k \otimes \I^l \ar[r] \ar[d]^{\op{can}\otimes\op{can}} & \I^{k+l} \ar[d]^{\op{can}} \\
      \colim_j \I^j \otimes\colim_j \I^j \ar[r] & \colim_j\I^{j} \\}
  \]
\end{remark}

Next, following \cite{scheiderer} and \cite{jacobson}, we recall the map of Zariski sheaves which the latter reference is denoted by $\colim_j\mathbf{I}^j \to \op{supp}_\ast\circ\iota_* \Z$. Here $\iota\colon X(\mathbb{R})\to X_{\op{r}}$ is the natural morphism from the real points of $X$ to the real spectrum, and $\op{supp}\colon X_{\op{r}}\to X_{\op{Zar}}$ is the natural morphism from the real spectrum to the Zariski topology of $X$. Since the real spectrum and real-\'etale topology are not relevant for our considerations, we will only deal with the composition $\rho = \op{supp}\circ\iota\colon X(\mathbb{R})\to X_{\op{Zar}}$ taking a real point to the corresponding closed point in $X$.

\begin{definition}
  As in \cite[Section 8]{jacobson}, the signature of \prettyref{def:signature} induces a ring homomorphism $\op{sign}\colon \op{W}(X)\to \op{H}^0(X_{\op{r}},\Z)\cong \op{H}^0(X(\R),\Z)$. After Zariski sheafification and restriction to fundamental ideals, this induces a morphism of sheaves $\op{sign}\colon \mathbf{I}^n\to \rho_\ast 2^n\Z$ which we can interpret as morphisms of coefficient data
  \[
    (X,\I^j) \xleftarrow{(\rho,\sign)} (X(\R),2^j\Z)
  \]
  Since  $\pfist{-1}$ realizes to $2\in \op{W}(\mathbb{R})\cong\Z$, we get a morphism of diagrams inducing a morphism
  \[
    \op{sign}_\infty\colon \colim_j\mathbf{I}^j\rightarrow \rho_* \Z.
  \]
\end{definition}

The following is a direct consequence of the fundamental work of Jacobson~\cite{jacobson}:
\begin{proposition}
  \label{prop:J-homotopy-modules}
  The morphism $\op{sign}_\infty\colon \colim_n \mathbf{I}^n \to \rho_\ast\Z$  is an isomorphism of  sheaves of abelian groups with $\Gm$-action on the big site $\op{Sm}_{F,\op{Zar}}$. In particular, the results of subsection \ref{sec:twistedcoefficients} can be applied.

  We have a commutative diagram of sheaves with $\mathbb{G}_{\op{m}}$-action on the big site:
  \[
    \xymatrix{
      \I^j \ar[r]_-{\mathrm{can}} \ar[d]_{\sign} &
      \colim_j \I^j \ar[d]^{\sign_\infty}\\
      \rho_*2^j\Z \ar[r]_-{\mathrm{can}}^-{\cong} & \rho_*(\colim_j 2^j\Z)\mathrlap{\cong \rho_*\Z}
    }
  \]

  A similar statement is true for the mod 2 identification $\colim_n\overline{\mathbf{I}}^n\cong\rho_*(\Z/2\Z)$.
\end{proposition}

\begin{proof} The isomorphism follows from \cite[Theorem~8.6]{jacobson},
keeping in mind that $\rho_* = \op{supp}_* \circ \iota_*$ and $\iota_* \Z = \Z$. We now explain where the $\mathbb{G}_{\op{m}}$-action on $\rho_\ast\Z$ comes from. For $X$ a scheme over $\mathbb{R}$, $U\subseteq X$ open and $u\in \mathcal{O}_X(U)^\times$ a unit, we get continuous invertible function $U(\mathbb{R})\to \mathbb{R}^\times$. The $\mathbb{G}_{\op{m}}$-action is then given by multiplying a section $\sigma\in (\rho_\ast\Z)(U)\cong \Z(U(\mathbb{R}))$ by the composition $U(\mathbb{R})\to \mathbb{R}^\times\to \pm 1$. The compatibility of the signature maps $\op{sign}$ and $\op{sign}_\infty$ with the $\mathbb{G}_{\op{m}}$-actions is then clear.
The commutativity of the diagram is then also clear, cf.\ also \cite[Section~8]{jacobson}.
\end{proof}

\begin{remark}
  The identification $\colim_n\mathbf{I}^n\xrightarrow{\cong} \rho_\ast\Z$ is  also compatible with transfers on fields, cf.\ \cite[Lemma~20]{bachmann}. This ultimately implies compatibility with pushforwards along finite morphisms, but we won't discuss this in the present paper.

  The identification is also compatible with residue maps, hence induces a morphism of Gersten complexes. The residue maps for $\rho_\ast\Z$ are denoted by $\beta$ and their compatibility with the residue maps on $\mathbf{I}^n$ and $\colim_j\mathbf{I}^j$ are discussed in \cite[Section~3 ]{jacobson}.
\end{remark}

It should be noted at this point, that the left triangle in diagram~\eqref{big-diagram} commutes by definition because the identification $\op{sign}_\infty\colon \colim_j\mathbf{I}^j\xrightarrow{\cong} \rho_\ast\Z$ is induced from the signature map. The morphism $\op{H}^s(X,\mathbf{I}^t)\to \op{H}^s(X,\rho_\ast\Z)$ can either be viewed as $\op{sign}_\infty\circ\op{can}$ or $\op{can}\circ\op{sign}$ in the diagram of \prettyref{prop:J-homotopy-modules}.

The final morphisms in diagram~\eqref{big-diagram} are simply pullback morphisms induced by the morphism $\rho\colon X(\R)\to X_{\rm Zar}$, as in Example~\ref{eg:canonical-morphisms-of-coefficient-data} and \prettyref{def:sheaf-pullback}. The fact that these pullback morphisms are in fact isomorphisms is established in \cite[Remark~4.4 and Lemma~4.6]{jacobson} and \cite[Chapter~II, Theorem~5.7]{delfs}.

\begin{remark}\label{homotopyoutline}
  At this point, we can already outline how the formal homotopy-theoretic definition of the cycle class map agrees with the sheaf-theoretic definition. A central point in the sheaf-theoretic approach to cycle class maps is Jacobson's identification $\colim_n\mathbf{I}^n\cong\rho_\ast\Z$.  Combining this with \cite[Corollary 38]{bachmann}, we see that the Eilenberg--Mac~Lane spectrum for $\I^\bullet$ is mapped by real realization to the Eilenberg--Mac~Lane spectrum $\mathbb{H}\Z$ representing integral cohomology. Moreover, it is possible to trace through the definitions to see that the formal homotopy-theoretic cycle class map agrees with the sheaf-theoretic, essentially by virtue of the definition (or \cite[Proposition~36]{bachmann}). For a comparison of Thom classes obtained this way with ours, the reader may consult \cite[section 3.5]{asok-fasel:euler}.
\end{remark}

\begin{remark}
  Our discussion of cycle classes above makes no assumptions about singularities of the schemes involved, similar to the discussion in \cite{scheiderer:purity} and \cite{jacobson} for cycle classes with mod 2 and integral coefficients, respectively. However, some of the compatibility results discussed later will require restriction to smooth schemes.
\end{remark}

\subsection{Remarks on twisting and supports}
\label{sec:twist-support}

In this section, we discuss the modifications necessary to add twists by line bundles and supports to the real cycle class map discussed above. For a scheme $X/\R$ with closed subvariety $Z\hookrightarrow X$ and a line bundle $\mathcal{L}$ over $X$, this extension of the real cycle class map should have the form \[
\op{H}^s_Z(X,\mathbf{I}^t(\mathcal{L}))\to \op{H}^s_{Z(\R)}(X(\R),\Z(\mathcal{L})).
\]
As previously, the cycle class map is obtained as a composition in a diagram:
\begin{equation}\label{diagram-twist-support}
  \begin{aligned}
    \xymatrix@C=5em{
      &\op{H}^n_Z(X,\mathbf{I}^t(\mathcal{L}))\ar[ld]_{\textrm{can}} \ar[d] \ar[rd]^{\textrm{real cycle class}}\\
      \op{H}^n_{Z}(X,\colim_j\mathbf{I}^j(\mathcal{L})) \ar[r]_-{\sign_\infty}^-\cong & \op{H}^n_{Z}(X,\rho_*\Z(\mathcal{L})) \ar[r]_-{\rho^*}^-\cong &
      \op{H}^n_{Z(\R)}(X(\R),\Z(\mathcal{L}))
    }
  \end{aligned}
\end{equation}

We will spend the rest of the section explaining why the maps in the above diagram exist, and discuss some of their properties.

By \prettyref{prop:can-sheaf-map}, we know that $\op{can}\colon \mathbf{I}^t\to \colim_j\mathbf{I}^j$ is a $\mathbb{G}_{\op{m}}$-equivariant morphism of sheaves on the big site. In particular, \prettyref{def:general-twist} provides, for any real scheme $X$ with a line bundle $\mathcal{L}$, a twisted canonical map $\op{can}\colon \mathbf{I}^t(\mathcal{L})\to \colim_j\mathbf{I}^j(\mathcal{L})$ of sheaves on the small Zariski site of $X$. Note also that the isomorphism of \prettyref{prop:identtwist} is compatible with the \(\Gm\)-action, hence we have induced isomorphism \(\colim_n \I^n_{\mathcal{L}} \cong \colim_n \I^n(\mathcal{L})\). Consequently, there is no difference between the two possible ways to twist the map $\op{can}\colon \mathbf{I}^t\to \colim_j\mathbf{I}^j$, discussed in Definitions~\ref{def:W-and-I-sheaves} and \ref{def:general-twist}. Now the first map in the diagram,
\[
  \op{can}\colon \op{H}^s_Z(X,\mathbf{I}^t(\mathcal{L}))\to \op{H}^s_Z(X,\colim_j\mathbf{I}^j(\mathcal{L}))
\]
is then given by the pullback with supports, cf.\ \prettyref{def:Go-pullback-with-support}, for the twisted canonical map $\mathbf{I}^t(\mathcal{L})\to \colim_j\mathbf{I}^j(\mathcal{L})$. The same argument applies to the identification $\op{sign}_\infty\colon \colim_j\mathbf{I}^j\xrightarrow{\cong} \rho_\ast\Z$, by \prettyref{prop:J-homotopy-modules}.

We briefly discuss the middle vertical arrow
\[
  \op{H}^s_Z(X,\mathbf{I}^t(\mathcal{L}))\to \op{H}^n_Z(X,\rho_\ast\mathbb{Z}(\mathcal{L}))
\]
which can be viewed as a  twisted versions of the signature map. In fact, this morphism is the pullback on cohomology with supports for the morphism of sheaves $\mathbf{I}^t(\mathcal{L})\to \rho_\ast \Z(\mathcal{L})$, and the latter is simply the composition $\op{sign}_\infty\circ \op{can}=\op{can}\circ\op{sign}$ of \prettyref{prop:J-homotopy-modules}. It would also be possible to write out in detail how to obtain this twisted version of the signature from \prettyref{def:signature}, but we won't need this.

The signature is equivariant with respect to the natural \(\Gm\)-actions on both sides, so we can twist this morphism by an arbitrary line bundle \(\sheaf L\) over \(X\) and get the following twisted morphism of ringed coefficient data:
\[
  (X(\R),\Z(\rho^\ast\mathcal{L}))
  \xrightarrow{(\rho,\op{sign}_\infty)} (X,\colim_j\I^j(\sheaf L))
\]
For a closed subscheme $Z\subset X$,  \prettyref{def:Go-pullback-with-support} provides an induced pullback morphism on cohomology with supports
\[
  (\rho, \op{sign}_\infty)^\ast\colon \op{H}^s_Z(X,\colim_j\mathbf{I}^j(\mathcal{L})) \to \op{H}^s_{Z(\R)}(X(\R),\Z(\mathcal{L}))
\]
The real cycle class map in diagram~\eqref{diagram-twist-support} is the composition of the previous twisted signature map with the above pullback with supports.

We briefly mention the following twisted version of Jacobson's result \cite[Theorem~8.6 (i) and Corollary 8.9 (i)]{jacobson}. As in the untwisted case, it holds for 
schemes $X$ with $2$ invertible in $\mathcal{O}_X$ rather than only schemes over $\R$.
\begin{theorem}\label{thm:sign}
Let $\sheaf L$ be a line bundle on $X$.
  \begin{enumerate}
  \item The twisted signature map $\op{sign}_\infty\colon \colim_n\mathbf{I}^n(\sheaf L)\to
    (\rho_*\Z)(\sheaf L) \cong \rho_*(\Z(\rho^*\sheaf L))$ is an isomorphism.
    (Here we use the canonical isomorphism of \prettyref{lem:general-twist-of-structure-morphism}.)
  \item
    The map
    \[
      \sign_{\sheaf L}\colon \op{H}^i(X, \mathbf{I}^j(\sheaf L)) \rightarrow \op{H}^i(X, \rho_* \Z(\sheaf L)) \cong \op{H}^i(X(\R), \Z(\sheaf L(\R)))
    \]
    induced from the twisted signature map is an isomorphism of groups for $j\geq dim(X) +1 $.
  \end{enumerate}
\end{theorem}

\begin{proof}
  (1) The map is obtained using  \prettyref{def:general-twist} applied to the map $\op{sign}_\infty$, noting that  \prettyref{prop:J-homotopy-modules} provides the required $\Gm$-equivariance. Moreover, it follows from \prettyref{def:general-twist} that twisting an isomorphism of sheaves by a line bundle produces an isomorphism. The claim then follows locally from the untwisted isomorphism $\op{sign}_\infty\colon \colim_j\mathbf{I}^j\xrightarrow{\cong} \rho_\ast\Z$ of \cite[Theorem~8.6]{jacobson}.

  (2) follows from (1), as Corollary 8.9 (i) follows from Theorem~8.6 (i) in \cite{jacobson}.
\end{proof}

\begin{remark}
  \label{rem:realization-support}
  It  follows easily from the above that the real cycle class maps are compatible with the maps $\op{H}^n_{Z}(X,\I^t) \to \op{H}^n(X,\I^t)$ forgetting the support.
\end{remark}

\subsection{Geometric description of cycle class maps}
\label{sec:cyclegeom}

In this section, we discuss a more geometric and explicit description of the real cycle class map which is mostly based on the coniveau spectral sequence. It produces Gersten--Rost--Schmid-type complexes which compute the integral singular cohomology of $X(\R)$ for a smooth $\R$-scheme $X$. This is an integral version of the Gersten-type resolutions in \cite{scheiderer:purity}.

First, we note that the second part of the real cycle class map is given by an isomorphism
\[
  \op{H}^s(X,\rho_\ast \Z(\mathcal{L}))\xrightarrow{\cong} \op{H}^s(X(\mathbb{R}),\Z(\mathcal{L})).
\]
To make this isomorphism more explicit, we consider the coniveau spectral sequence computing $\op{H}^s(X,\rho_\ast\Z(\mathcal{L}))$ arising from the filtration by codimension of supports.

\begin{proposition}
  \label{prop:scheiderer-21}
  Let $X$ be a smooth scheme over $\mathbb{R}$ and let $\mathcal{L}$ be a line bundle over $X$. Then there is a complex of abelian groups
  \begin{multline*}
    0\to \bigoplus_{x\in X^{(0)}} \op{H}^0_x(X(\R),\Z(\mathcal{L})) \to \bigoplus_{x\in X^{(1)}} \op{H}^1_x(X(\R),\Z(\mathcal{L})) \to\\\to \bigoplus_{x\in X^{(2)}} \op{H}^2_x(X(\R),\Z(\mathcal{L})) \to\cdots
  \end{multline*}
  whose $q$-th cohomology is canonically isomorphic to $\op{H}^q(X(\mathbb{R}),\Z(\mathcal{L}))$. Here, we use the notation
  \[
    \op{H}^\ast_x(X(\mathbb{R}),\Z(\mathcal{L})):= \colim \op{H}^\ast_{(\overline{x}\cap U)(\R)}(U(\R),\Z(\mathcal{L}))
  \]
  with the colimit taken over Zariski open neighborhoods $U$ of $x$.
\end{proposition}

\begin{proof}
  The argument is the standard one for coniveau spectral sequences, cf.\ \cite[Theorem~2.1]{scheiderer:purity}. As in loc.\ cit., let $\mathcal{Z}^p$ be the set of closed reduced subschemes of $X$ of codimension $\geq p$, and $\mathcal{Z}^p/\mathcal{Z}^{p+1}$ be the set of pairs $(Y,Z)$ with $Y\in\mathcal{Z}^p$, $Z\in \mathcal{Z}^{p+1}$ and $Z\subset Y$. Denoting
  \[
    \op{H}^n_{\mathcal{Z}^p/\mathcal{Z}^{p+1}}(X(\mathbb{R}),\Z(\mathcal{L})):= \colim_{(Y,Z)\in \mathcal{Z}^p/\mathcal{Z}^{p+1}} \op{H}^n_{(Y\setminus Z)(\R)}((X\setminus Z)(\R),\Z(\mathcal{L})).
  \]
  the arguments in \cite{scheiderer:purity} (including the well-known analogue of the semi-purity statement \cite[Corollary 1.12]{scheiderer:purity} for singular cohomology of real points) imply the existence of a spectral sequence
  \[
    E^{p,q}_1=\op{H}^{p+q}_{\mathcal{Z}^p/\mathcal{Z}^{p+1}}(X(\R),\Z(\mathcal{L})) \Rightarrow \op{H}^{p+q}(X(\R),\Z(\mathcal{L})).
  \]
  For the rest of the argument, the isomorphism
  \[
    \phi\colon \op{H}^\ast_{\mathcal{Z}^p/\mathcal{Z}^{p+1}}(X(\R),\Z(\mathcal{L})) \cong \bigoplus_{x\in X^{(p)}}\op{H}^\ast_x(X(\R),\Z(\mathcal{L}))
  \]
  is established as in \cite{scheiderer:purity}.

  The claim then follows from the vanishing of $\op{H}^n_x(X(\R),\Z(\mathcal{L}))$ for $n\geq \op{codim} x$. By considering small enough Zariski neighborhoods of $x$, we can assume that $\overline{x}\cap U$ is smooth in $U$, and the required vanishing of such local cohomology groups follows from the cohomological dimension properties of singular cohomology.
\end{proof}

\begin{remark}
  Since the complex is a row in the $E_1$-term of a spectral sequence arising from an exact couple, we can describe the canonical morphisms from the cohomology of the complex in \prettyref{prop:scheiderer-21} to $\op{H}^q(X(\R),\Z(\mathcal{L}))$. We use the notation from the proof of the proposition. Consider the relevant piece of the complex:
  \begin{multline*}
    \cdots\to \bigoplus_{x\in X^{(q-1)}} \op{H}^{q-1}_x(X(\R),\Z(\mathcal{L})) \to \bigoplus_{x\in X^{(q)}} \op{H}^q_x(X(\R),\Z(\mathcal{L})) \xrightarrow{\partial} \\\xrightarrow{\partial} \bigoplus_{x\in X^{(q+1)}} \op{H}^{q+1}_x(X(\R),\Z(\mathcal{L})) \to\cdots
  \end{multline*}
  The boundary morphism is given by the boundary in the localization sequences. For a pair $(Y,Z)\in\mathcal{Z}^q/\mathcal{Z}^{q+1}$ it maps a class in $\op{H}^q_{(Y\setminus Z)(\R)}((X\setminus Z)(\R),\Z(\mathcal{L}))$ to its boundary in $\op{H}^{q+1}_{Z(\R)}(X(\R),\Z(\mathcal{L}))$. Consequently, a class in the kernel of $\partial$ is actually closed in $X(\R)$, and the canonical map $\ker\partial\to \op{H}^q(X(\R),\Z(\mathcal{L}))$ is given by lifting the class to $\op{H}^q_{Y(\R)}(X(\R),\Z(\mathcal{L}))$ and then forgetting the support.
\end{remark}

In the complex of \prettyref{prop:scheiderer-21}, we can further identify the local cohomology groups. The result is actually the Rost--Schmid complex in the sense of Morel \cite{MField} for the strictly $\mathbb{A}^1$-invariant sheaf of abelian groups $\colim_j\mathbf{I}^j\cong\rho_\ast\Z$. For details on the construction of Rost--Schmid complexes (in particular, the description of the differential which we won't reproduce here) see \cite[Section 5.1]{MField}.

\begin{definition}
  \label{def:local-twist}
  For a point \(x\) of a smooth $\R$-scheme \(X\) with residue field \(\kappa(x)\), we define
  \[
  \op{H}^0(x,\rho_\ast\Z):=\colim_U \op{H}^0((\bar x\cap U)(\R),\mathbb{Z}),
  \]
  with the colimit taken over the Zariski neighborhoods $U$ of $x\in X$.  This colimit has a natural action of \(\kappa(x)^\times\) defined as follows:  A section $\sigma$ of $\rho_\ast\Z$ on a neighborhood $U$ of $x$ is a section of $\Z$ on $U(\R)$.  For a sufficiently small neighborhood $U$ of $x$, an element $f\in \kappa(x)^\times$ yields an invertible $\mathbb{R}$-valued function $f\colon U(\R)\to\mathbb{R}$, and the action is given by $(f,\sigma)\mapsto \op{sgn}(f)\cdot\sigma$. More generally, for a one-dimensional \(\kappa(x)\)-vector space \(L\), we define
  \[
    \op{H}^0\left(x,\rho_\ast \Z(L)\right):= \op{H}^0\left(x,\rho_\ast\Z\right)\otimes_{\Z[\kappa(x)^\times]} \Z[L^\times].
  \]
\end{definition}

\begin{remark}
  Given a line bundle \(\sheaf L\) over \(X\), or over a Zariski open neighborhood of \(x\) in \(X\), let \(\sheaf L_x\) denote the fiber of \(\sheaf L\) at \(x\).  For this one-dimensional \(\kappa(x)\)-vector space, we have canonical isomorphisms
  \[
    \op{H}^0\left(x,\rho_\ast \Z(\sheaf L_x)\right) \cong \colim_U \op{H}^0\left((\bar x\cap U)(\R),\Z(\sheaf L)\right),
  \]
  where the colimit is again taken over the Zariski neighborhoods $U$ of $x\in X$.
\end{remark}

\begin{proposition}
  \label{prop:real-purity}
  Let $X$ be a  smooth $\R$-scheme and let $\mathcal{L}$ be a line bundle over $X$. For a point $x\in X$ of codimension $p$, we have $\op{H}^n_x(X(\R),\Z(\mathcal{L}))=0$ for $n\neq p$, and there is a canonical isomorphism
  \[
    \op{H}^p_x(X(\R),\Z(\mathcal{L}))\cong \op{H}^0(x,\rho_\ast \Z(\mathcal{L}_x\otimes \Lambda^X_x)),
  \]
  where
  \(
    \Lambda^X_x:=\bigwedge\nolimits_{\kappa(x)}^p(\mathfrak{m}_x/\mathfrak{m}_x^2)^\vee
  \)
  is the fiber of the conormal sheaf of $x$ in $X$.
\end{proposition}

\begin{proof}
Note that the colimits on both sides have the same indexing sets, the Zariski neighborhoods $U$ of the point $x\in X$. In particular, for a Zariski neighborhood $U$ of $x$ and its closed subspace $Y:=\overline{x}\cap U$, it suffices to produce canonical isomorphisms
\[
  \op{H}^p_{Y(\R)}(U(\R),\mathbb{Z}(\mathcal{L}))\cong \op{H}^0(Y(\R),\mathbb{Z}(\mathcal{L}\otimes\omega_{Y/U})),
\]
where \(\omega_{Y/U}\) is the conormal sheaf of \(Y\) in \(U\).
If they are canonical, they'll assemble into an isomorphism of diagrams and induce a canonical isomorphism of colimits. In fact, we only need to establish such canonical isomorphisms for a cofinal system of such neighborhoods. In particular, we can assume that $Y:=\overline{x}\cap U$ is smooth and a complete intersection inside $U$, defined by $p$ equations. The choice of such equations provides a section of $\omega_{Y/U}$, and equips every one of the connected components of $Y(\R)$ with a choice of orientation. Now we can apply cohomological purity: the arguments in \cite[Section 1]{scheiderer:purity} go through for singular cohomology of the real points and show that $\op{R}^pi^!\mathcal{F}\cong i^\ast\mathcal{F}$ for a locally constant sheaf $\mathcal{F}$ on $Y(\R)$.
This isomorphism is not canonical, as it depends on the choice of orientations of the components of $Y(\R)$.  But we obtain a canonical isomorphism $\op{R}^pi^!\mathcal{F}\cong i^\ast\mathcal{F}\otimes \omega_{Y/U}$ by tensoring with the section of $\omega_{Y/U}$.  For \(\sheaf F = \Z(\sheaf L)\), this canonical isomorphism of sheaves induces an isomorphism
\[
\op{H}^p_{Y(\R)}(U(\R),\mathbb{Z}(\mathcal{L}))\to \op{H}^0(Y(\R),\mathbb{Z}(\mathcal{L}\otimes\omega_{Y/U}))
\]
This map does not depend on the choice of orientation (or presentation of $Y$ as complete intersection inside $U$), and hence provides the required canonical isomorphisms and proves the claim.  In a sense, the twisting by $\Lambda^X_x$ takes care of the ambiguity of choices of orientations as discussed in \cite[Remark 1.8]{scheiderer:purity}.
\end{proof}

\begin{corollary}
  \label{cor:rost-schmid-singular}
  Let $X$ be a smooth $\R$-scheme, let $Y$ be a closed subscheme and let $\mathcal{L}$ be a line bundle on $X$.
  \begin{enumerate}
  \item Via the identifications of \prettyref{prop:real-purity}, the complex of \prettyref{prop:scheiderer-21} is canonically isomorphic to the Rost--Schmid complex for $\rho_\ast\Z$ twisted by $\mathcal{L}$:
    \[
      0\to \bigoplus_{x\in X^{(0)}}\op{H}^0(x,\rho_\ast\Z(\mathcal{L}_x\otimes \Lambda^X_x)) \to \bigoplus_{y\in X^{(1)}}\op{H}^0(y,\rho_\ast\Z(\mathcal{L}_y\otimes \Lambda^X_y))\to \cdots
    \]
    In particular, the $q$-th cohomology of the Rost--Schmid complex above is canonically isomorphic to $\op{H}^q(X(\R),\Z(\mathcal{L}))$.
  \item The $q$-th cohomology of the subcomplex
    \[
      0\to \bigoplus_{x\in X^{(0)}\cap Y}\op{H}^0(x,\rho_\ast\Z(\mathcal{L}_x\otimes \Lambda^X_x)) \to  \bigoplus_{y\in X^{(1)}\cap Y}\op{H}^0(y,\rho_\ast\Z(\mathcal{L}_y\otimes \Lambda^X_y))\to \cdots
    \]
    of the Rost--Schmid complex appearing in (1) is canonically isomorphic to $\op{H}^q_{Y(\R)}(X(\R),\Z(\mathcal{L}))$.
  \item For a quasi-projective $\R$-scheme $Y$, the $q$-th homology of the homological Rost--Schmid complex \[
      \cdots\to \bigoplus_{y\in Y_{(1)}}\op{H}^0(y,\rho_\ast\Z(\mathcal{L}_y\otimes \Lambda^Y_y)) \to \bigoplus_{x\in Y_{(0)}}\op{H}^0(x,\rho_\ast\Z(\mathcal{L}_x\otimes \Lambda^Y_x)) \to 0
    \]
    computes the Borel--Moore homology $\op{H}_n^{\op{BM}}(Y(\R),\Z(\mathcal{L}))$.
  \end{enumerate}
\end{corollary}

\begin{proof}
  (1) follows from the computation of Rost--Schmid complexes with the Gersten-type complexes coming out of the coniveau spectral sequence machinery as in Section 5.3, especially Corollary 5.44 of \cite{MField}. (2) and (3) follow from (1) as in \cite[Section 3]{scheiderer:purity}.
\end{proof}

Now we can give a more explicit description of the real cycle class map
\[
  \op{H}^s(X,\mathbf{I}^t(\mathcal{L}))\to \op{H}^s(X(\R),\Z(\mathcal{L}))
\]
for a smooth $\R$-scheme $X$ with a line bundle $\mathcal{L}$. The first morphism $\op{H}^s(X,\mathbf{I}^t(\mathcal{L}))\to \op{H}^s(X,\rho_\ast\Z(\mathcal{L}))$  simply changes the coefficients; on the level of Gersten--Rost--Schmid complexes, it is induced from the twisted signature maps $\op{sgn}\colon\op{W}(\kappa(x),\mathcal{L}_x\otimes \Lambda^X_x)\to \rho_\ast\Z(\mathcal{L}_x\otimes\Lambda^X_x)$. The second morphism
\[
  \op{H}^s(X,\rho_\ast\Z(\mathcal{L}))\to \op{H}^s(X(\R),\Z(\mathcal{L}))
\]
has an explicit description in terms of Rost--Schmid complexes as discussed above. The real cycle class map is the composite of these two maps.

\begin{remark}
  Although the above results are close parallels of \cite[Section~3]{scheiderer:purity}, we cannot obtain a definition of fundamental classes as in the definition before Proposition~3.3 in loc.\ cit.: First of all, an irreducible closed subscheme $Z\subset X$ of codimension $q$ with generic point $z$ does not come with a canonical section ``\(1\)'' of $\op{H}^0(z,\Z(\mathcal{L}_z\otimes \Lambda^X_z))$.  Secondly, and more importantly, there is no reason to expect any such non-trivial section to be unramified, i.e.\ to define a cocycle in the complex of \prettyref{cor:rost-schmid-singular}~(2) and hence a class in $\op{H}^0_{Z(\R)}(X(\R),\Z(\mathcal{L}))$.  In case $Z$ is smooth, the line bundle $\mathcal{L}|_Z\otimes\omega_{Z/X}$ could be nontrivial; worse problems arise in the singular case. It's the same orientability problem as always: closed subvarieties in $X$ don't necessarily give rise to cohomology classes for $X(\R)$, this only works with $\Z/2\Z$ coefficients. On the other hand, even though fundamental classes generally don't exist for subvarieties, this is no obstruction for the cycle class map $\op{H}^s(X,\mathbf{I}^t(\mathcal{L}))\to \op{H}^s(X(\R),\Z(\mathcal{L}))$, since the cycles in \(\I\)-cohomology have the solutions to the orientability problem built into their coefficients.
\end{remark}

For later use, and for the sake of example, we provide a computation of the cycle class map for $\mathbb{A}^n\setminus\{0\}$ over $\R$.

\begin{proposition}
  \label{prop:computation-cycle-class}
  For $\mathbb{A}^n=\mathbb{A}^n_{\R}$, the real cycle class map $\op{H}^{n-1}(\mathbb{A}^n\setminus\{0\},\mathbf{I}^n)\to \op{H}^{n-1}(\mathbb{R}^n\setminus\{0\},\Z)$ is an isomorphism for all $n \geq 0$.

  More precisely, equip $\mathbb{A}^n\setminus\{0\}$ with the orientation given by trivializing $\omega_{\mathbb{A}^n\setminus\{0\}}$ by $\wedge_{i=1}^n{\op{d}}x_i$, and similarly equip $\mathbb{R}^n\setminus\{0\}$ with the orientation given by the standard coordinate functions. For $n \geq 2$, these choices of orientations provide isomorphisms $\op{H}^{n-1}(\mathbb{A}^n\setminus\{0\},\mathbf{I}^n)\cong \op{W}(\R)\cong\Z$ and $\op{H}^{n-1}(\R^n\setminus\{0\},\Z)\cong\Z$ such that the following composition is the identity \[
    \Z\cong \op{H}^{n-1}(\mathbb{A}^n\setminus\{0\},\mathbf{I}^n)\to \op{H}^{n-1}(\mathbb{R}^n\setminus\{0\},\Z)\cong\Z.
  \]
\end{proposition}

\begin{proof}
  When \(n=1\), we need to compare \(\op{H}^0(\A^1\setminus\{0\},\W)\) and \(\op{H}^0(\R\setminus\{0\},\Z)\), both of which are isomorphic to \(\Z\oplus \Z\).  We leave it as an exercise to the reader to verify that the cycle class map is an isomorphism in this case,
  and assume \(n\geq 2\) for the remainder of the proof.
  We first investigate the cohomology group $\op{H}^{n-1}(\mathbb{A}^n\setminus\{0\},\mathbf{I}^n)$. The following statements can essentially be found in \cite[Section 3.3]{fasel:orbits}, see also \cite[Lemma 4.2.5]{ABF:models}. A cycle in the cohomology group consists of codimension $n-1$-points $x$, i.e., generic points of curves in $\mathbb{A}^n\setminus\{0\}$, and the multiplicities are elements in $\op{I}(\kappa(x))$. By the standard computations of cohomology of spheres, $\op{H}^{n-1}(\mathbb{A}^n\setminus\{0\},\mathbf{I}^n)\cong \op{W}(F)$, and the canonical morphism sends a cycle to its residue at the origin in $\mathbb{A}^n$. In particular, an explicit generator for $\op{H}^{n-1}(\mathbb{A}^n\setminus\{0\},\mathbf{I}^n)$ is given by the form $\langle 1, X_1\rangle$ on the line $X_2=\cdots=X_n=0$ through the origin; the better way of writing this which takes proper care of the choices of orientation is using the Koszul complex as in \cite[Section 3.3]{fasel:orbits}. Note that by definition of the boundary map in the localization sequence, this class corresponds to the class in $\op{H}^n_{\{0\}}(\mathbb{A}^n,\mathbf{I}^n)$ of the form $\langle 1\rangle$ supported on the origin. Moreover, the choice of orientation formulated in the statement of the proposition is such that the above generator is mapped to 1 under the isomorphism $\op{H}^{n-1}(\mathbb{A}^{n}\setminus\{0\},\mathbf{I}^n)\cong \op{W}(F)$.

  The other cohomology group we want to consider is the one arising from the Gersten-type complex for integral cohomology of $\mathbb{R}^n\setminus\{0\}$. The relevant piece of the complex is
  \[
    \op{H}^{n-2}_{\mathcal{Z}^{n-2}/\mathcal{Z}^{n-1}}(\mathbb{R}^n\setminus\{0\},\Z) \to \op{H}^{n-1}_{\mathcal{Z}^{n-1}/\mathcal{Z}^{n}}(\mathbb{R}^n\setminus\{0\},\Z) \xrightarrow{\partial} \op{H}^{n}_{\mathcal{Z}^{n}}(\mathbb{R}^n\setminus\{0\},\Z)
  \]
  The boundary morphism $\partial$ is given by the boundary in the localization sequence, for a pair $(Y,Z)$ of irreducible closed subschemes of $\mathbb{A}^n\setminus\{0\}$ of codimensions $n-1$ and $n$, respectively, it maps a class in $\op{H}^{n-1}_{Y\setminus Z(\mathbb{R})}(\mathbb{R}^n\setminus (\{0\}\cup Z(\mathbb{R})),\Z)$ to its boundary, which is a class in $\op{H}^n_{Z(\mathbb{R})}(\mathbb{R}^n\setminus\{0\},\Z)$. The canonical map from the cohomology of the above complex at $\op{H}^{n-1}$ to the sheaf cohomology of $\mathbb{R}^n\setminus\{0\}$ is the morphism  $\ker\partial\to \op{H}^{n-1}(\mathbb{R}^n\setminus\{0\},\Z)$ which forgets the support. In this picture, we can also identify a generator of the cohomology in the Gersten complex: take a generator of $\op{H}^{n-1}(\mathbb{R}^n\setminus\{0\},\Z)$ and use the canonical identification
  \[
    \op{H}^{n-1}(\mathbb{R}^n\setminus\{0\},\Z) \cong \op{H}^{n-1}(\mathbb{R}^n\setminus\{0\},\mathbb{R}^n \setminus \mathbb{R}^1,\Z)\cong \op{H}^{n-1}_{\mathbb{R}^1\setminus \{0\}}(\mathbb{R}^n\setminus\{0\},\Z)
  \]
  where the embedding $\mathbb{R}^1\hookrightarrow\mathbb{R}^n$ is given by the first coordinate axis. As before, the choice of orientation used here is such that the isomorphism $\op{H}^{n-1}_{\mathbb{R}^1\setminus\{0\}}(\mathbb{R}^n\setminus\{0\},\Z)\cong \Z$ takes the section $\op{sgn}(X_1)$ of the sheaf $\Z$ on the line $X_2=\cdots=X_n=0$ to $1\in\Z$.

  Now we can trace through the explicit description of the cycle class map, using the signature to go from the Gersten complex for $\mathbf{I}^n$ to the Gersten complex for sheaf cohomology of $\Z$ on $X(\mathbb{R})$. The generator $\langle 1,X_1\rangle$ in the Gersten complex for $\mathbf{I}^n$ has signature 2 on the half-line $X_1>0$ and -2 on the half-line $X_1<0$, viewed as a generator of the global sections of the sheaf $\rho_\ast 2\Z$ on the line $X_2=\cdots=X_n=0$. Its image under the canonical map $\op{can}\colon\rho_\ast 2\Z\to \rho_\ast (\colim_j 2^j\Z)\cong \rho_\ast \Z$ is then exactly the generator of the global sections of $\rho_\ast\Z$ discussed above, see also the notational square at the beginning of Section 3.
\end{proof}

\begin{remark}
  We can also provide a quick and easy description of the cycle class map $\op{H}^n_{\{0\}}(\mathbb{A}^n,\mathbf{I}^n)\to \op{H}^n_{\{0\}}(\R^n,\Z)$. The complex computing \(\I\)-cohomology with support is simply
  \[
    \cdots\to 0\to \op{W}(\mathbb{R})(\Lambda^{\mathbb{A}^n}_0)\to 0\to \cdots
  \]
  where the nontrivial group is placed in degree $n$, coming from the closed point $0\in\mathbb{A}^n$. By \prettyref{cor:rost-schmid-singular}, the complex computing $\op{H}^n_{\{0\}}(\R,\Z)$ is similarly given by
  \[
    \cdots\to 0\to \Z(\Lambda^{\A^n}_0)\to 0\to\cdots
  \]
  The orientations used are in both cases $\bigwedge\nolimits_{i=1}^n X_i$ (which is an element of $\Lambda^{\mathbb{A}^n}_0$ and therefore provides an isomorphism $\Lambda^{\mathbb{A}^n}_0\cong\mathbb{R}$), so that the morphism $\op{W}(\mathbb{R})(\Lambda^{\mathbb{A}^n}_0)\to \Z(\Lambda^{\mathbb{A}^n}_0)$ is simply given by the signature morphism $\op{W}(\R)\to \Z$ mapping the form $\langle 1\rangle$ to 1.
\end{remark}

In slight variation of the discussion in \cite[Section 3]{scheiderer:purity}, the mod 2 cycle class map $\op{Ch}^n(X)\to \op{H}^n(X(\R),\Z/2\Z)$ of Borel--Haefliger \cite{borel-haefliger} has the following sheaf-theoretic description: $\op{H}^n(X(\R),\Z/2\Z)$ can be computed in terms of a Gersten resolution
\[
  0\to \bigoplus_{x\in X^{(0)}}\op{H}^0(x,\rho_\ast\Z/2\Z)\to \bigoplus_{y\in X^{(1)}}\op{H}^0(y,\rho_\ast\Z/2\Z)\to \bigoplus_{z\in X^{(2)}}\op{H}^0(z,\rho_\ast\Z/2\Z)\to \cdots
\]
The cycle class map is then induced from mapping an irreducible closed subvariety $Y\subset X$ with generic point $y$ to the constant section $1\in\op{H}^0(y,\rho_\ast\Z/2\Z)$.

\begin{proposition}
  \label{prop:jacobson-vs-borel-haefliger}
  Let $X$ be a smooth variety over $\R$. Then there is a commutative triangle
  \[
    \xymatrix{
      \op{Ch}^n(X) \ar[rr] \ar[rd]_\cong && \op{H}^n(X(\R),\Z/2\Z) \\
      &\op{H}^n(X,\mathbf{I}^n/\mathbf{I}^{n+1}) \ar[ru]
    }
  \]
  where the top horizontal map is the Borel--Haefliger cycle class map, the downward isomorphism is induced from the identification $\mathbf{K}^{\op{M}}_n/2\cong  \mathbf{I}^n/\mathbf{I}^{n+1}$ and the upward arrow is Jacobson's cycle class map.
\end{proposition}

\begin{proof}
  As mentioned above, \cite[Section 3]{scheiderer:purity} implies that the horizontal map is induced on Gersten complexes by mapping an irreducible closed subvariety $Y\subset X$ with generic point $y$ to $1\in \op{H}^0(y,\rho_\ast\Z/2\Z)$. The isomorphism $\op{Ch}^n(X)\cong \op{H}^n(X,\mathbf{I}^n/\mathbf{I}^{n+1})$ is also induced from a morphism on Gersten complexes: in the relevant degree, we have \[
    \bigoplus_{y\in X^{(n)}}\Z/2\Z\cong \bigoplus_{y\in X^{(n)}}\op{W}(\kappa(y))/\op{I}(\kappa(y))
  \]
  with the identification (from right to left) given by mapping a quadratic form to its rank mod 2. Finally, Jacobson's cycle class map $\op{H}^n(X,\mathbf{I}^n/\mathbf{I}^{n+1})\to \op{H}^n(X(\R),\Z/2\Z)$ is induced on the level of Gersten complexes by the direct sum of signature maps
  \[
    \bigoplus_{y\in X^{(n)}}\op{W}(\kappa(y))/\op{I}(\kappa(y)) \to \bigoplus_{y\in X^{(n)}}\op{H}^0(y,\rho_\ast\Z/2\Z).
  \]
  The commutativity of the triangle then follows from the fact that for a real closed field $F$, the two morphisms $\Z\cong \op{W}(F)\to \Z/2\Z$ given by rank mod 2 and signature mod 2 agree.
\end{proof}

\begin{remark}
  Kirsten Wickelgren has informed us that the above identification of the reduction of Jacobson's cycle class map with the one of Borel--Haefliger was already established by Jacobson in unpublished work.
\end{remark}

\begin{remark}
  At some point we wondered if there could be a real analogue of Totaro's factorization \cite{totaro:mu} of the cycle class map
  \[
    \op{CH}^n(X)\to \op{MU}^\ast(X(\mathbb{C}))\otimes_{\op{MU}^\ast}\Z \to \op{H}^{2n}(X(\mathbb{C}),\Z)
  \]
  for a smooth complex variety $X$. It now seems that a comparable factorization of the real cycle class map $\op{H}^n(X,\mathbf{I}^n)\to \op{H}^n(X(\R),\Z)$ through $\op{MSO}$ would not be possible (and a factorization of the mod 2 cycle class map through $\op{MO}$ would be useless because $\op{MO}$ splits as a wedge of Eilenberg--Mac Lane spaces).

  The basic reason lies in orientability issues. To map a cycle in $\op{H}^n(X,\mathbf{I}^n)$ to an element of $\op{MSO}^\ast(X(\R))\otimes_{\op{MSO}^\ast}\Z$, we would take an irreducible subvariety $Z\subset X$, take a resolution of singularities $\tilde{Z}\to Z$ and take the realization  $\tilde{Z}(\R)\to X(\R)$ of the composition $\tilde{Z}\to Z\to X$. For this to yield a class in $\op{MSO}^\ast(X)$, the map needs to be orientable, but we don't have enough control over the resolution of singularities to guarantee that. Moreover, when we want to show independence of resolutions, further blowups are required to be able to compare two given resolutions $\tilde{Z}_1\to Z$ and $\tilde{Z}_2\to Z$. Depending on the (co)dimensions, the real realization of a blowup of a smooth subvariety isn't orientable, and then there are no pushforward maps to compare fundamental classes. So the construction of an MSO-factorization seems to break down at a rather fundamental level.
\end{remark}

\subsection{Cycle class maps for Chow--Witt rings}
\label{sec:chowwittcycle}

In this section, we will discuss a cycle class map for Chow--Witt rings. We first recall the equivariant cycle class map on Chow groups of real varieties introduced by Benoist and Wittenberg in \cite{benoist:wittenberg} and then combine this with the cycle class map on $\I^n$-cohomology discussed above.

Based on earlier work of Krasnov and van Hamel, Benoist and Wittenberg \cite{benoist:wittenberg} discuss a cycle class map  $\op{CH}^\bullet(X)\to \op{H}^\bullet_{\op{C}_2}(X(\mathbb{C}),\Z)$ for $X$ a real variety, which maps from the Chow groups to $\op{C}_2$-equivariant cohomology of the complex points with integral coefficients \cite{benoist:wittenberg}. For an irreducible closed smooth subvariety $i\colon Z\subset X$ of codimension $k$, the cycle class map takes $[Z]$ to the fundamental class $[Z]\in \op{H}^{2k}_{\op{C}_2}(X(\mathbb{C}),\Z(k))$, which is given as the image of $1\in \Z= \op{H}^0_{\op{C}_2}(Z(\mathbb{C}),\Z)$
under the proper pushforward map
\[
  i_\ast\colon\op{H}^{0}_{\op{C}_2}(Z(\mathbb{C}),\Z)\cong \op{H}^{2k}_{\op{C}_2,Z(\mathbb{C})}(X(\mathbb{C}),\Z(k))\to  \op{H}^{2k}_{\op{C}_2}(X(\mathbb{C}),\Z(k)).
\]
The general definition of the cycle class map for a potentially singular irreducible closed subvariety $Z\subset X$ is obtained by d\'evissage. The equivariant cycle class map $\op{CH}^\bullet(X)\to \op{H}^\bullet_{\op{C}_2}(X(\mathbb{C}),\Z)$ is then compatible with proper pushforward, products and pullbacks, cf.\ \cite[Propositions 2.1.3, 2.3.3]{krasnov}. It is also easy to see that the equivariant cycle class map is compatible with the classical cycle class map via complexification (or alternatively forgetting the $\op{C}_2$-equivariance), since the latter also takes $Z\subset X$ to $i_\ast 1$ with
\[
  i_\ast\colon\op{H}^{0}(Z(\mathbb{C}),\Z)\cong \op{H}^{2k}_{Z(\mathbb{C})}(X(\mathbb{C}),\Z)\to  \op{H}^{2k}(X(\mathbb{C}),\Z).
\]

The images of algebraic cycles under the equivariant cycle class maps satisfy additional constraints related to the Steenrod squares. Benoist and Wittenberg consider the composition
\[
  \op{H}^{2k}_{\op{C}_2}(X(\mathbb{C}),\Z(k)) \to \op{H}^{2k}_{\op{C}_2}(X(\R),\Z(k))\to \bigoplus_{0\leq p\leq 2k, p\equiv k\bmod 2} \op{H}^p(X(\R),\Z/2\Z)
\]
where the second map arises from the decomposition
\[
  \op{H}^i_{\op{C}_2}(X(\R),\Z) \cong \bigoplus_{p+q=i}\op{H}^p(X(\R),\op{H}^q(\op{C}_2,\Z(j))).
\]
For a class $\alpha\in \op{H}^{2k}_{\op{C}_2}(X(\mathbb{C}),\Z(k))$, the component of the image in $\op{H}^p(X(\R),\Z/2\Z)$ is denoted by $\alpha_p$. With this notation, the restricted equivariant cohomology is defined as
\[
  \op{H}^{2k}_{\op{C}_2}(X(\mathbb{C}),\Z(k))_0 =\{\alpha\mid \alpha_{k+i}=\op{Sq}^i(\alpha_k) \textrm{ for all } i\in 2\Z\}\subseteq \op{H}^{2k}_{\op{C}_2}(X(\mathbb{C}),\Z(k)).
\]
Due to results of Kahn and Krasnov, the equivariant cycle class map factors through a morphism $\op{CH}^k(X)\to \op{H}^{2k}_{\op{C}_2}(X(\mathbb{C}),\Z(k))_0$.

The following compatibility of cycle class maps is noted in \cite{benoist:wittenberg}, after Definition~1.17.

\begin{proposition}
  \label{prop:krasnov-vs-borel-haefliger}
  Let $X$ be a smooth quasi-projective variety over $\R$ (or more generally over a real closed field $F$). Then there is a commutative diagram
  \[
    \xymatrix{
      \op{CH}^n(X) \ar[r] \ar[d] & \op{H}^{2n}_{\op{C}_2}(X(\mathbb{C}),\Z(n))_0 \ar[d] \\
      \op{Ch}^n(X) \ar[r] & \op{H}^n(X(\R),\Z/2\Z)
    }
  \]
  where the top horizontal is the refined equivariant cycle class map discussed above, and the bottom horizontal is the Borel--Haefliger cycle class map \cite{borel-haefliger}, which we may reinterpret in our terms using \prettyref{prop:jacobson-vs-borel-haefliger} above. The left vertical map is the natural reduction mod 2, and the right-hand vertical map is the reduction map  \cite[Eq. (1.58)]{benoist:wittenberg}.
\end{proposition}

Finally, we can combine the cycle class maps considered above into a cycle class map defined on the Chow--Witt ring. Recall that there is a fiber product description of the Milnor--Witt K-theory sheaves, cf.\ \cite{morel:puissances}:
\[
  \mathbf{K}^{\op{MW}}_n\cong \mathbf{K}^{\op{M}}_n\times_{\mathbf{I}^n/\mathbf{I}^{n+1}} \mathbf{I}^n
\]
where the maps are the natural projection $\mathbf{I}^n\to \mathbf{I}^n/\mathbf{I}^{n+1}$
and the composition of the projection $\mathbf{K}^{\op{M}}_n\to \mathbf{K}^{\op{M}}_n/2$ with Milnor's homomorphism $\mathbf{K}^{\op{M}}_n/2 \to \mathbf{I}^n/\mathbf{I}^{n+1}$.  (By the solution to the Milnor conjecture, the homomorphism of Milnor that appears here is an
isomorphism.) The fiber product presentation above induces a commutative square
\[
  \xymatrix{
    \widetilde{\op{CH}}^n(X,\sheaf L) \ar[r] \ar[d] & \ker\partial \ar[d] \\
    \op{H}^n(X,\mathbf{I}^n(\mathcal{L}))\ar[r] & \op{Ch}^n(X)
  }
\]
where $\partial\colon \op{CH}^n(X)\to \op{H}^{n+1}(X,\mathbf{I}^n(\mathcal{L}))$ is the boundary map in sheaf cohomology associated to the short exact sequence of sheaves $0\to \mathbf{I}^{n+1}(\mathcal{L})\to \mathbf{K}^{\op{MW}}_n(\mathcal{L})\to \mathbf{K}^{\op{M}}_n\to 0$, an analogue of integral Bockstein maps in classical topology. By \cite[Proposition~2.11]{hornbostelwendt}, the map $\widetilde{\op{CH}}^n(X,\sheaf L) \to \op{H}^n(X,\mathbf{I}^n(\mathcal{L})) \times_{\op{Ch}^n(X)} \ker\partial$ is surjective, and injective whenever the 2-torsion in $\op{CH}^n(X)$ is trivial.

By \prettyref{prop:krasnov-vs-borel-haefliger}, the composition of the equivariant cycle class map of \cite{benoist:wittenberg}  $\op{CH}^n(X)\xrightarrow{\cong} \op{H}^{2n}_{\op{C}_2}(X(\mathbb{C}), \Z(n))_0$ of \prettyref{prop:bw-iso} with the inclusion $\ker\partial\subseteq \op{CH}^n(X)$ factors through the kernel of the composition
\[
  \psi\colon \op{H}^{2n}_{\op{C}_2}(X(\mathbb{C}),\Z(n))_0\to \op{H}^n(X(\mathbb{R}),\Z/2\Z)\xrightarrow{\beta} \op{H}^{n+1}(X(\mathbb{R}),\Z).
\]

\begin{corollary}\label{CWcycleclassmap}
  Let $X$ be a smooth quasi-projective variety over $\R$ (or more generally over a field $F$ with a real embedding $\sigma\colon F\hookrightarrow \R$).
  Then there is a well-defined cycle class map
  \[
    \widetilde{\op{CH}}^n(X,\sheaf L)\to  \op{H}^n(X(\R),\Z(\mathcal{L})) \times_{\op{H}^n(X(\R),\Z/2\Z)} \ker\psi.
  \]
\end{corollary}

\begin{proof}
  The cycle class map is the composition of the above natural projection map $\widetilde{\op{CH}}^n(X,\sheaf L) \to \op{H}^n(X,\mathbf{I}^n(\mathcal{L})) \times_{\op{Ch}^n(X)} \ker\partial$ with the map
  \[
    \op{H}^n(X,\mathbf{I}^n(\mathcal{L})) \times_{\op{Ch}^n(X)} \ker\partial \to \op{H}^n(X(\R),\Z(\mathcal{L})) \times_{\op{H}^n(X(\R),\Z/2\Z)} \ker\psi
  \]
  which is the fiber product of Jacobson's cycle class map
  \[
    \op{H}^n(X,\mathbf{I}^n(\mathcal{L}))\to \op{H}^n(X(\R),\Z(\mathcal{L}))
  \]
  discussed above, and the refined cycle class map $\op{CH}^n(X)\to \op{H}^{2n}_{\op{C}_2}(X(\mathbb{C}),\Z(n))_0$ of Benoist--Wittenberg \cite{benoist:wittenberg} (restricted to $\ker\partial\to\ker\psi$ as discussed above). The result follows from Propositions~\ref{prop:krasnov-vs-borel-haefliger} and \ref{prop:jacobson-vs-borel-haefliger} and the following commutative diagram
  \[
    \xymatrix{
      \op{H}^n(X,\mathbf{I}^n(\mathcal{L})) \ar[d] \ar[r] & \op{H}^n(X(\R),\Z(\mathcal{L})) \ar[d] \\
      \op{H}^n(X,\mathbf{I}^n/\mathbf{I}^{n+1}) \ar[r] & \op{H}^n(X(\R),\Z/2\Z)
    }
  \]
  where the horizontal maps are Jacobson's cycle class maps and the vertical maps are the morphisms induced from the natural projection maps $\mathbf{I}^n\to \mathbf{I}^n/\mathbf{I}^{n+1}$ and $\Z\to\Z/2\Z$, respectively. The latter commutative diagram follows directly from \cite{jacobson} since both cycle class maps are induced from the signature on quadratic forms, hence compatible with reduction mod 2.
\end{proof}

\begin{remark}
  It would be interesting if this could be refined to a cycle class map whose target is the cohomology of a fiber product of sheaves in the same way as $\mathbf{K}^{\op{MW}}_n\cong \mathbf{I}^n\times_{\mathbf{I}^n/\mathbf{I}^{n+1}}\mathbf{K}^{\op{M}}_n$.
  This would require to state and prove a topological version of the Milnor conjecture. Then we could hope to detect the 2-torsion in the kernel of the natural projection
  \[
    \widetilde{\op{CH}}^\bullet(X,\sheaf L)\to \op{H}^\bullet(X,\mathbf{I}^\bullet(\sheaf L)) \times_{\op{Ch}^\bullet(X)}\op{CH}^\bullet(X).
  \]
  Of course, for cellular varieties as discussed in Section~\ref{sec:cellular}, this doesn't play any role at all.
\end{remark}

\section{Compatibilities}
\label{sec:compatibilities}

In this section, we prove the main results on compatibility of the real cycle class maps with pullbacks, pushforwards and the multiplicative structure.

\subsection{Compatibility with pullbacks}

\begin{proposition}\label{prop:pullbacks-compatible}
  Let $f\colon X\to Y$ be a morphism of smooth schemes over $\R$, let $i\colon Z\hookrightarrow Y$ be a (reduced) closed subscheme and let $\sheaf L$ be a line bundle over $Y$. The morphisms in diagram~\eqref{diagram-twist-support} are compatible with pullback along~\(f\).  In particular, the following diagram commutes for all non-negative integers \(s\) and \(t\):
  \[ \xymatrix{
      \op{H}^s_{f^{-1}(Z)}(X,\mathbf{I}^t(f^*\sheaf L))\ar[r] &\op{H}^s_{f^{-1}(Z)(\R)}(X(\R),\Z(f^*\sheaf L))
      \\
      \op{H}^s_Z(Y,\mathbf{I}^t(\sheaf L)) \ar[u]_{f^\ast}  \ar[r] &\op{H}^s_{Z(\R)}(Y(\R),\Z(\sheaf L)) \ar[u]_{f^\ast}
    }
  \]
\end{proposition}

\begin{proof}
  Recall from \prettyref{sec:cohomology-supports} that we define the cohomology groups with support $\op{H}^s_Z(Y,\mathbf{I}^t(\mathcal{L}))$ as the derived functor of $\Gamma_Z:=\Gamma i_\ast i^!$ evaluated on the sheaf $\mathbf{I}^t(\mathcal{L})$, and similarly for all the other cohomology groups and sheaves involved. We denote by $i'\colon f^{-1}(Z) \hookrightarrow X$ the base change of $i$ along $f$. Below, all cohomology groups are sheaf cohomology groups, and the pullback maps are those defined in \prettyref{def:Go-pullback-with-support} and \prettyref{def:I-pullback}. The claim will follow from three commutative diagrams, involving the maps can, $\sign_\infty$ and $\rho^\ast$ in the definition of the real cycle class map of diagram~\eqref{diagram-twist-support}. The commutativity of the diagrams below will be checked using the functoriality statement for pullbacks, cf.~\prettyref{prop:sheaf-pullback-functorial}.

  (1) We first discuss the map $\op{can}\colon \mathbf{I}^t\to\colim_j\mathbf{I}^j$. By \prettyref{prop:GoGe-pullback-compatible}, the sheaf pullback $f^\ast\colon \op{H}^s_Z(Y,\mathbf{I}^t(\mathcal{L}))\to \op{H}^s_{f^{-1}(Z)}(X,\mathbf{I}^t(f^\ast\mathcal{L}))$ agrees with the usual pullback defined in terms of Gersten complexes. By \prettyref{prop:sheaf-pullback-functorial} and \prettyref{cor:twist-key}, it suffices to show that the following diagram of morphisms of coefficient data is commutative
  \[
  \xymatrix{
  (X,\mathbf{I}^t(f^\ast\mathcal{L}))\ar[d]_{(f,\phi)} & (X,\colim_j\mathbf{I}^j(f^\ast\mathcal{L}))  \ar[l]_-{(\mathrm{id}, \mathrm{can})} \ar[d]^{(f,\phi_{\colim})} \\
  (Y,\mathbf{I}^t(\mathcal{L}))  & (Y,\colim_j\mathbf{I}^j(f^\ast\mathcal{L})) \ar[l]^-{(\mathrm{id},\mathrm{can})}.
  }
  \]
  Hopefully the direction of the horizontal arrows following the conventions in the definition of coefficient data at the beginning of \prettyref{sec:prelims} is not too confusing. The commutativity of the above diagram then is equivalent to the commutativity of the following diagram of morphisms of sheaves on $Y$:
  \[
  \xymatrix{
  \mathbf{I}^t(\mathcal{L}) \ar[r]^\phi \ar[d]_{\mathrm{can}} & f_\ast \mathbf{I}^t(f^\ast\mathcal{L}) \ar[d]^{f_\ast \mathrm{can}} \\
  \colim_j\mathbf{I}^j(\mathcal{L}) \ar[r]_{\phi_{\colim}} & f_\ast \colim_j\mathbf{I}^j(f^\ast\mathcal{L}),
  }
  \]
where the horizontal morphisms $\phi$ and $\phi_{\colim}$ are the structure morphisms discussed before \prettyref{def:induced-morphism-of-cd}. The commutativity of this diagram  follows from \prettyref{eg:twist-functorial} since the morphism $\op{can}\colon \mathbf{I}^t\to \colim_j\mathbf{I}^j$ is a morphism of sheaves on the big site.

  (2) Now we consider the identification $\sign_\infty\colon \colim_j\mathbf{I}^j\cong \rho_\ast\Z$. As in Step (1), using \prettyref{prop:sheaf-pullback-functorial} and \prettyref{cor:twist-key}, it suffices to prove commutativity of the following diagram of sheaves on $Y$:
  \[
    \xymatrix{
      \colim_n\mathbf{I}^n_Y(\sheaf L) \ar[r]^{\sign_\infty} \ar[d]_{\phi_{\colim}} & (\rho_Y)_\ast \Z(\sheaf L) \ar[d]^\tau \\
      f_\ast \colim_n\mathbf{I}^n_X(f^\ast\sheaf L) \ar[r]_{\sign_\infty} & f_\ast(\rho_X)_\ast \Z(f^\ast\sheaf L)
    }
  \]
  The commutativity of this square is again a consequence of \prettyref{eg:twist-functorial} since Jacobson's isomorphism $\sign_\infty\colon \colim_j\mathbf{I}^j\cong\rho_\ast\Z$ is a morphism of sheaves with $\mathbb{G}_{\op{m}}$-action on the big site and consequently can be twisted, cf.\ \prettyref{thm:sign}.

  (3) The last square concerns the compatibility of pullback morphisms with the pullback isomorphism $\rho^\ast$ to singular cohomology. As before, using \prettyref{prop:sheaf-pullback-functorial} and \prettyref{cor:twist-key}, it suffices to show that the following square of morphisms of coefficient data is commutative:
  \[
  \xymatrix{
  (X,\rho_\ast \mathbb{Z}(f^\ast\mathcal{L})) \ar[d]_{(f,\tau)} & (X(\R),\mathbb{Z}(f^\ast\mathcal{L})) \ar[l]_\rho \ar[d]^{(f,\tau)} \\
  (Y,\rho_\ast \Z(\mathcal{L}))  & (Y(\R),\Z(\mathcal{L})) \ar[l]^\rho.
  }
  \]
Note that here and below we are also tacitly using the canonical isomorphism of line bundles \((f^*\sheaf L)(\R) \cong f_\R^*\sheaf L(\R)\) on \(X(\R)\).

  On the level of continuous maps of topological spaces, the relevant input is the obvious commutativity of the diagram
  \[
  \xymatrix{
  X(\R) \ar[r]^\rho \ar[d]_f & X_{\mathrm{Zar}} \ar[d]^f \\ Y(\R) \ar[r]_\rho & Y_{\mathrm{Zar}}.
  }
  \]
  On the level of morphisms of sheaves on $Y$, this commutative diagram also implies the commutativity of the triangle
  \[
  \xymatrix{
  \rho_\ast\mathbb{Z}(\mathcal{L})\ar[r]^\tau \ar[rd]_{\rho_\ast(\tau)} & f_\ast \rho_\ast \mathbb{Z}(f^\ast\mathcal{L}) \ar[d]^\cong \\& \rho_\ast f_\ast\Z(f^\ast\mathcal{L})
  }
  \]
  and this implies the claim.

\end{proof}

\begin{remark} \label{retvariant}
The results and techniques of this subsection obviously apply also to real cohomology (with pullbacks defined as in  \prettyref{sec:construction-Godement}). More precisely, they apply to the map \(\op{supp}\colon X_{\op{r}}\to X\) for $X$ over any base field $F$, instead of \(\rho\colon X(\R)\to X\) and $X$ over $\R$.
If $F$ has a real embedding $F \hookrightarrow \R$, the results also apply to the induced map \(\iota\colon X_{\R}(\R)\to (X_{\R})_{\op{r}}\). Note that for \(F=\R\), the map $\rho$ is simply the composition \(\rho = \op{supp} \circ \iota\).
Similar statements apply to the compatibility results on products in the next subsection. For pushforwards, the real variant is less obvious, as one first would have to establish a theory of Thom classes in this setting. For other results related to descriptions using Gersten complexes in the real setting as e.g.\ in \cite{scheiderer:purity} we refer to Appendix~\ref{sec:app}.
\end{remark}

\subsection{Compatibility with ring structures}
\label{sec:compatibilities:ring}

\begin{proposition}
  \label{prop:cup-products-compatible}
  Let $X$ be a smooth scheme over $\R$, let $V, W \subset X$ be two closed subsets, and let $\mathcal{L}_1,\mathcal{L}_2$ be line bundles over $X$. The real cycle class map with supports as defined in diagram~\eqref{diagram-twist-support} is compatible with the cup product in the sense that the following diagram commutes for all integers $p,q,k,l$:
  \[
    \xymatrix{
      \op{H}^p_V(X,\mathbf{I}^k(\mathcal{L}_1))\otimes \op{H}^q_W(X,\mathbf{I}^l(\mathcal{L}_2)) \ar[r]^-\cup \ar[d] & \op{H}^{p+q}_{V\cap W}(X,\mathbf{I}^{k+l}(\mathcal{L}_1\otimes\mathcal{L}_2)) \ar[d] \\
      \op{H}^p_{V(\R)}(X(\R),\Z(\mathcal{L}_{1})) \otimes \op{H}^q_{W(\R)}(X(\R),\Z(\mathcal{L}_{2})) \ar[r]_-\cup & \op{H}^{p+q}_{V(\R)\cap W(\R)}(X(\R), \Z(\mathcal{L}_{1}\otimes \mathcal{L}_{2}))
    }
  \]
\end{proposition}

\begin{proof}
  Both cup products here can be obtained using Godement's construction, cf.\ \prettyref{thm:Go-cup-product}. For the cup product in \(\I\)-cohomology,  we use the multiplicative structure of the sheaves $\mathbf{I}^j$ as in the discussion above \prettyref{cor:godement-cup-product}.  Recall that this agrees with Fasel's definition of cup product via Gersten complexes by \prettyref{prop:GoGe-products-compatible}.

  As before, the result is the consequence of the commutativity of diagrams involving the maps \(\op{can}\), $\op{sign}_\infty$ and the sheaf pullback $\rho^\ast$.  We ignore the supports in the following more detailed argument, as these only complicate the notation without changing the argument.

  (1) The map $\op{can}\colon \mathbf{I}^t\to\colim_j\mathbf{I}^j$ is multiplicative in the sense that we have a commutative square as in \prettyref{rem:can-multiplicative}.  This square still commutes after twisting by line bundles. Using the naturality of Godement's cup product in \prettyref{thm:Go-cup-product} with respect to morphisms of sheaves, we thus find that the following square commutes:
  \[
    \xymatrix{
      \op{H}^p(X,\mathbf{I}^k(\mathcal{L}_1))\otimes \op{H}^q(X,\mathbf{I}^l(\mathcal{L}_2)) \ar[r]^-\cup \ar[d]^{\op{can}\otimes\op{can}} & \op{H}^{p+q}(X,\mathbf{I}^{k+l}(\mathcal{L}_1\otimes\mathcal{L}_2)) \ar[d]^{\op{can}} \\
      \op{H}^p(X,\colim_j\mathbf{I}^j(\mathcal{L}_1)) \otimes \op{H}^q(X,\colim_j\mathbf{I}^j(\mathcal{L}_2)) \ar[r]_-\cup & \op{H}^{p+q}(X,\colim_j\mathbf{I}^j(\mathcal{L}_1\otimes \mathcal{L}_2))
    }
  \]

  (2) In the second step, we establish the compatibility of cup product with the identification $\op{sign}_\infty\colon \colim_j\mathbf{I}^j\cong\rho_\ast\Z$. We claim that the following diagram is commutative:
  \[
    \xymatrix{
      \op{H}^p(X,\colim_j\mathbf{I}^j(\mathcal{L}_1))\otimes \op{H}^q(X,\colim_j\mathbf{I}^j(\mathcal{L}_2)) \ar[r]^-\cup \ar[d] & \op{H}^{p+q}(X,\colim_j\mathbf{I}^{j}(\mathcal{L}_1\otimes\mathcal{L}_2)) \ar[d] \\
      \op{H}^p(X,\rho_\ast\Z(\mathcal{L}_{1}))) \otimes \op{H}^q(X,\rho_\ast\Z(\mathcal{L}_{2})) \ar[r]_-\cup & \op{H}^{p+q}(X, \rho_\ast\Z(\mathcal{L}_{1}\otimes \mathcal{L}_{2}))
    }
  \]
  The vertical maps are the identifications $\op{sign}_\infty$, and the horizontal maps are cup products. Again, the naturality of Godement's cup product with respect to morphisms of sheaves implies the commutativity of the diagram once we verify that the identification $\op{sign}_\infty$ preserves the multiplicative structures on both sides. This follows because the signature map $\op{sign}\colon \op{W}(X)\to \rho_\ast\Z$ is a ring homomorphism.

  (3) Finally, we need to check the compatibility of the following diagram, encoding the compatibility of the sheaf pullback with the cup products:
  \[
    \xymatrix{
      \op{H}^p(X,\rho_\ast\Z(\mathcal{L}_{1}))) \otimes \op{H}^q(X,\rho_\ast\Z(\mathcal{L}_{2})) \ar[r]_-\cup \ar[d]^{\rho^*\otimes\rho^*} & \op{H}^{p+q}(X,\rho_\ast\Z(\mathcal{L}_{1}\otimes \mathcal{L}_{2})) \ar[d]^{\rho^*} \\
      \op{H}^p(X(\R),\Z(\mathcal{L}_{1})) \otimes \op{H}^q(X(\R),\Z(\mathcal{L}_{2})) \ar[r]_-\cup & \op{H}^{p+q}(X(\R), \Z(\mathcal{L}_{1}\otimes \mathcal{L}_{2}))
    }
  \]
  The commutativity of this diagram follows from the naturality of the cup product with respect to continuous maps, cf.\ \prettyref{prop:Go-product-natural}.
\end{proof}

\begin{corollary}
  \label{cor:realization-is-ring-homo}
  For any smooth scheme $X$ over $\R$, the real cycle class map
  $$
  \bigoplus_{s,t}\op{H}^s(X,\I^t) \to \bigoplus_{s} \op{H}^s(X(\R),\Z)
  $$
  is a morphism of graded rings.
\end{corollary}

\subsection{Compatibility with pushforwards and the localization sequence}
\label{sec:comp:push}

In this section, we will now establish the compatibility of the real cycle class map with supports with pushforwards. We restrict to the case of pushforwards along closed immersions which is enough for our intended applications. This will also imply the compatibility of the real cycle class maps with the localization sequence.

\begin{remark}
  We want to point out that the real cycle class maps should actually be compatible with arbitrary proper pushforwards; tracing through the Gersten-type definition of pushforwards for \(\I\)-cohomology, the key point is the compatibility of the identification $\colim \mathbf{I}^m\cong\rho_\ast\Z$ with transfers as established in \cite[Lemma~20]{bachmann}.
\end{remark}

\begin{remark}
  Another thing to point out before we get started is that there are differences in the treatment of pushforward maps in algebra and topology, which could be a potential source for confusion. On the algebraic side, for any point $i\colon\Spec(\R) \to \P^1$, there is a canonical pushforward $i_*\colon  \op{H}^0(\Spec(\R),\mathbf{I}^0(\omega_{\R}))
  \to \op{H}^1(\P^1,\mathbf{I}^1(\omega_{\P^1}))$. On the topological side, pushforwards generally don't appear as frequently, and there is no canonical pushforward $i_*\colon \op{H}^0(\op{pt},\Z) \to
  \op{H}^1(\R \P^1, \Z)$ in topology. This is not a contradiction:
  real realization maps untwisted coefficients $\mathcal{O}$ to untwisted
  local coefficients $\Z$ canonically. But the isomorphism $\omega_{\P^1}
  \cong \mathcal{O}$ of line bundles
  on $\P^1$ involves a choice,
  and this choice corresponds to
  the choice of an orientation
  on $\R \P^1$, hence of a topological fundamental class,
  and the latter uniquely determines the topological pushforward $i_*$. This means that both on the algebraic and topological side, the pushforward maps depend on choices of orientations, but these are treated differently in algebra and topology.
\end{remark}

The main result to be established in this subsection is the following compatibility of real cycle class maps with pushforwards:

\begin{theorem}
  \label{thm:pushforwards-compatible}
  Let $X$ be a smooth real variety and let $Z\subset X$ be a smooth closed subvariety of codimension $c$. We denote the normal bundle of the inclusion $i\colon Z \to X$ by $\Nb=\Nb_Z X$ and set $\Lb:=\op{det}\Nb_Z X$. Then the following diagram is commutative:
  \begin{adjustwidth}{-10cm}{-10cm}
    \begin{equation}
      \begin{aligned}
        \xymatrix@C=13pt{
          \op{H}^{n-c}(Z,\mathbf{I}^{j-c}(\Lb)) \ar[r] \ar[d]^{i_*}
          &
          \op{H}^{n-c}(Z,\rho_\ast\Z(\Lb)) \ar[r]^-{\rho^*}_-\cong &  \op{H}^{n-c}(Z(\R),\Z(\Lb)) \ar@{=}[r] \ar[d]^{i_*^{\op{Th}}}& \op{H}^{n-c}(Z(\R),\Z(\Lb)) \ar@{..>}[d]^{i_*^{\op{PD}}} \\
          \op{H}^n(X,\mathbf{I}^j) \ar[r]
          &  \op{H}^n(X,\rho_\ast\Z) \ar[r]^-{\rho^*}_-\cong &  \op{H}^n(X(\R),\Z) \ar@{=}[r] & \op{H}^n(X(\R),\Z)
        }
      \end{aligned}
    \end{equation}
  \end{adjustwidth}
  The horizontal compositions are the cycle class maps, the left vertical morphism is the pushforward defined on Gersten complexes as in \prettyref{def:pushforwardWittgroups}, the vertical morphism $i_\ast^{\op{Th}}$ is the pushforward defined in terms of the topological Thom class, cf.\ \prettyref{def:topological-thom-class} and the discussion at the end of \prettyref{sec:singcohomology}. The vertical morphism $i_\ast^{\op{PD}}$ is the Poincar\'e dual of the pushforward in singular Borel--Moore homology, which is defined only if both $X(\R)$ and $Z(\R)$ are oriented and hence
  $\Z(\Lb)=\Z$ is constant.
\end{theorem}

\begin{proof}
  The commutativity of the left rectangle is the core of the result, and will be proved in \prettyref{prop:comparedefconetubnbhd} below. The basic idea is to show that the algebraic Thom class induces generators of the cohomology with support of the fibers of the normal bundle and use the characterization of the topological Thom class to deduce that the algebraic Thom class maps to a topological Thom class. To get a comparison of the pushforward maps, we need a discussion of the relation between deformation to the normal bundle and the topological tubular neighborhood construction, which is the main technical point in \prettyref{prop:comparedefconetubnbhd}.

  Concerning the right hand side square, the morphism $i_*^{\op{Th}}$ is defined using a Thom class in $\op{H}^n_{Z(\R)}(X(\R),\Z)$ combined with the tubular neighborhood theorem, compare e.\ g.\ \cite[section 15.6]{tomdieck},  \cite[Chapter~4]{hirsch} and \cite[VIII.2.4]{iversen}. On the other hand $i_*^{\op{PD}}$ is defined using Poincar\'e duality for both $Z$ and $X$ and $i_*$ in Borel--Moore homology. The commutativity of the right hand side square for oriented manifolds  follows from \cite[Corollary 10.33]{daviskirk} or \cite[IX.7.5]{iversen}.
  For the latter, note that we have a Thom isomorphism $i_*: \op{H}^{*-c}(Z(\R)) \to \op{H}^*_{Z(\R)}(X(\R))$ given by the
(cup) product with the Iversen Thom class $\tau_Z$, see VIII.2.3 of loc. cit..
We have another Thom isomorphism $s_*\colon \op{H}^{*-c}(Z(\R)) \to \op{H}^*_{Z(\R)}(\Nb_{Z(\R)}X(\R))$ given by
the (cup) product with the usual Thom class $th(\Nb(\R))$ as in
Definition \ref{def:topological-thom-class}. There is an isomorphism $d$ between the two targets (see below for possible constructions of it). As in degree $*=c$ -- where both Thom classes
live -- all three groups are free abelian of rank one (we may assume $Z(\R)$ connected), the isos $i_*$,
$s_*$ and $d$ map generators to generators. Hence $d$ maps $\tau_Z$ to
$th(\Nb)$, possibly up to sign if a preferred sign was
choosen. Now the claim follows from
equality IX.7.5 of \cite{iversen}
(one cap there has to be a cup).
\end{proof}

\begin{proposition}
  \label{prop:proto-localization}
  Let $i\colon Z\hookrightarrow X$ be a closed immersion between smooth $\mathbb{R}$-schemes with open complement $U:=X\setminus Z$. Then there is a ladder of long exact sequences defining cohomology with supports
  \[
    \xymatrix@C=1em{
      \cdots \ar[r] & \op{H}^{s}_Z(X,\mathbf{I}^{t}(i^\ast\mathcal{L})) \ar[r] \ar[d] & \op{H}^s(X,\mathbf{I}^t(\mathcal{L})) \ar[r] \ar[d] & \op{H}^s(U,\mathbf{I}^t(\mathcal{L})) \ar[r] \ar[d] & \cdots \\
      \cdots \ar[r] & \op{H}^{s}_{Z(\R)}(X(\R),\Z(i^\ast\mathcal{L})) \ar[r] & \op{H}^s(X(\R),\Z(\mathcal{L})) \ar[r] & \op{H}^s(U(\R),\Z(\mathcal{L})) \ar[r] & \cdots
    }
  \]
\end{proposition}

\begin{proof}
The upper long exact sequence is the localization sequence of \cite[Th\'eor{\`e}me 9.3.4]{fasel:memoir}, which coincides by \prettyref{prop:GoGe-supports-compatible} with the localization sequence from \prettyref{prop:localization-sequence}. The morphisms $\op{can}\colon\mathbf{I}^t\to\colim_j\mathbf{I}^j$ and $\sign_\infty\colon \colim_j\mathbf{I}^j\cong \rho_\ast\mathbb{Z}$ induce morphisms of the Gersten resolutions. From the \AKTONLY{(absent) }proof of \prettyref{prop:localization-sequence}, it is then clear that these morphisms induce ladders of appropriate localization sequences because the localization sequence arises from the short exact sequence of complexes arising from the Gersten resolutions by applying $\Gamma_Z$ or $\Gamma j^\ast$:
\[
\xymatrix{
0\ar[r] & C^\bullet(X,\mathbf{I}^t(\mathcal{L}))_Z \ar[r] \ar[d] & C^\bullet(X,\mathbf{I}^t(\mathcal{L})) \ar[r] \ar[d] & C^\bullet(U,\mathbf{I}^t(\mathcal{L})) \ar[r] \ar[d] & 0\\
0\ar[r] & C^\bullet(X,\rho_\ast\mathbb{Z}(\mathcal{L}))_Z \ar[r] & C^\bullet(X,\rho_\ast\mathbb{Z}(\mathcal{L})) \ar[r] & C^\bullet(U,\rho_\ast\mathbb{Z}(\mathcal{L})) \ar[r] & 0
}
\]
Now we want to compare the lower exact sequence to a short exact sequence arising, as in \ARXIVONLY{the proof of }\prettyref{prop:localization-sequence}, from the Godement resolution of $\mathbb{Z}(\mathcal{L})$ on $X$. There is a morphism of short exact sequence of complexes, from the Gersten-complex sequence to the Godement-complex sequence; it is induced by the zig-zag of quasi-isomorphisms
\[
  C^\bullet(X,\rho_\ast\mathbb{Z}(\mathcal{L}))\leftarrow \mathbb{G}\rho_\ast\mathbb{Z}(\mathcal{L}) \to \rho_\ast\mathbb{G}\mathbb{Z}(\mathcal{L}).
  \qedhere
\]
\end{proof}

\begin{remark}
  In other words, the localization sequence is compatible with morphisms of coefficient data.
\end{remark}

\begin{proposition}
  \label{prop:thom-class-comparison-local}
  There is a commutative diagram
  \[
    \xymatrix{
      \op{H}^0(\Spec \R,\mathbf{I}^0) \ar[r] \ar[d] & \op{H}^0(\op{pt}, \Z) \ar[d] \\
      \op{H}^n_{\{0\}}(\mathbb{A}^n,\mathbf{I}^n) \ar[r] & \op{H}^n_{\{0\}}(\R^n,\Z)
    }
  \]
  with maps given as follows. The left vertical map is the algebraic pushforward in \(\I\)-cohomology of \prettyref{def:pushforwardWittgroups}, the right vertical map is the topological pushforward for sheaf cohomology with $\Z$-coefficients, and the  horizontal maps are the respective real cycle class maps, where we use the cycle class map with support from \prettyref{sec:twist-support} for the lower morphism.
  The orientations are chosen as in \prettyref{prop:computation-cycle-class}.
\end{proposition}

\begin{proof}
  We have $\op{H}^0(\Spec \R,\mathbf{I}^0)\cong \op{W}(\R)\cong\Z$. The algebraic pushforward
  \[
    \op{H}^0(\Spec \R,\mathbf{I}^0)\to  \op{H}^n_{\{0\}}(\mathbb{A}^n,\mathbf{I}^n)
  \]
  is an isomorphism by \prettyref{thm:I-devissage} and maps the form $\langle 1\rangle$ to the algebraic Thom class which consequently is a generator of $\op{H}^n_{\{0\}}(\mathbb{A}^n,\mathbf{I}^n)$. The real realization
  \[
  \op{H}^0(\Spec \R,\mathbf{I}^0) \to \op{H}^0(\op{pt},\Z)
  \]
  maps the form $\langle 1\rangle$ to its signature, which is a generator of $\op{H}^0(\op{pt},\Z)$. On the right-hand side, we have the topological Thom isomorphism $\op{H}^0(\op{pt},\Z) \cong \op{H}^n_{\{0\}}(\R^n,\Z)$. The lower horizontal map is also an isomorphism: we have a commutative diagram of exact sequences
  \[
    \xymatrix{
      \op{H}^{n-1}(\mathbb{A}^n,\mathbf{I}^n) \ar[r] \ar[d] & \op{H}^{n-1}(\mathbb{A}^n\setminus\{0\},\mathbf{I}^n) \ar[r] \ar[d] & \op{H}^n_{\{0\}}(\mathbb{A}^n,\mathbf{I}^n) \ar[r] \ar[d] & \op{H}^n(\mathbb{A}^n,\mathbf{I}^n) \ar[d] \\  \op{H}^{n-1}(\mathbb{R}^n,\Z) \ar[r] & \op{H}^{n-1}(\mathbb{R}^n\setminus\{0\},\Z) \ar[r] & \op{H}^n_{\{0\}}(\mathbb{R}^n,\Z) \ar[r] & \op{H}^n(\mathbb{R}^n,\Z)
    }
  \]
  where the vertical arrows are the real cycle class maps (with support) and the rows are the exact sequences defining cohomology with supports. The ladder above is commutative by \prettyref{prop:proto-localization}.
  Using the five lemma, we are reduced (if $n \neq 0$, otherwise everything follows trivially) to show that the real cycle class map $\op{H}^{n-1}(\mathbb{A}^n\setminus \{0\},\I^n) \to \op{H}^{n-1}(\R^n \setminus \{0\},\Z)$ is an isomorphism, which is established in \prettyref{prop:computation-cycle-class}. More precisely, \prettyref{prop:computation-cycle-class} actually shows that the real cycle class map takes the algebraic Thom class to the topological Thom class (provided we use the canonical choices of orientations induced from the standard coordinates on $\mathbb{A}^n$). This proves the commutativity of the diagram as claimed.
\end{proof}

\begin{proposition}
  \label{prop:thom-class-comparison}
  Let $X$ be a smooth real scheme and let $p\colon V\to X$ be a rank $d$ vector bundle on $X$. Then the algebraic Thom class $\op{th}^{\op{alg}}(V)\in \op{H}^d_X(V,\mathbf{I}^d(p^*\det V^\vee))$ of \prettyref{def:algebraic-thom-class} maps to a topological Thom class $\op{th}^{\op{top}}(V(\R)) \in  \op{H}^d_{X(\R)}(V(\R),\Z(p^*\det V^\vee))$ as in \prettyref{def:topological-thom-class}. In other words, there is a commutative diagram
  \[
    \xymatrix{
      \op{H}^0(X,\mathbf{I}^0) \ar[r] \ar[d] & \op{H}^0(X(\R), \Z) \ar[d] \\
      \op{H}^d_{X}(V,\mathbf{I}^d(p^*\det V^\vee)) \ar[r] & \op{H}^d_{X(\R)}(V(\R),\Z(p^*\det V^\vee))
    }
  \]
  where the horizontal maps are the real cycle class maps of \prettyref{sec:cyclesheaf} and the vertical maps are the respective pushforward maps, or equivalently, the Thom isomorphisms.
\end{proposition}
\begin{proof}
  Denoting the zero section of $V$ by $s$, we have a natural isomorphism
  \[
    s^\ast\colon \op{H}^0(V,\mathbf{I}^0)\to \op{H}^0(X,\mathbf{I}^0),
  \]
  making the top left group a $\op{H}^0(V,\mathbf{I}^0)$-module; the bottom left group is a $\op{H}^0(V,\mathbf{I}^0)$-module via the module structure for cohomology with supports induced by the cup product of \prettyref{thm:Go-cup-product} (which agrees with the algebraically defined intersection product by \prettyref{prop:GoGe-products-compatible}). The left-hand vertical morphism is then an $\op{H}^0(V,\mathbf{I}^0)$-linear map, hence the left-hand vertical map is completely determined by the image of $\langle 1\rangle$ (which by \prettyref{def:algebraic-thom-class} is the algebraic Thom class). A similar statement holds for the right-hand vertical map by definition, see section \ref{sec:singcohomology}. Since the real cycle class map is compatible with the product by \prettyref{prop:cup-products-compatible}, the upper horizontal morphism is a ring homomorphism. This induces an $\op{H}^0(V,\mathbf{I}^0)$-module structure on $\op{H}^d_{X(\R)}(V(\R),\Z({\det V})$ and, by \prettyref{prop:cup-products-compatible}, the lower horizontal cycle class map is an $\op{H}^0(V,\mathbf{I}^0)$-module homomorphism.  As a consequence, the lower horizontal map is completely determined by the image of the algebraic Thom class $\op{th}^{\op{alg}}(V)$ which is a generator of the cyclic $\op{H}^0(V,\mathbf{I}^0)$-module $\op{H}^d_X(V,\mathbf{I}^d(\det V))$ by \prettyref{def:algebraic-thom-class}. Consequently, the commutativity of the diagram follows from the claim on Thom classes.

  To prove the claim about the Thom classes, we use the  characterization of the topological Thom class from \prettyref{def:topological-thom-class}. That is, a class in $\op{H}^d_{X(\R)}(V(\R),\Z({\det V}))$ is a Thom class for the vector bundle $V(\R)\to X(\R)$ if and only for every point $x\in X(\R)$ the restriction of the class to $\op{H}^d_{\{x\}}(V_x(\R),\Z)\cong \Z$ is a generator. Consider the following diagram
  \[
    \xymatrix{
      \op{H}^d_X(V,\mathbf{I}^d(p^*\det V^\vee)) \ar[r] \ar[d]_{i_x^\ast} & \op{H}^d_{X(\R)}(V(\R),\Z(p^*\det V^\vee) \ar[d]^{i_x(\R)^\ast} \\
      \op{H}^d_{\{x\}}(V_x,\mathbf{I}^d) \ar[r] & \op{H}^d_{\{x\}}(V_x(\R),\Z)
    }
  \]
  where again the horizontal morphisms are the respective real cycle class maps and the vertical morphisms are restrictions to the fiber of the respective vector bundle. In particular, the diagram is commutative by \prettyref{prop:pullbacks-compatible} and it suffices to show that the algebraic Thom class maps to a generator of the lower-right group.

  Under the left vertical morphism, the algebraic Thom class $\op{th}^{\op{alg}}(V)$ restricts to the algebraic Thom class of the trivial vector bundle on the point $x$, in particular to a generator of the group $\op{H}^d_{\{x\}}(V_{x},\mathbf{I}^d)$. This follows from the obvious transversal and cartesian diagram of smooth $\R$-schemes

  \[
    \xymatrix{
      \Spec \R \ar[r]_{{x}} \ar[r] \ar[d] & X  \ar[d]^{s} \\
      \mathbb{A}^d \cong V_x \ar[r] & V
    }
  \]
  and the base change theorem \cite[Theorem~2.12]{asok-fasel:euler}. By \prettyref{prop:thom-class-comparison-local}, the local algebraic Thom class in $\op{H}^d_{\{x\}}(V_{x},\mathbf{I}^d)$ maps to a generator of $\op{H}^d_{\{x\}}(V_x(\R),\Z)$, as required. We have proved that the image of the algebraic Thom class in $\op{H}^d_{X(\R)}(V(\R),\Z(p^*\det V^\vee))$ restricts to a generator of the local cohomology group for every fiber. Therefore the algebraic Thom class for $V\to X$ maps to a topological Thom class for $V(\R)\to X(\R)$ as claimed.
\end{proof}

Before the final step in the proof of compatibility of pushfowards, we need to relate the deformation to the normal bundle with tubular neighborhoods.  Recall our notation for the deformation to the nomal cone from \prettyref{fig:deformation-to-the-normal-bundle} in \prettyref{sec:I-pullbacks}, and the construction of the deformation isomorphism $d(Z,X)\colon \op{H}^p_Z(\Nb_ZX,\mathbf{I}^q)\to \op{H}^p_Z(X,\mathbf{I}^q)$ from \prettyref{sec:thomclass}.
The next proposition shows that this isomorphism is compatible under the real cycle class map with a similar isomorphism arising from a tubular neighborhood argument in topology.

\begin{proposition}
  \label{prop:comparedefconetubnbhd}
  Let $i\colon Z\hookrightarrow X$ be a closed immersion of codimension $c$ between smooth $\R$-schemes. Then there is a commutative diagram
  \[
    \xymatrix{
      \op{H}^i(Z, \mathbf{I}^j(\det\Nb_ZX))
      \ar[r]
      \ar[d]
      &
      \op{H}^i(Z(\R), \Z(\det\Nb_{Z(\R)}X(\R)))
      \ar[d]
      \\
      \op{H}^{i+c}(X,\mathbf{I}^{j+c})
      \ar[r]
      &
      \op{H}^{i+c}(X(\R),\Z)
    }
  \]
  where the horizontal morphisms are real cycle class maps and the vertical morphisms are the respective pushforward maps, defined using deformation to the normal bundle and a tubular neighborhood, respectively.
\end{proposition}

\begin{proof}
  Consider the following diagram
  \[
    \xymatrix{
      \op{H}^i(Z, \mathbf{I}^j(\det\Nb_ZX))
      \ar[d]_{\op{th}^{\op{alg}}(\Nb_ZX)}
      \ar[r]
      &
      \op{H}^i(Z(\R), \Z(\det\Nb_{Z(\R)}X(\R)))
      \ar[d]^{\op{th}^{\op{top}}(\Nb_{Z(R)}X(R))}
      \\
      \op{H}^{i+c}_Z(\Nb_ZX, \mathbf{I}^{j+c})
      \ar[d]_{d(Z,X)}^-\cong
      \ar[r]
      &
      \op{H}^{i+c}_{Z(\R)}(\Nb_{Z(\R)}X(\R), \Z)
      \ar[d]^{d(Z,X)(\R)}
      \\
      \op{H}^{i+c}_Z(X, \mathbf{I}^{j+c})
      \ar[r]
      \ar[d]
      &
      \op{H}^{i+c}_{Z(\R)}(X(\R), \Z)
      \ar[d]
      \\
      \op{H}^{i+c}(X,\mathbf{I}^{j+c})
      \ar[r]
      &
      \op{H}^{i+c}(X(\R),\Z)
    }
  \]
  By \prettyref{def:pushforwardWittgroups}, the left vertical composition is the pushforward for \(\I\)-cohomology. The maps in the first square are given by multiplication with the respective Thom classes, hence the square is commutative by \prettyref{prop:thom-class-comparison} and
  \prettyref{prop:cup-products-compatible}. The last square is commutative since real realization with supports commutes with forgetting the support,
  see Remark \ref{rem:realization-support}.

  We recalled the construction of $d(Z,X)$ before the statement of \prettyref{prop:comparedefconetubnbhd}. We now discuss the map $d(Z,X)(\R)$ on the topological side. The topological variants of the  proofs in \cite{nenashev} and \cite[Lemma~2.8]{fasel:excess} show that $i_0(\R)^\ast$ and $i_1(\R)^\ast$ are isomorphisms in singular cohomology with supports and appropriate local coefficient systems. Note that classical excision is sufficient because the real realization of a Nisnevich cover is a surjective local diffeomorphism, and hence can be refined by an honest open cover. As we define the map $d(Z,X)(\R)=i_1(\R)^\ast\circ (i_0(\R)^\ast)^{-1}$, the commutativity of the middle square follows from the compatibility of real realization and pullback morphisms, cf.\ \prettyref{prop:pullbacks-compatible}.

  Finally, it remains to prove that the right vertical composition in our big diagram is the pushforward in singular cohomology. For that, we consider the real realization of the diagram for the deformation to the normal bundle just before the statement of \prettyref{prop:comparedefconetubnbhd}. The normal bundle $\Nb_{Z(\R)\times\mathbb{R}^1}(D(Z,X)(\mathbb{R}))$ is extended from $Z(\R)$ by homotopy invariance. Moreover, since its restriction to $Z(\R)\times\{t\}$ (for any $t\in\R$) is the normal bundle of $i(\R)$ we know that $\Nb_{Z(\R)\times\mathbb{R}^1}(D(Z,X)(\mathbb{R}))$ is the pullback of $\Nb_{Z(\R)}(X(\R))$ along the projection $Z(\R)\times\mathbb{R}\to Z(\R)$. We now choose a tubular neighborhood of $Z(\R)$ in $\Nb_{Z(\R)}(X(\R))$. By definition (see e.g.\ \cite[Chapter~4.5]{hirsch}), this is a smooth map $T\colon\Nb_{Z(\R)}(X(\mathbb{R})) \to \Nb_{Z(\R)}(X(\mathbb{R}))$ such that $T\circ s=s$ which restricts to a diffeomorphism from an open neighborhood of $Z(\R)$ in $\Nb_{Z(\R)\times\mathbb{R}^1}(D(Z,X)(\mathbb{R}))$  to an open neighborhood of $Z(\R)$ in $\Nb_{Z(\R)}(X(\R))$. In our case, we may simply choose $T=\mathrm{id}$ and choose the open disc bundle $\mathcal{D}$ as open neighborhood. Since the above conditions are stable under pullback along $Z \times \R^1 \to Z$, up to the above trivializing isomorphism for $\Nb_{Z(\R)\times\mathbb{R}^1}(D(Z,X)(\mathbb{R}))$ above, the fiber product $\mathcal{T} \cong \mathcal{D} \times \R^1$ and its restriction $\mathcal{T}|_1$ over $Z \times \{1\}$ are tubular neighborhoods as well.

  To summarize, we obtain the following diagram
  \[
    \xymatrix{
      Z(\R) \ar[r] \ar[d]_s & Z(\R)\times \mathbb{R}^1 \ar[d] & Z(\R) \ar[l] \ar[d]^f \\
      \mathcal{D} \ar[r]_{e_0} \ar[d] & \mathcal{T}  \ar[d] & \mathcal{T}|_1 \ar[d]  \ar[l]^{e_1}
      \\
      \Nb_{Z(\R)} X(\R) \ar[r]_{i_0(\R)} & D(Z,X)(\R) & X(\R) \ar[l]^{i_1(\R)}
    }
  \]
  where the first and third row are the real points of the diagram for the deformation to the normal bundle  and the second row contains the above tubular neighborhood $\mathcal{T}$ and its restrictions to the fibers over $0$ and $1$, respectively. The diagram is obviously commutative. By excision, the lower vertical maps induce isomorphisms in cohomology with supports
  \begin{eqnarray*}
    \op{H}^i_{Z(\R)}(\Nb_{Z(\R)}X(\R),\Z)&\xrightarrow{\cong}& \op{H}^i_{Z(\R)}(\mathcal{D},\Z), \\
    \op{H}^i_{Z(\R)\times \R^1}(D(Z,X)(\R),\Z) &\xrightarrow{\cong}& \op{H}^i_{Z(\R)\times \R^1}(\mathcal{T},\Z),
                                                                             \textrm{ and } \\
    \op{H}^i_{Z(\R)}(X(\R),\Z) &\xrightarrow{\cong}& \op{H}^i_{Z(\R)}(\mathcal{T}|_1,\Z).
  \end{eqnarray*}
  Consequently, the morphisms $e_0^\ast$ and $e_1^\ast$ of the appropriate cohomology groups with supports are also isomorphisms. Therefore, we can safely replace the composition $d(Z,X)(\R)=i_1^\ast\circ (i_0^\ast)^{-1}$ by the composition $e_1^\ast\circ (e_0^\ast)^{-1}$. Now we recall that the normal bundle $\Nb_{Z(\R)\times\mathbb{R}^1}(D(Z,X)(\mathbb{R}))$ is extended from $Z$ and we can choose a global isomorphism to $\Nb_{Z(\R)}(X(\R))\times\R$. In particular, there is a natural projection morphism $\mathcal{T}\to \mathcal{T}|_1$ fitting into a commutative diagram
  \[
    \xymatrix{
      Z(\R)\times\R \ar[r]^{\op{pr}} \ar[d] & Z(\R) \ar[d] \\
      \mathcal{T} \ar[r]_{\op{pr}'} & \mathcal{T}|_1
    }
  \]
  such that $\op{pr}' \circ e_0 = \op{id}$ and hence $(\op{pr}')^\ast\colon \op{H}^i_{Z(\R)}(\mathcal{T}|_1,\Z)\to \op{H}^i_{Z(\R)\times \R^1}(\mathcal{T},\Z), $ is the inverse of $e_1^\ast$.  Using these maps, we can rewrite the upper two squares of the above diagram involving the tubular neighborhoods to
  \[
    \xymatrix{
      Z(\R) \ar[r]^{\op{id}} \ar[d] & Z(\R) \ar[d] \\
      \mathcal{D} \ar[r] & \mathcal{T}|_1 \subseteq X(\R)
    }
  \]
  where the lower horizontal morphism $\op{pr}'$ is the inclusion of the disc bundle $\mathcal{D}$ in $\Nb_{Z(\R)}(X(\R))$ as the tubular neighborhood $\mathcal{T}|_1$ of $i(Z(\R))$. Moreover, the pullback $e_0^\ast\circ(\op{pr}')^\ast$ is the inverse of the composition $e_1^\ast\circ(e_0^\ast)^{-1}$. Consequently, the identification
  \[
    \op{H}^i_{Z(\R)}(\Nb_{Z(\R)}(X(\R)),\Z) \cong \op{H}^i_{Z(\R)}(X(\R),\Z)
  \]
  coming from the deformation to the normal bundle is inverse to the identification
  \[
    \op{H}^i_{Z(\R)}(X(\R),\Z) \to \op{H}^i_{Z(\R)}(\mathcal{D},\Z)\cong \op{H}^i_{Z(\R)}(\Nb_{Z(\R)}(X(\R)),\Z)
  \]
  coming from the tubular neighborhood theorem.
\end{proof}

The following compatibility of real cycle class maps with the localization sequences is now a direct consequence of \prettyref{thm:pushforwards-compatible} and \prettyref{prop:proto-localization}.

\begin{corollary}
  \label{cor:compat-localization}
  Let $i\colon Z\hookrightarrow X$ be a closed immersion between smooth $\mathbb{R}$-schemes of codimension $c$ and let $\mathcal{L}$ be a line bundle on $X$. Denote by $U:=X\setminus Z$ the complement of $Z$ and by $\omega$ the determinant of the normal bundle of the inclusion $i$. Then there is a ladder of long exact localization sequences
  \[
    \xymatrix@C=1em{
      \cdots \ar[r] & \op{H}^{s-c}(Z,\mathbf{I}^{t-c}(i^\ast\mathcal{L}\otimes \omega)) \ar[r] \ar[d] & \op{H}^s(X,\mathbf{I}^t(\mathcal{L})) \ar[r] \ar[d] & \op{H}^s(U,\mathbf{I}^t(\mathcal{L})) \ar[r] \ar[d] & \cdots \\
      \cdots \ar[r] & \op{H}^{s-c}(Z(\R),\Z(i^\ast\mathcal{L}\otimes \omega)) \ar[r] & \op{H}^s(X(\R),\Z(\mathcal{L})) \ar[r] & \op{H}^s(U(\R),\Z(\mathcal{L})) \ar[r] & \cdots
    }
  \]
\end{corollary}

\subsection{Compatibility with the Bär sequence}\label{sec:compatiblity:Baer}
The short exact sequence of sheaves $0\to\I^{n+1}(\mathcal{L})\to \I^n(\mathcal{L})\to \bar{\I}^n\to 0$ gives rise to a long exact sequence of cohomology groups, the so-called B\"ar sequence.
The boundary map of this sequence is usually denoted \(\beta_{\sheaf L}\) and referred to as Bockstein homomorphism.  We briefly discuss the compatibility of the real cycle class map with this homomorphism. By \cite[Theorem~8.5(4)]{jacobson}, there is an isomorphism of short exact sequences
\[
  \xymatrix{
    0 \ar[r] & \colim_n\mathbf{I}^{n+1} \ar[r]^\eta \ar[d]_\cong &  \colim_n\mathbf{I}^n \ar[r]^\rho \ar[d]_\cong & \colim_n\overline{\mathbf{I}}^n\ar[r] \ar[d]^\cong & 0\\
    0\ar[r] & \rho_\ast \Z \ar[r]_2 & \rho_\ast \Z\ar[r] & \rho_\ast \Z/2\Z\ar[r] & 0
  }
\]
By the discussion in \prettyref{sec:cyclesheaf}, the horizontal arrows in the above diagram are in fact isomorphisms of sheaves on the big site ${\op{Sm}}_F$ and equivariant for the $\mathbb{G}_{\op{m}}$-actions on $\op{colim}\mathbf{I}^n$ and $\rho_\ast\Z$. In particular, for a smooth scheme $X$ and a line bundle $\mathcal{L}$ on $X$, we can immediately obtain an isomorphism of short exact sequences of sheaves on the small site of $X$, twisted by $\mathcal{L}$ as in \prettyref{def:general-twist}:
\[
  \xymatrix{
    0 \ar[r] & \colim_n\mathbf{I}^{n+1}(\sheaf L) \ar[r]^\eta \ar[d]_\cong &  \colim_n\mathbf{I}^n(\sheaf L) \ar[r]^\rho \ar[d]_\cong & \colim_n\overline{\mathbf{I}}^n\ar[r] \ar[d]^\cong & 0\\
    0\ar[r] & \rho_\ast\Z(\sheaf L) \ar[r]_2 & \rho_\ast\Z(\sheaf L)\ar[r] &  \Z/2\Z\ar[r] & 0.
  }
\]
It should be noted here that the $\mathbb{G}_{\op{m}}$-action on $\overline{\mathbf{I}}^n$ and $\rho_\ast\Z/2\Z$ is trivial so that twisting has no effect.

\begin{proposition}
  \label{prop:compatbockstein}
  Let $X$ be a smooth scheme and let $\mathcal{L}$ be a line bundle on $X$. Then there is a commutative diagram
  \[
    \xymatrix{
      \op{Ch}^n(X) \ar[r]^{\beta_{\sheaf L}} \ar[d] & \op{H}^{n+1}(X,\mathbf{I}^{n+1}(\sheaf L)) \ar[d] \\
      \op{H}^n(X(\R),\Z/2\Z) \ar[r]_{\beta_{\sheaf L}} & \op{H}^{n+1}(X(\R),\Z(\sheaf L))
    }
  \]
  In particular, the algebraic and topological Bockstein maps correspond under the real cycle class maps.
\end{proposition}

\begin{proof}
  The key point is the isomorphism of exact sequences discussed above, resulting from \cite[Theorem~8.5(4)]{jacobson}. There is another commutative ladder of short exact sequences
  \[
    \xymatrix{
      0 \ar[r] & \mathbf{I}^{n+1}(\mathcal{L}) \ar[r]^\eta \ar[d] & \mathbf{I}^n(\mathcal{L}) \ar[r]  \ar[d] & \overline{\mathbf{I}}^n \ar[r] \ar[d] & 0\\
      0 \ar[r] & \colim_j \mathbf{I}^{j+1}(\mathcal{L}) \ar[r]^\eta  & \colim_j\mathbf{I}^j(\mathcal{L}) \ar[r] & \colim_j\overline{\mathbf{I}}^j \ar[r] & 0
    }
  \]
  which arise by twisting from the obvious commutative diagram involving the canonical maps $\mathbf{I}^n\to \colim_j\mathbf{I}^j$. Combining these commutative ladders of exact sequences implies that we have a commutative square for the boundary morphisms
  \[
    \xymatrix{
      \op{H}^n(X,\overline{\mathbf{I}}^n)\cong \op{Ch}^n(X) \ar[r]^{\beta_{\sheaf L}} \ar[d] & \op{H}^{n+1}(X,\mathbf{I}^{n+1}(\sheaf L)) \ar[d] \\
      \op{H}^n(X,\rho_\ast \Z/2\Z) \ar[r]_{\beta_{\sheaf L}} & \op{H}^{n+1}(X,\rho_\ast \Z(\sheaf L))
    }
  \]

  The claim follows by combining this with the commutative diagram
  \[
    \xymatrix{
      \op{H}^n(X,\rho_\ast \Z/2\Z) \ar[r]_{\beta_{\sheaf L}} \ar[d] & \op{H}^{n+1}(X,\rho_\ast \Z(\sheaf L)) \ar[d] \\
      \op{H}^n(X(\R),\Z/2\Z) \ar[r]_{\beta_{\sheaf L}} & \op{H}^{n+1}(X(\R),\Z(\sheaf L))
    }
  \]
  where the vertical arrows are pullbacks along $\rho$. The lower horizontal arrow is the boundary map for the short exact sequence of sheaves
  \[
    0\to \Z(\mathcal{L}(\R)) \to \Z(\mathcal{L}(\R)) \to \Z/2\Z \to 0
  \]
  and the upper horizontal arrow is the boundary map the pushforward of the above sequence along $\rho$. So the commutativity of the final square above is a consequence of the compatibility of pullbacks in sheaf cohomology with long exact sequences of sheaves, cf.~\prettyref{prop:pullback-exact-seq}.
\end{proof}

\section{Cellular varieties}\label{sec:cellular}
\begin{definition}[{\cite[Example~1.9.1]{fultonbook}}]
  A variety \(X\) over a base \(S\) is \textbf{cellular} if it has a filtration over \(S\) by closed subvarieties
  \[
    X = Z_0 \supset Z_1 \supset \cdots \supset Z_N = \emptyset
  \]
  such that, for each \(k\), we have an isomorphism of \(S\)-varieties \(Z_{k-1}\setminus Z_k \cong \A_S^{n_k}\) for some integer \(n_k\).
\end{definition}

\begin{remark}\label{rem:cellular-open-filtration}
  For a smooth cellular variety, it is sometimes more useful to consider the complementary filtration by open subvarieties, cf.\ \cite[\S\,2.1]{wendt:m-cells} or \cite[p.\,485]{zib:WCCV}.  This allows us to ``piece the variety together without leaving the smooth world''.
\end{remark}

\begin{proposition}\label{prop:cellular-cycle-isos}
  For a smooth cellular variety \(X\) over \(\R\), the complex, mod 2 complex and mod 2 real cycle class maps define ring isomorphisms:
  \begin{align*}
    \op{CH}^n(X) & \xrightarrow{\cong} \op{H}^{2n}(X(\C))\\
    \op{Ch}^n(X) & \xrightarrow{\cong} \op{h}^{2n}(X(\C))\\
    \op{Ch}^n(X) & \xrightarrow{\cong} \op{h}^n(X(\R))
  \end{align*}
\end{proposition}

\begin{proof}
  The first two maps are ring homomorphisms by \cite[Corollary 19.2(b)]{fultonbook}, see also \cite[Theorem~4.1]{totaro:mu}. The statement for the (integral) complex cycle class map is well-known, see for example \cite[Example~19.1.11(b)]{fultonbook}.  That it also defines an isomorphism \(\op{Ch}^n(X)\cong \op{h}^{2n}(X(\C))\) follows immediately since \(\op{Ch}^n(X)=\op{CH}^n(X)/2\) for any \(X\) and \(\op{h}^{2n}(Y) = \op{H}^{2n}(Y)/2\) for any topological space \(Y\) whose cohomology is concentrated in even degrees.

   For the real cycle class maps, compatibility with the ring structure is established in \prettyref{cor:realization-is-ring-homo}. That it is an isomorphism can be proved similarly as in the complex case, by cellular induction.  Indeed, consider the main induction step (cf.\ \prettyref{rem:cellular-open-filtration}): we assume that \(X\) is the union of a regularly embedded codimension-\(c\) cell \(Z\) and its open complement \(U\), we know that the real cycle class map is an isomorphism for \(Z\). We need to compare the following two exact sequences:
  \[\xymatrix{
      \op{Ch}^{n-c}(Z) \ar[r] \ar[d]^{\cong}
      & \op{Ch}^n(X) \ar[r] \ar[d]
      & \op{Ch}^n(U) \ar[r] \ar[d]
      & 0 
      \\
      \op{h}^{n-c}(Z(\R)) \ar[r]
      & \op{h}^n(X(\R)) \ar[r]
      & \op{h}^n(U(\R)) 
    }\]
  The square on the left is commutative by \prettyref{thm:pushforwards-compatible}, and the second square commutes by \prettyref{prop:pullbacks-compatible}.
  The corresponding diagram for the complex cycle class map has (top and) lower row a short exact sequence, because \(\op{H}^*(X(\C))\) is concentrated in even degrees only. In contrast, \(\op{h}^*(X(\R))\) can be non-zero in any degree.  Nonetheless, assuming by induction that the real cycle class map is surjective for \(U\), we find that the restriction map \(\op{h}^n(X(\R)) \to \op{h}^n(U(\R))\) is also surjective in all degrees \(n\), and hence that the pushforward \(\op{h}^{n-c}(Z(\R))\to \op{h}^n(X(\R))\) is injective for all~\(n\).  Thus, assuming by induction that the real cycle class map is an isomorphism for \(U\), we may deduce inductively that the real cycle class map for \(X\) is also an isomorphism.
\end{proof}

\begin{proposition}
  \label{prop:mult2-barI}
  For a smooth cellular variety \(X\)  over \(\R\), multiplication by \(\pfist{-1}\in \op{I}(\R)\) induces isomorphisms
  \[
    \op{H}^i(X,\bar{\I}^j) \xrightarrow{\pfist{-1}} \op{H}^i(X,\bar{\I}^{j+1})
    \quad  \text{ for all } \quad j\geq i.
  \]
\end{proposition}

Note that \(\op{H}^i(X,\bar{\I}^j) = 0\) for all \(j<i\), for any smooth \(X\) over any field.  This follows directly from the definition\slash computation of \(\bar{\I}^j\)-cohomology in terms of (quotients) of Gersten complexes.
The proposition says that for real cellular \(X\), the non-zero values of \(\bar{\I}^j\)-cohomology are independent of~\(j\).

\begin{proof}
  We have \(\op{W}(\R)=\Z\) with \(\pfist{-1} = 2\) and \(\op{I}^j(\R) = 2^j\Z\).  So the conclusion holds for \(X=\Spec(\R)\). We now use homotopy invariance, the localization sequence and  d{\'e}vissage for \(\bar\I^j\)-cohomology, which even hold for $\I^j$-cohomology, see section \ref{sec:I-recollections}.

  We also use that each of these isomorphisms\slash exact sequences is compatible with the \(\op{W}(\R)\)-module structure on \(\bar{\I}^j\)-cohomology. First, by homotopy invariance, the claim follows for a ``cell'', i.e.,\ for \(\A^n\), for any \(n\).  By \prettyref{rem:cellular-open-filtration}, it now suffices to consider the situation that \(X\) is the union of a regularly embedded codimension-\(c\) cell \(Z\) and its open complement \(U\), and to assume that the hypothesis already holds for \(U\).  To deduce that the claim also holds for \(X\), we compare the localization sequences for \(\bar{\I}^j\) and \(\bar{\I}^{j+1}\)-cohomology:
  \[
    \xymatrix@C=8pt{
      \op{H}^{i-1}(U,\bar{\I}^j) \ar[r]^-{\partial} \ar[d]^{\cong}
      & \op{H}^{i-c}(Z,\bar{\I}^{j-c}) \ar[r] \ar[d]^{\cong}
      & \op{H}^i(X,\bar{\I}^j) \ar[r] \ar[d]^{?}
      & \op{H}^i(U,\bar{\I}^j) \ar[r]^-{\partial} \ar[d]^{\cong}
      & \op{H}^{i+1-c}(Z,\bar{\I}^{j-c}) \ar@{>->}[d]
      \\
      \op{H}^{i-1}(U,\bar{\I}^{j+1}) \ar[r]^-{\partial}
      & \op{H}^{i-c}(Z,\bar{\I}^{j+1-c}) \ar[r]
      & \op{H}^i(X,\bar{\I}^{j+1}) \ar[r]
      & \op{H}^i(U,\bar{\I}^{j+1}) \ar[r]^-{\partial}
      & \op{H}^{i+1-c}(Z,\bar{\I}^{j+1-c})
    }
  \]
  The diagram commutes since the localization sequences are induced from short exact sequences of $\op{W}(F)$-modules, and all vertical arrows are given by multiplication by \(\pfist{-1}\). The arrow we are interested in is the central arrow. The first and fourth arrow are isomorphisms by induction hypothesis.  The second arrow is an isomorphism since \(Z\) is a cell. The fifth arrow is either an isomorphism (in case \(j\geq i+1\)) or the inclusion of \(0\) (in case \(j = i\)), so in any case it is injective.  This shows that the central arrow is an isomorphism, as desired.
\end{proof}

\begin{proposition}
  \label{prop:mult2-I}
  Suppose \(X\) is a smooth variety over \(\R\) for which the conclusion of \prettyref{prop:mult2-barI} holds, i.e.,\ such that multiplication by \(\pfist{-1}\) induces isomorphisms
  \[
    \op{H}^i(X,\bar{\I}^j) \cong \op{H}^i(X,\bar{\I}^{j+1})
  \]
  for all \(j\geq i\).
  Then for any line bundle \(\sheaf L\) over \(X\) multiplication by \(\pfist{-1}\) also induces isomorphisms
  \[
    \op{H}^i(X,{\I}^j(\sheaf L)) \xrightarrow{\pfist{-1}} \op{H}^i(X,{\I}^{j+1}(\sheaf L))
    \quad  \text{ for all } \quad j\geq i.
  \]
\end{proposition}
\begin{remark}\label{rem:j-dichotomy}
The inclusions \(\I^j(\sheaf L)\hookrightarrow \W(\sheaf L)\) induce isomorphisms \( \op{H}^i(X,\I^j(\sheaf L))\cong \op{H}^i(X,\W(\sheaf L)) \) for all \(j<i\).  This follows easily from the Bär sequence. So for \(X\) and \(\sheaf L\) as in \prettyref{prop:mult2-I} and a fixed cohomological degree~\(i\), the groups \(\op{H}^i(X,\I^j(\sheaf L))\) take only two different values according to whether \(j\geq i\) or \(j<i\).
\end{remark}
\begin{proof}
  As Jacobson explains in the proof of Corollary~8.11 in \cite{jacobson}, there is an integer \(d\) such that multiplication by \(\pfist{-1}\) induces isomorphisms of sheaves \(\I^j\cong \I^{j+1}\) for all \(j\geq d+1\).  In fact, we can take \(d\) to be exactly the Krull dimension of \(X\), but we do not need to know this.  Twisting with \(\sheaf L\), we see that \(\pfist{-1}\) also induces isomorphisms \(\I^j(\sheaf L)\cong \I^{j+1}(\sheaf L)\) in this range.  Thus, the conclusion of the proposition holds for \(j\geq d+1\). We now run a downward induction over~\(j\).  Suppose the conclusion holds for some particular value~\(j\). We now compare the B{\"ar} sequences
  for \(\I^{j-1}(\sheaf L)\) and \(\I^{j}(\sheaf L)\) around some cohomological degree \(i\leq j-1\):
  \[\xymatrix@C=5pt{
      \op{H}^{i-1}(X,\bar{\I}^{j-1}) \ar[r]^-{\partial} \ar[d]^{\cong}
      & \op{H}^i(X,{\I}^{j}(\sheaf L)) \ar[r] \ar[d]^{\cong}
      & \op{H}^i(X,{\I}^{j-1}(\sheaf L)) \ar[r] \ar[d]^{?}
      & \op{H}^i(X,\bar{\I}^{j-1}) \ar[r]^-{\partial} \ar[d]^{\cong}
      & \op{H}^{i+1}(X,{\I}^{j}(\sheaf L)) \ar[d]^{\cong}
      \\
      \op{H}^{i-1}(X,\bar{\I}^{j}) \ar[r]^-{\partial}
      & \op{H}^i(X,{\I}^{j+1}(\sheaf L)) \ar[r]
      & \op{H}^i(X,{\I}^{j}(\sheaf L)) \ar[r]
      & \op{H}^i(X,\bar{\I}^{j}) \ar[r]^-{\partial}
      & \op{H}^{i+1}(X,{\I}^{j+1}(\sheaf L))
    }\]
  Again the diagram commutes because the long exact sequences are induced from a short sequence of sheaves of $\op{W}(F)$-modules and all vertical arrows are multiplication by \(\pfist{-1}\). We need to show that the central arrow is an isomorphism.  The first and fourth arrows are isomorphisms by our assumption on~\(X\). The second and fifth arrow are isomorphisms by our induction hypothesis.  So the claim follows.
\end{proof}

\begin{theorem}\label{thm:cellular-iso}
  Let \(X\) be a smooth cellular variety over \(\R\), and let \(\sheaf L\) be a line bundle over \(X\).  Consider the real cycle class maps
  \[
    \op{H}^i(X,\I^j(\sheaf L)) \to \op{H}^i(X(\R),\Z(\sheaf L)).
  \]
  \begin{enumerate}[(a)]
  \item
    For \(j\geq i\), the real cycle class map is a group isomorphism.  In particular, in diagonal bidegrees the homomorphism of graded commutative rings considered in \prettyref{cor:realization-is-ring-homo} is an isomorphism  \(
  \bigoplus_i \op{H}^i(X,\I^i) \cong \bigoplus_i \op{H}^i(X(\R),\Z)
  \).
  \item
    For \(j < i\), the image of the real cycle class map is the subgroup \(2^{i-j}\cdot \op{H}^i(X(\R),\Z(\sheaf L)) \subset \op{H}^i(X(\R),\Z(\sheaf L))\).  In case \(j=i-1\), the real cycle class map is moreover injective.  Thus, by composing the isomorphisms of  \prettyref{rem:j-dichotomy} with the real cycle class map we obtain an isomorphism
    \(\op{H}^i(X,\I^j(\sheaf L)) \cong 2\cdot \op{H}^i(X(\R),\Z(\sheaf L))\) for all \(j < i\).
  \end{enumerate}
\end{theorem}
\begin{proof}
  (a) Consider first the case \(j\geq i\). It suffices to show that \(\sign\colon \op{H}^i(X,\I^j(\sheaf L))\to \op{H}^i(X(\R),2^j\Z(\sheaf L))\) is an isomorphism.  Consider the following commutative square of morphisms of coefficient data (cf.\ \cite[Definition~8.4]{jacobson}):
  \[
    \xymatrix@C=5em{
      (X,\I^j)\ar@{<-}[r]^-{(\rho,\sign)}
      &(X(\R),2^j\Z)
      \\
      (X,\I^{j+1}) \ar[u]^{(\id,\pfist{-1})} \ar@{<-}[r]^-{(\rho,\sign)}
      &(X(\R),2^{j+1}\Z) \ar[u]^{(\id,2)}_{\cong}
    }
  \]
  By \prettyref{cor:twist-key}, this square remains commutative if we twist by \(\sheaf L\). We thus obtain the following commutative square in cohomology:
  \[\xymatrix{
      \op{H}^i(X,\I^j(\sheaf L)) \ar[r] \ar[d]_{\pfist{-1}}
      & \op{H}^i(X(\R),2^j\Z(\sheaf L)) \ar[d]_{2}
      \\
      \op{H}^i(X,\I^{j+1}(\sheaf L)) \ar[r]^-{\cong}
      & \op{H}^i(X(\R),2^{j+1}\Z(\sheaf L))
    }\]
  By \prettyref{prop:mult2-I}, the left vertical morphism is an isomorphism for all \(j\geq i\).  One of Jacobson's main results in \cite{jacobson} is that the horizontal arrows are isomorphisms for sufficiently large~\(j\) and trivial twists, see once again (the proof of) Corollary 8.11 in loc.\ cit.
  The proof given there generalizes to arbitrary twists, cf.\ \prettyref{thm:sign}. So the proof can be completed by backward induction over~\(j\).  The statement concerning the ring structure is immediate from \prettyref{cor:realization-is-ring-homo}.

  (b) The claim concerning \(j<i\) follows from the isomorphism for \(j\geq i\) and the compatibility of the Bär sequence with the Bockstein sequence (\prettyref{prop:compatbockstein}).
\end{proof}

\begin{proposition}
  \label{prop:bw-iso}
  For any smooth cellular variety $X$ over  $\mathbb{R}$, the restricted equivariant cycle class map $\op{CH}^\bullet(X)\to \op{H}^\bullet_{\op{C}_2}(X(\mathbb{C}),\Z)_{(0)}$ of \cite{benoist:wittenberg} is an isomorphism.
\end{proposition}

\begin{proof}
  For $X=\op{Spec}\R$, we obtain by \cite[1.2.2]{benoist:wittenberg}
  \[
    \op{H}^{2k}_{\op{C}_2}(X(\mathbb{C}),\Z(k))=\left\{
      \begin{array}{ll}
        0 & k \textrm{ odd}\\
        \Z/2\Z & k>0 \textrm{ even}
      \end{array}\right.
  \]
  and the map \cite[(1.57)]{benoist:wittenberg} has the form
  \[
    \op{H}^{2k}_{\op{C}_2}(X(\mathbb{C}),\Z(k)) \to \bigoplus_{0\leq p\leq 2k, p\equiv 0\bmod 2} \op{H}^p(X(\R),\Z/2\Z)
  \]
  where the direct sum in the target has its only nontrivial term for $p=0$. In particular, for a class $\alpha\in \op{H}^{2k}_{\op{C}_2}(X(\mathbb{C}),\Z(k))$ the only potentially nontrivial component is $\alpha_0$, in the notation of \cite[1.6.3]{benoist:wittenberg}. Moreover, since $\op{H}^{2k}_{\op{C}_2}(X(\mathbb{C}),\Z(k)) \to \op{H}^0(X(\R),\Z/2\Z)$ is an isomorphism, $\alpha_0=0$ if and only if $\alpha=0$. This implies that $\op{H}^{2k}_{\op{C}_2}(X(\mathbb{C}),\Z(k))_0$ as defined in \cite[Definition~1.17]{benoist:wittenberg} is trivial: the generator $\alpha$ of $\op{H}^{2k}_{\op{C}_2}(X(\mathbb{C}),\Z(k))$ has the only non-trivial component $\alpha_0$ but then the condition $\alpha_0=\op{Sq}^{-2k}(\alpha_k)$ cannot be satisfied. In particular, for $X=\op{Spec}\R$, the equivariant cycle class map $\op{CH}^\bullet(X)\to \op{H}^{2\bullet}_{\op{C}_2}(X(\mathbb{C}),\Z(\bullet))_{(0)}$ of \cite{benoist:wittenberg} is an isomorphism. By homotopy invariance, this also implies that it is an isomorphism for $X=\mathbb{A}^n_{\R}$.

  We now want to argue by induction on the number of cells of $X$. As before, we have a diagram
  \[\xymatrix{
      \op{CH}^{n-c}(Z) \ar[r] \ar[d]^{\cong}
      & \op{CH}^n(X) \ar[r] \ar[d]
      & \op{CH}^n(U) \ar[r] \ar[d]
      & 0
      \\
      \op{H}^{2n-2c}_{\op{C}_2}(Z(\mathbb{C}),\Z(n-c))_0 \ar[r]
      & \op{H}^{2n}_{\op{C}_2}(X(\mathbb{C}),\Z(n))_0 \ar[r]
      & \op{H}^{2n}_{\op{C}_2}(U(\mathbb{C}), \Z(n))_0
    }\]
  The top row is the localization sequence for Chow groups. The vertical arrows are the restricted cycle class maps of \cite{benoist:wittenberg} and therefore commutativity of the diagram follows by the compatibility statements in 1.6.1 and 1.6.4 of \cite{benoist:wittenberg}. However, we don't know exactness of the lower row in general. In our special case where $Z$ is an affine space, we can argue as follows: we have the localization sequence for equivariant cohomology
  \[
    \op{H}^{2n-2c}_{\op{C}_2}(Z(\mathbb{C}),\Z(n-c)) \to \op{H}^{2n}_{\op{C}_2}(X(\mathbb{C}),\Z(n))\to  \op{H}^{2n}_{\op{C}_2}(U(\mathbb{C}), \Z(n))
  \]
  which is exact. The problem with exactness for the restricted cohomology is that a restricted class on $X(\mathbb{C})$ which vanishes on $U(\mathbb{C})$ may only come from an equivariant cohomology class on $Z(\mathbb{C})$ which may not be restricted. However, if $Z$ is an affine space, we know what the non-restricted classes look like: essentially, they are 2-torsion classes from the cohomology of $\op{C}_2$. By \cite[Theorem~1.20]{benoist:wittenberg}, the Steenrod square condition is still going to be violated for the pushforward of such a 2-torsion class. This implies the exactness in the middle of the lower row.

  Now we can use the inductive assumption that the right-hand vertical map is an isomorphism. This implies that the restriction morphism $\op{H}^{2n}_{\op{C}_2}(X(\mathbb{C}),\Z(n))_0\to  \op{H}^{2n}_{\op{C}_2}(U(\mathbb{C}), \Z(n))_0$ is surjective. The injectivity of the pushforward morphism $\Z\cong \op{H}^{0}_{\op{C}_2}(Z(\mathbb{C}),\Z)_0 \to \op{H}^{2c}_{\op{C}_2}(X(\mathbb{C}),\Z(c))_0$ also follows directly from non-triviality which can be checked on mod 2 reduction. The diagram chase then implies that the middle vertical morphism must be an isomorphism.
\end{proof}

\begin{corollary}
  For any smooth cellular variety $X$ over $\R$, the cycle class map in Corollary~\ref{CWcycleclassmap} is an isomorphism.
\end{corollary}

\begin{proof}
  By \prettyref{prop:krasnov-vs-borel-haefliger}, the isomorphism $\op{CH}^n(X)\xrightarrow{\cong} \op{H}^{2n}_{\op{C}_2}(X(\mathbb{C}), \Z(n))_0$ of \prettyref{prop:bw-iso} induces an isomorphism between the kernel of $\partial\colon \op{CH}^n(X)\to \op{Ch}^n(X)\xrightarrow{\beta} \op{H}^{n+1}(X,\mathbf{I}^{n+1})$ and the kernel of the composition
  \[
    \op{H}^{2n}_{\op{C}_2}(X(\mathbb{C}),\Z(n))_0\to \op{H}^n(X(\mathbb{R}),\Z/2\Z)\xrightarrow{\beta} \op{H}^{n+1}(X(\mathbb{R}),\Z).
  \]
  Using the compatibility of \prettyref{prop:jacobson-vs-borel-haefliger}, the above isomorphism combined with the isomorphisms $\op{H}^n(X,\mathbf{I}^n)\xrightarrow{\cong}  \op{H}^n(X(\R),\Z)$ of \prettyref{thm:cellular-iso} imply that the map in Corollary~\ref{CWcycleclassmap} is an isomorphism because $\op{CH}^n(X)$ is 2-torsion-free for cellular $X$.
\end{proof}
\begin{remark}\label{rem:base-fields-for-cellular-applications}
  We could choose base fields \(F\) slightly more general than \(F=\R\).  For \prettyref{prop:mult2-barI}, we just need that the conclusion holds for \(\Spec(F)\), i.e.,\ that
  \(
  \pfist{-1}\colon \bar{\op{I}}^j(F) \to \bar{\op{I}}^{j+1}(F)
  \)
  is an isomorphism for all \(j\geq 0\).

  For \prettyref{prop:mult2-I},
  it should similarly suffice to assume that
  \(
  \pfist{-1}\colon \op{I}^j(F) \to \op{I}^{j+1}(F)
  \)
  is an isomorphism for all \(j\geq 0\). By \prettyref{lem:field-condition} below, this (stronger) condition implies that \(\vcd_2(F) = 0\).
  The latter condition ensures that, in the proof of \prettyref{prop:mult2-I}, multiplication by \(\pfist{-1}\) still induces isomorphisms \(\I^j\cong \I^{j+1}\) for sufficiently large~\(j\) (as in the proof of \cite[Theorem~4.5]{jacobson-pre2}).

  \prettyref{thm:cellular-iso} also follows for smooth cellular \(X\) over any field \(F\) satisfying the conditions of \prettyref{lem:field-condition}, provided we replace singular by real cohomology.
\end{remark}

\begin{lemma}\label{lem:field-condition}
  For a field \(F\) of characteristic \(\neq 2\), the following conditions are equivalent (as always, \(\op{I}^0(F) = \op{W}(F)\) by convention):
  \begin{enumerate}
  \item \(\vcd_2(F) = 0\) and \(F\) is not quadratically closed
  \item \(\pfist{-1}\colon \op{I}^j(F)\to \op{I}^{j+1}(F)\) is an isomorphism for all \(j\geq 0\)
  \end{enumerate}

\end{lemma}
\begin{proof}
  In the proof, we identify \(\pfist{-1}\) with \(2\in \op{W}(F)\). As the torsion in the Witt ring of a field is \(2\)-primary \cite[\S\,31]{elman-karpenko-merkurjev}, \(2\colon \op{I}^j(F)\to \op{I}^{j+1}(F)\) is injective if and only if \(\op{I}^j(F)\) is torsion-free.
  By \cite[Lem.~2.4]{jacobson-pre2}, \(\vcd_2(F)=0\) if and only if \(\op{I}(F)\) is torsion-free and \(2\colon \op{I}^j(F)\to \op{I}^{j+1}(F)\) is surjective for all \(j\geq 0\).
  The field \(F\) is quadratically closed if and only if \(\op{I}(F)\neq 0\) \cite[Lem.~31.3]{elman-karpenko-merkurjev}.  Thus, under either condition, \(2\colon \op{W}(F)\to \op{I}(F)\) is surjective, \(2\colon \op{I}^j(F) \to \op{I}^{j+1}(F)\) is an isomorphism for all \(j\geq 1\), and \(\op{I}(F)\) is torsion-free.   We need to show that, in this situation, \(\op{I}(F)\) is non-zero if and only if \(2\colon \op{W}(F)\to \op{I}(F)\) is injective.  This can be checked using the following commutative diagram:
  \[\xymatrix{
      0 \ar[r] & \op{I}(F) \ar[r]\ar[d]_{\cong}^{2} & \op{W}(F) \ar[r]\ar@{->>}[d]^{2} & \Z/2 \ar[r]\ar@{->>}[d]^{2} & 0\\
      0 \ar[r] & \op{I}^2(F) \ar[r] & \op{I}(F) \ar[r] & \op{I}(F)/\op{I}^2(F) \ar[r] & 0
    }\]
  In our situation, the first vertical arrow is an isomorphism, the second and (hence) the third vertical arrows are surjective, and \(\op{I}(F)/\op{I}^2(F)=\op{I}(F)/2\op{I}(F)\).  Given that \(\op{I}(F)\) is torsion-free, \(\op{I}(F)\) is non-zero if and only if \(\op{I}(F)/2\op{I}(F)\) is non-zero, which is equivalent to the third vertical map being injective, which is equivalent to the central vertical map being injective, as claimed.
\end{proof}

\section{Characteristic classes}
\label{sec:characteristic}

In this section, we discuss the compatibility of the real cycle class maps with characteristic classes of vector bundles. For spaces like ${\op{B}}\op{GL}_n$ or the finite Grassmannians this will imply a strengthening of \prettyref{thm:cellular-iso}: not only is the real cycle class map a ring isomorphism, it also maps the natural generators (characteristic classes of the tautological bundles) in the algebraic setting of \(\I\)-cohomology to their topological counterparts. This also settles a point which was only mentioned but not fully proved in \cite{hornbostelwendt} and \cite{real-grassmannian}.

\begin{proposition}
  \label{prop:compateuler}
  Let $X$ be  a smooth scheme over $\R$ and let $\mathcal{E}\to X$ be a vector bundle of rank $n$ over $X$. Then the real cycle class map
  \[
    \op{H}^n(X,\mathbf{I}^n(p^*\det\mathcal{E}^\vee))\to \op{H}^n(X(\R),\Z(p^*\det\mathcal{E}^\vee))
  \]
  maps the algebraic Euler class to the topological Euler class.
\end{proposition}

\begin{proof}
  The algebraic Euler class is defined as $\op{e}_n(\mathcal{E}):=(p^\ast)^{-1}s_\ast\langle 1\rangle$ where $p\colon \mathcal{E}\to X$ is the structure morphism of the vector bundle, $s\colon X\to \mathcal{E}$ is the zero section and $\langle 1\rangle$ is the image of the natural element of $\op{W}(F)$  under the morphism $\op{W}(F)\cong\op{H}^0(\op{Spec}\R,\mathbf{I}^0)\to \op{H}^0(X,\mathbf{I}^0)$ induced from the structure morphism $X\to\op{Spec}\R$.

  The topological definition is the same, cf.\ e.g.\ the end of the proof of Theorem~4D.10 in \cite{hatcher} in the case of orientable vector bundles. The claim then follows from the compatibility of real cycle class maps with pullbacks and pushforwards, cf.\ \prettyref{prop:pullbacks-compatible} and \prettyref{thm:pushforwards-compatible}.
\end{proof}

\begin{remark}
  A similar statement is true for the characteristic classes in the real-\'etale cohomology, as well as the characteristic classes in semi-algebraic cohomology over real closed fields. We won't need these statements here.
\end{remark}

\begin{theorem}
  \label{thm:glnret}
  Let $X$ be one of the spaces ${\op{B}}\op{Sp}_{2n}$, ${\op{B}}\op{SL}_n$, ${\op{B}}\op{GL}_n$, or $\op{Gr}(k,n)$ over the base field $\R$.
  The real cycle class map induces an isomorphism
  \[
    \bigoplus_{j,\sheaf L} \op{H}^j(X, \mathbf{I}^j(\sheaf L)) \xrightarrow{\cong}  \bigoplus_{j,\sheaf L} \op{H}^j(X(\R),\Z(\sheaf L)),
  \]
  and this identification maps the algebraic characteristic classes (Euler class, Pontryagin classes, Bockstein classes) to their topological counterparts.
\end{theorem}

\begin{proof}

  Compatibility of the Euler class with the identification with singular cohomology under the real cycle class map is established in \prettyref{prop:compateuler}. Compatibility of the Pontryagin classes with the real cycle class map then follows from this: the Pontryagin (or Borel) classes for the symplectic groups are obtained from the Euler class by stabilization, cf.\ \cite[Proposition~4.3]{hornbostelwendt}. The Pontryagin classes for $\op{SL}_n$ and $\op{GL}_n$ are then simply obtained as characteristic classes of the symplectification. For the Stiefel--Whitney classes, we note that these also are uniquely defined by stabilization and the fact that the top class should be the reduction of the Euler class. Then \prettyref{thm:cellular-iso} (or better its $\overline{\mathbf{I}}^\bullet$-version) implies that the real cycle class maps induce isomorphisms
  \[
    \bigoplus_i \op{Ch}^i(X)\xrightarrow{\cong} \bigoplus_i \op{H}^i(X,\overline{\I}^i)\xrightarrow{\cong} \bigoplus_i \op{H}^i(X(\R),\Z/2\Z)
  \]
  for the spaces $X$ in the statement. The compatibility of Bockstein classes with the real realization maps then follows from this and the compatibility with the Bockstein maps in \prettyref{prop:compatbockstein}.
\end{proof}

\begin{remark}
  As a special case, the ring structure for $\bigoplus_{j,\sheaf L}\op{H}^j(\mathbb{P}^n,\mathbf{I}^j(\sheaf L))$ determined in \cite{fasel:ij}, cf.\ also \cite[Propositions 4.5 and 6.5]{real-grassmannian} completely agrees with the presentation known in topology. The natural generators in both cases are the Euler classes $\op{e}_1$ and $\op{e}_{n}^\perp$ of the tautological line bundle and its complement (and the orientation class whenever $\mathbb{P}^n$ is orientable); and the above result shows that these natural generators correspond under the real realization maps. The same statement is true for the Grassmannians.
\end{remark}

\appendix
\section{A Gersten complex for real cohomology}\label{sec:app}
The purpose of this section is to give an alternative proof of \prettyref{thm:pushforwards-compatible}.
\subsection{The Rost--Schmid complex}
The references for this section are \cite{Schmidt} and \cite{fasel:memoir}. Let $X$ be a smooth integral  scheme of Krull dimension $n$ over a field $k$.  The Rost--Schmid complex $\op {C}^{RS}(X, \W)$ for Witt groups is defined as
\[ \xymatrix @ -1pc { \bigoplus_{x \in X^{(0)}} \op{W}(k(x), \omega_x) \ar[r] &  \bigoplus_{x \in X^{(1)}} \op{W}(k(x), \omega_x)  \ar[r] & \cdots \ar[r]  &  \bigoplus_{x \in X^{(n)}} \op{W}(k(x), \omega_x)   }   \]
where $\omega_x$ denotes for any point $x \in X^{(j)}$ the vector space $\Hom_{\mathcal{O}_{X,x}}(\det \mathfrak{m}_{X,x}/\mathfrak{m}_{X,x}^2, \mathcal{O}_{X,x})$. Here, $\Hom_{\mathcal{O}_{X,x}}(\det \mathfrak{m}_{X,x}/\mathfrak{m}_{X,x}^2, \mathcal{O}_{X,x})$ is canonically isomorphic to   $\textnormal{Ext}^j_{\mathcal{O}_{X,x}} (k(x), \mathcal{O}_{X,x})$.
The differential
\[ \partial^i_{RS}\colon \bigoplus_{x \in X^{(i)}} \op{W}(k(x), \omega_x) \rightarrow \bigoplus_{y \in X^{(i+1)}} \op{W}(k(y), \omega_y)   \]
may be described geometrically by the following composition (see e.g. \cite[section 7.3]{fasel:memoir}).
\[ \op{W}(k(x), \omega_{k(x)/k}) \stackrel{\oplus\partial}\longrightarrow \bigoplus_{ \tilde{y}} \op{W}(k(\tilde{y}), \omega_{k(\tilde{y})/k}) \stackrel{\Sigma\textnormal{tr}}\longrightarrow  \op{W}(k(y), \omega_{k(y)/k})  \]
where the sum is taken over all points $\tilde{y}$ that dominate y and
live inside the normalization $\widetilde{\overline{\{x\}}}$ of $\overline{\{x\}}$
in its residue field $k(x)$. Note that $\mathcal{O}_{\widetilde{\overline{\{x\}}}, \tilde{y}}$ is the integral closure of the one dimensional local domain $\mathcal{O}_{\overline{\{x\}}, y} $ in its field of fractions which is isomorphic to $k(x)$. Therefore,  $\mathcal{O}_{\widetilde{\overline{\{x\}}}, \tilde{y}}$ is a discrete valuation ring with maximal ideal $\mathfrak{m}_{\tilde{x},\tilde{y}}$ and the uniformizing parameter $\pi$.
The first map $\partial$ is induced by the Rost--Schmid residue map
\[ \delta\colon \op{W}(k(x)) = \op{W}(k(\tilde{x})) \rightarrow \op{W}(k(\tilde{y}), (\mathfrak{m}_{\widetilde{\overline{\{x\}}}, \tilde{y}}/\mathfrak{m}^2_{\widetilde{\overline{\{x\}}}, \tilde{y}})^* )\]
defined by sending
a form $\langle a \rangle$ on $k(x)$ to $ \delta_2^{\pi_{\tilde{y}}}(\langle a \rangle) \otimes \pi^*$
where $\delta_2$ is Milnor's second residue and $\pi^*$ is the dual basis of the basis $\pi$ in the free rank one $k(\tilde{y})$ vector space  $\mathfrak{m}_{\widetilde{\overline{\{x\}}}, \tilde{y}}/\mathfrak{m}^2_{\widetilde{\overline{\{x\}}}, \tilde{y}}$. The map $\partial$ is obtained by twisting $\delta$ with the rank one free $k(x)$ vector space $\omega_{k(x)/k}$ in view of the following canonical isomorphism
\[ \omega_{k(\tilde{y})/k}\cong \omega_{k(\tilde{x})/k} \otimes  ( \mathfrak{m}_{\widetilde{\overline{\{x\}}}, \tilde{y}}/\mathfrak{m}^2_{\widetilde{\overline{\{x\}}}, \tilde{y}})^*  \] Note that the differential $\partial$ does not depend on the choice of uniformizing parameter. The map $\textnormal{tr}$ is the twisted transfer on Witt groups, for the field extension $k(\tilde{y})$ of $k(y)$, cf.\ \cite[6.4]{fasel:memoir}.

For any line bundle $\mathcal{L}$ on $X$, we have a twisted Rost-Schmid complex $\op{C}^{RS}(X, \W,\mathcal{L})$ for Witt groups
\[ \xymatrix @ -1pc { \bigoplus_{x \in X^{(0)}} \op{W}(k(x), \omega_x^\mathcal{L}) \ar[r] &  \bigoplus_{x \in X^{(1)}} \op{W}(k(x), \omega_x^{\mathcal{L}})  \ar[r] & \cdots \ar[r]  &  \bigoplus_{x \in X^{(n)}} \op{W}(k(x), \omega_x^\mathcal{L})   }   \]  defined by a similar natural above.  It is known that the Rost--Schmid complex coincides with the Gersten complex defined by Balmer--Walter (reviewed in Section \ref{sec:construction-Gersten}) \cite[Section 7]{fasel:memoir}.

\subsection{The Jacobson--Scheiderer complex and a generalization}
Let $X$ be a smooth integral  scheme of Krull dimension $d$ over a field $k$ with a line bundle $\mathcal{L}$. Let $x \in X^{(i)}, y \in X^{(i+1)}$. Our  goal in this section is to define a twisted differential
\[\partial_{re}\colon  \op{C}(k(x)_r, \Z({\omega_{k(x)/k}})) \rightarrow \op{C}(k(y)_r, \Z({\omega_{k(y)/k}}))\]
cf. Section \ref{sec:tr} and Definition \ref{def:realdiff} below. This will be defined by a twisted residue map followed by a twisted transfer map on the real spectrum.
Here and in all that follows, we use the short-hand \(F_r := \Sper F\) for a field \(F\).
\subsubsection{Twisted residue} \label{sec:tr}
Let $R$ be a d.v.r with field of fractions $F$ and let $\mathfrak{m}$ be its maximal ideal. Let $k_\mathfrak{m}$ be it is residue field. Choose a uniformizing parameter $\pi$ of $R$. Suppose $P$ is an ordering on $F$. We say $R$ is \textit{convex} on $(F,P)$ whenever for all $x,y,z \in F$, $x \leq_P z \leq_P y$ and  $x,y \in R$ implies $z \in R$. For any ordering $\bar{\xi} \in k_\mathfrak{m}$, the set
\[
Y_{\bar{\xi}}:= \{P \in \Sper F: R \textnormal{ is convex in } (F,P), \textnormal{ and $\bar{\xi} = \bar{P}$ on $k_\mathfrak{m}$}  \}
\]
maps bijectively to $\{\pm 1\}$ by sending $P$ to $\sign_P(\pi)$
\cite[\S\,3]{jacobson}. Denote $\xi_\pm^\pi$ the ordering on $F$ such that $\sign_{\xi_\pm^\pi}(\pi) = \pm1$. Jacobson \cite[\S\,3]{jacobson} defines a residue map (a group homomorphism)
\begin{align*}
  \beta_\pi\colon \op{C}(F_r, \Z) &\rightarrow \op{C}((k_\mathfrak{m})_r, \Z)\\
  s &\mapsto \left(\bar{\xi}\mapsto s(\xi_+^\pi) - s(\xi_-^\pi)\right)
\end{align*}
This map depends on the choice of uniformizing parameter. We modify it slightly in the spirit of Rost--Schmid. For a field $Q$ and a rank one $Q$-module $H$, we define $\op{C}(Q_r, \Z(H))$ as the group
\[
  \op{C}(Q_r, \Z) \otimes_{\Z[Q^\times]} \Z[H^\times]
\]
where the underlying map
\begin{align*}
  \zeta\colon \Z[Q^\times] &\rightarrow \op{C}(Q_r, \Z)
\shortintertext{is the ring homomorphism given by }
                        \left(\textstyle\sum n_a a\right) &\mapsto \left( \xi \mapsto \sum n_a \sign_\xi a\right).
\end{align*}
\begin{definition}
The twisted residue map is the group homomorphism
\begin{align*}
          \beta\colon \op{C}(F_r, \Z) &\rightarrow \op{C}((k_{\mathfrak m})_r, \Z({(\mathfrak{m}/\mathfrak{m}^2)^*}))\\
          s &\mapsto \beta_\pi(s) \otimes \pi^*
\end{align*}
where $\pi^*$ is the dual basis of the basis $\pi$ of the free rank one $k_{\mathfrak{m}}$-vector space  $\mathfrak{m}/\mathfrak{m}^2$.\
\end{definition}

\begin{lemma}
The twisted residue map $\beta$ does not depend on the choice of the uniformizing parameter.
\end{lemma}
\begin{proof}
Let $a \pi$ be another uniformizing parameter with $a \in R - \mathfrak{m}$.
Note that $\beta_{a\pi}(s)(\bar{\xi})  = \sign_P(a)\beta_\pi(s)(\bar{\xi})$ and $\sign_P(a) = \sign_{\bar{\xi}}(a)$.
It follows that
\[
\beta_{a\pi}(s)  \otimes {a\pi}^* = \beta_{a\pi}(s) \zeta(a)  \otimes {\pi}^*
\]
Now, $\beta_{a\pi}(s) \zeta(a)(\bar{\xi})= \sign_{\bar{\xi}}(a)^2\beta_\pi(s)(\bar{\xi})   =\beta_\pi(s)(\bar{\xi}).$
The result follows.
\end{proof}
\subsubsection{Twisted transfer}
Let $L$ be a finite field extension of $F$. Assume that $F$ has an ordering $P$. Hence, $F$ has characteristic zero. We say an ordering $R$ (on $L$) is an extension of $P$ if $P\subset R$ under the injection $F \hookrightarrow L$. Define a map
\begin{align*}
  t\colon \op{C}(L_r, \Z) &\rightarrow \op{C}(F_r, \Z)\\
  \phi &\mapsto \left(P \mapsto \sum_{R\supset P} \phi(R)\right)
\end{align*}
where the sum runs through all extensions $R$ of $P$.  This map is well-defined as the sum is finite. (Note that the number of extensions equals $\sign_P(Tr_*\langle 1 \rangle)$, cf.\ \cite[Chapter~3, Theorem~4.5]{Schar}). This map will be called the transfer map.

Now we twist the transfer. Note that \(\op{C}(L_r,\Z) = \op{C}(L_r,\Z(L))\).
Suppose $k \subset F \subset L$ with $L$ a separable finite extension of $F$. There is a canonical isomorphism (cf.\ \cite[6.4]{fasel:memoir})
\[ \sigma\colon \omega_{F/k} \otimes_F L \cong \omega_{L/k} \]
\begin{definition}
  We define the twisted transfer as
  \[
    t\colon \op{C}(L_r, \Z({\omega_{L/k}})) \rightarrow \op{C}(F_r, \Z({\omega_{F/k}})) : \phi \mapsto t(\phi): P \mapsto \sum_{R\supset P} \sigma\phi(R)
    \]
 which is the composition
    \[
      \op{C}(L_r,\Z({\omega_{L/k}}))
      \xrightarrow[\cong]{\id\otimes\sigma^{-1}}
      \op{C}(L_r,\Z({\omega_{F/k}}))
      \xrightarrow{t\otimes\id}
      \op{C}(F_r,\Z({\omega_{F/k}})).
    \]
\end{definition}

\subsubsection{The differential}

Let $X$ be a smooth integral  scheme of Krull dimension $d$ over a field $k$ with a line bundle $\mathcal{L}$. Let $x \in X^{(i)}, y \in X^{(i+1)}$. Now, we are ready to define the map
\begin{definition}\label{def:realdiff}
  The differential in the Gersten complex of the real spectrum
  \[
    \partial_{re}\colon  \op{C}(k(x)_r, \Z({\omega_{k(x)/k}})) \rightarrow \op{C}(k(y)_r, \Z({\omega_{k(y)/k}}))
  \]
  is defined by
  \[
    \op{C}(k(x)_r, \Z({\omega_{k(x)/k}})) \stackrel{\oplus\beta}\longrightarrow \bigoplus_{ \tilde{y}}   \op{C}(k(\tilde{y})_r, \Z({\omega_{k(\tilde{y})/k}})) \stackrel{\Sigma t}\longrightarrow  \op{C}(k(y)_r, \Z({\omega_{k(y)/k}}))
  \]
  where the sum is taken over all points $\tilde{y}$ that dominate y and
live inside the normalization $\widetilde{\overline{\{x\}}}$ of $\overline{\{x\}}$ in its residue field $k(x)$. The maps $\beta$ and $t$ are defined as above.
\end{definition}

\begin{theorem}\label{thm:comparison}
  For any smooth integral scheme \(X\) over a field $k$, the sequence
  \[
    \xymatrix @ -1pc { \bigoplus_{x \in X^{(0)}} \op{C}(k(x)_r, \Z({\omega^\mathcal{L}_x})) \ar[r] &  \bigoplus_{x \in X^{(1)}} \op{C}(k(x)_r, \Z({\omega^\mathcal{L}_x})) \ar[r] & \cdots \ar[r]  & \bigoplus_{x \in X^{(n)}} \op{C}(k(x)_r, \Z({\omega^\mathcal{L}_x})) \ar[r] & 0   }
  \]
  is a cochain complex. Its cohomology is the sheaf cohomology of \(X\) with coefficients in the Zariski sheaf $\support_* \Z(\mathcal{L})$.
\end{theorem}

We denote this complex by $\op{C}^{re}(X,\Z,\mathcal{L})$. The proof will be given at the end of Section \ref{sec:comparison}.
\subsection{Comparison via the signature}\label{sec:comparison}

\subsubsection{Signature}
Let $H$ be a free rank one vector space over a field $F$. Choose a trivialization $s\colon F\rightarrow H$.
\begin{definition} The twisted signature is the map
  \begin{align*}
    \sign\colon \op{W}(F,H) &\rightarrow \op{C}(F_r, \Z(H))\\
    \phi &\mapsto \sign_s(\phi) \otimes s(1)
  \end{align*}
  where $\sign_s(\phi)(P) = \sign_P(s^{-1} \phi)$.
\end{definition}
\begin{lemma}
The twisted signature does not depend on the choice of trivialization $s\colon F\rightarrow H$.
\end{lemma}
\begin{proof}
  Let $u\colon F\rightarrow H$ be another trivialization. We may assume the rank of $\phi$ is one. Then $s^{-1}\phi = \langle a \rangle$ and $u^{-1} \phi = \langle b \rangle$ for some $a,b \in F^\times$, and we find that
  \[
    \sign_s(\phi) \otimes s(1) = \mathrm{const}_1 \otimes  s(a)
    =\mathrm{const}_1 \otimes \phi(1,1)
    =\mathrm{const}_1 \otimes u(b)  =  \sign_u(\phi) \otimes u(1),
  \]
  where \(\mathrm{const}_1\) denotes the constant map \(\Sper F\to \Z\) with value~\(1\).
\end{proof}
\begin{lemma}\label{lem:trivilization}
  Consider the following diagram, in which the top horizontal morphism is given by composition with \(s\) and the lower horizontal morphism is given by tensoring with \(s(1)\):
  \[
    \xymatrix{
      \op{W}(F)
      \ar[d]^-{\sign}
      \ar[r]^-{s_*}
      &
      \op{W}(F,H)
      \ar[d]^-{\sign}
      \\
      \op{C}(F_r, \Z)
      \ar[r]^-{s_*}
      &
      \op{C}(F_r, \Z(H))
    }
  \]
  This diagram commutes.
\end{lemma}
\subsubsection{Comparison}
Let $R$ be a one dimensional local domain containing a field $k$ and essentially of finite type over \(k\).
Let $F$ be its field of fractions and $k_\mathfrak{m}$ its residue field.
Let \(\bar R\) denote the integral closure of \(R\) in \(F\), let \(\tilde R\) denote the d.v.r.\ obtained by localizing \(\bar R\) at a maximal ideal \(\tilde m\), and let \(k_{\tilde{\mathfrak m}} := \tilde R/\tilde{\mathfrak{m}}\) denote the residue field at \(\tilde m\). As \(R\) is essentially of finite type over a field, \(R\) is excellent, and hence the normalization \(\Spec \bar R\to \Spec R\) is a finite morphism [EGAIV, 2, 7.8.6\,(ii)].  In particular, $\bar R$ contains only finitely many maximal ideals \(\tilde m\), and the residue fields \(k_{\tilde{\mathfrak m}}\) are finite extensions of $k_{\mathfrak{m}} = R/\mathfrak{m}$.

\begin{theorem}\label{thm:signature-compatible-with-differential}
  The following square commutes:
  \[
    \xymatrix{
      \op{W}(F,\omega_{F/k}) \ar[r]^-{\partial_{RS}} \ar[d]^-{\sign} & \op{W}(k_\mathfrak{m}, \omega_{k_\mathfrak{m}/k}) \ar[d]^-{2\sign} \\
      \op{C}(F_r, \Z({\omega_{F/k}}) \ar[r]^-{\partial_{re}}) & \op{C}((k_{\mathfrak m})_r, \Z({\omega_{k_\mathfrak{m}/k}}))
    }
  \]
\end{theorem}
\begin{proof}

  By the discussion above, we can assume the integral closure $\bar R$ of $R$ in $F$ is a finitely generated $R$-module. There are finitely many d.v.r.s $\tilde{R} $ dominating $R$, namely the localizations of $\bar R$ at its maximal ideals. Let $\tilde{\mathfrak{m}}$ denote the maximal ideal of such $\tilde{R}$. It is well-know that $k_{\tilde{\mathfrak{m}}}:=\tilde{R}/\tilde{\mathfrak{m}}$ is a finite field extension of $k_{\mathfrak{m}} = R/\mathfrak{m}$.
  The square
  \begin{equation}\label{eq:Jacobsons-square}
    \begin{aligned}
    \xymatrix{
      \op{W}(F)
      \ar[r]^-{\delta}
      \ar[d]^-{\sign}
      &
      \op{W}(k_{\tilde{\mathfrak{m}}}, (\tilde{\mathfrak{m}}/\tilde{\mathfrak{m}}^2)^* )
      \ar[d]^-{2\sign}
      \\
      \op{C}(F_r, \Z)
      \ar[r]^-{\beta}
      &
      \op{C}((k_{\tilde{\mathfrak m}})_r, \Z({(\tilde{\mathfrak{m}}/\tilde{\mathfrak{m}}^2)^*}) )
    }
    \end{aligned}
  \end{equation}
  is commutative by \cite[Lemmas~3.1 and 3.2]{jacobson} and \prettyref{lem:trivilization}. (Choose an arbitrary uniformizing parameter \(\pi\), and choose the trivialization \(s\colon k_{\mathfrak{m}}\to (\mathfrak{m}/\mathfrak{m}^2)^*\) given by \(1\mapsto \pi^*\).)
 Now, twisting the square above, we conclude that the square
  \[
    \xymatrix{
      \op{W}(k_{\tilde{\mathfrak{m}}}, \omega_{k_{\tilde{\mathfrak{m}}}/k})
      \ar[d]^-{\sign}
      \ar[r]^-{tr}
      &
      \op{W}(k_{\mathfrak{m}}, \omega_{k_{\mathfrak{m}}/k})
      \ar[d]^-{\sign}
      \\
      \op{C}((k_{\tilde{\mathfrak m}})_r, \Z({\omega_{k_{\tilde{\mathfrak{m}}}/k}}))
      \ar[r]^-{t}
      &
      \op{C}((k_{\mathfrak m})_r, \Z({\omega_{k_{\mathfrak{m}}/k}}))
    }
  \]
  also commutes.  Without twists, this is \cite[Chapter~3, Theorem~4.5]{Schar}.  To see that it commutes with arbitrary twists, choose a trivialization \(s\colon k_{\mathfrak{m}}\to \omega_{k_{\mathfrak{m}}/k}\) and define a trivialization \(\tilde s\colon k_{\tilde{\mathfrak{m}}}\to \omega_{k_{\tilde{\mathfrak{m}}}/k}\) as the composition
  \[
    k_{\tilde{\mathfrak{m}}}
    \cong
    k_{\mathfrak{m}} \otimes_{k_{\mathfrak{m}}} k_{\tilde{\mathfrak{m}}}
    \xrightarrow{s\otimes\id}
    \omega_{k_{\mathfrak{m}}/k} \otimes_{k_{\mathfrak{m}}} k_{\tilde{\mathfrak{m}}}
    \xrightarrow{\sigma}
    \omega_{k_{\tilde{\mathfrak{m}}}/k}.
  \]
  Then use \prettyref{lem:trivilization} again.
\end{proof}
\begin{corollary}\label{cor:compatibilityWC} The signature induces a commutative diagram
  \[
    \xymatrix @ -1pc {
      \bigoplus_{x \in X^{(0)}} \op{W}(k(x), \omega_x^\mathcal{L}) \ar[r]^-{\partial_{RS}} \ar[d]^-{\sign}&  \bigoplus_{x \in X^{(1)}} \op{W}(k(x), \omega_x^{\mathcal{L}})  \ar[r]^-{\partial_{RS}} \ar[d]^-{2\sign}& \cdots \ar[r]^-{\partial_{RS}}  &  \bigoplus_{x \in X^{(n)}} \op{W}(k(x), \omega_x^\mathcal{L}) \ar[d]^-{2^n\sign} \ar[r] & 0\\
      \bigoplus_{x \in X^{(0)}} \op{C}(k(x)_r, \Z({\omega^\mathcal{L}_x})) \ar[r]^-{\partial_{re}} &  \bigoplus_{x \in X^{(1)}} \op{C}(k(x)_r, \Z({\omega^\mathcal{L}_x}) \ar[r]^-{\partial_{re}})  & \cdots \ar[r]^-{\partial_{re}}   & \bigoplus_{x \in X^{(n)}} \op{C}(k(x)_r, \Z({\omega^\mathcal{L}_x}))  \ar[r] & 0 }
  \]
\end{corollary}
The upper horizontal sequence is the Rost--Schmid complex, which coincides with the Gersten complex defined by Balmer--Walter \cite[Section 7]{fasel:memoir}. Moreover, the upper sequence restricts from $\W$ to powers
$\I^j$ of the fundamental ideal, and this complex is indeed a flasque
resolution for all $j$ computing the Zariski sheaf cohomology of $\I^j$ by the proof of \cite[Theorem 3.26]{fasel:chowwittring}.

The restriction of the signature to powers of the fundamental ideal \(\op{I}^j(F)\) takes values in the subgroup \(\op{C}(F_r,2^j\Z) \subset \op{C}(F_r,\Z)\). Indeed, \(\op{I}(F)\) is additively generated by Pfister forms \(\pfist{a}\), and $\sign(\pfist{a}) (P) = \sign_P(1) + \sign_P(-a)$ is always even.  So we obtain a map
\[
  \sign\colon \op{I}^j(F) \rightarrow \op{C}(F_r, 2^j\Z).
\]
Jacobson's commutative square \eqref{eq:Jacobsons-square}, displayed as the far side of the following cube, thus restricts to the square displayed as the front side of the cube:
\[
  \xymatrix@R=1em@C=1em{
    &
    \op{W}(F)
    \ar[rr]^(0.6){\delta_\pi}
    \ar'[d][dd]_(0.4){\sign}
    &&
    \op{W}(k_{\mathfrak m})
    \ar[dd]_(0.6){2\sign}
    \\
    \op{I}^{j+1}(F)
    \ar[rr]^(0.7){\delta_\pi}
    \ar[dd]_(0.6){\sign}
    \ar@{^{(}->}[ur]
    &&
    \op{I}^j(k_{\mathfrak m})
    \ar[dd]_(0.7){2\sign}
    \ar@{^{(}->}[ur]
    \\
    &
    \op{C}(F_r,\Z)
    \ar'[r]^(0.7){\beta_\pi}[rr]
    &&
    \op{C}(k_{\mathfrak m},\Z)
    \\
      \op{C}(F_r, 2^{j+1}\Z)
      \ar[rr]^-{\beta_\pi}
          \ar@{^{(}->}[ur]
      &&
      \op{C}((k_{\mathfrak m})_r, 2^{j+1}\Z)
          \ar@{^{(}->}[ur]
  }
\]
Denoting by \(2^{-j}\sign\) the obvious composition
\(
 \op{I}^j(F) \xrightarrow{\sign} \op{C}(F_r, 2^j\Z)\xrightarrow{\cong}\op{C}(F_r, \Z)
\),
we can rewrite the commutative square at the front as:
\[
  \xymatrix{
    \op{I}^{j+1}(F)
    \ar[r]^-{\delta_\pi}
    \ar[d]^-{2^{-(j+1)}\sign}
    &
    \op{I}^j(k_\mathfrak{m})
    \ar[d]^-{2^{-j}\sign}
      \\
      \op{C}(F_r, \Z)
      \ar[r]^-{\beta_\pi}
      & \op{C}((k_{\mathfrak m})_r, \Z)
    }
  \]
The arguments used in the proof of \prettyref{thm:signature-compatible-with-differential} therefore yield the following generalization.
\begin{corollary}
  The following square commutes:
  \[
    \xymatrix{
      \op{I}^{j+1}(F,\omega_{F/k})
      \ar[r]^-{\partial_{RS}}
      \ar[d]^-{2^{-(j+1)}\sign}
      &
      \op{I}^j(k_\mathfrak{m}, \omega_{k_\mathfrak{m}/k})
      \ar[d]^-{2^{-j}\sign}
      \\
      \op{C}(F_r, \Z({\omega_{F/k}}))
      \ar[r]^-{\partial_{re}}
      &
      \op{C}((k_{\mathfrak m})_r, \Z({\omega_{k_\mathfrak{m}/k}}))
    }
  \]
\end{corollary}

Next, we want to pass to the colimit in the diagram in \prettyref{cor:compatibilityWC}.
Note that multiplication by $2$ on the Witt group is  multiplication by the Pfister form $\pfist{-1}$.
Clearly, the following triangle commutes:
\[
  \xymatrix{
    \op{I}^j(F) \ar[dr]_-{2^{-j}\sign} \ar[rr]^-{\cdot 2} && \op{I}^{j+1}(F) \ar[dl]^-{2^{-(j+1)}\sign}  \\
    & \op{C}(F_r, \Z)
  }
\]
The induced morphism
\begin{align*}
  \sign_\infty\colon &\colim_{j} \op{I}^{j}(F) \rightarrow \op{C}(F_r,\Z)
\end{align*}
is an isomorphism by a result of Arason and Knebusch (cf.\ \cite[Proposition~2.7]{jacobson}).
\begin{corollary}\label{cor:sign_infty-isomorphism-of-complexes}
The signature induces an isomorphism of complexes:
\[
  \colim_j\op{C}^{RS}(X,\I^j,\mathcal L)
  \xrightarrow[\sign_\infty]{\cong}
  \op{C}^{re}(X,\Z({\mathcal L})).
\]
(In particular, \(\op{C}^{re}(X,\Z({\mathcal L}))\) indeed \emph{is} a complex.)
\end{corollary}
\begin{proof}[Proof of \prettyref{thm:comparison}]
  By replacing \(X\) with open subsets of \(X\), we can view the complexes appearing in \prettyref{cor:sign_infty-isomorphism-of-complexes} as complexes of flasque Zariski sheaves on \(X\).
  The Rost--Schmid complexes of sheaves \(\op{C}^{RS}(-,\I^j,\mathcal L)\) are known to be exact in positive degrees over any smooth scheme \(X\) over a field $k$.  Passing to the colimit, and using the isomorphism \(\sign_\infty\), we find that the complex of sheaves \(\op{C}^{re}(-,\Z,\mathcal L)\) is likewise exact in positive degrees over such \(X\), and hence provides a flasque resolution of the kernel of its first differential.  The identification of this kernel with $\support_* \Z$ follows directly from \cite[Proposition~4.7 and Lemma~4.9]{jacobson}.
\end{proof}

Now we consider cohomology with support in a closed subset $Z\subset X$.
Recall
\ARXIVONLY{from \prettyref{rem:cohomology-with-support-via-flasque-resolution} }%
that, given an abelian sheaf $\sheaf F$ on $X$ with a flasque resolution \(\sheaf F\to \sheaf F^\bullet\), the cohomology with support $\op{H}^s_Z(X, \sheaf F)$ can be computed as the homology of the kernel of the restriction morphism $\sheaf F^\bullet(X) \rightarrow \sheaf F^\bullet(X\setminus Z)$.
In our situation, we obtain:
\begin{corollary}\label{cor:isosignature}
  Let $X$ be a smooth scheme over a field $k$, and let $Z$ be a closed subset of $X$. The signature map
  \[
    \sign_\infty \colon \op{H}^i_Z(X, \colim\limits \I^j(\mathcal{L})) \rightarrow \op{H}^i_Z(X, \support_*\Z(\mathcal{L}))
  \]
  is an isomorphism.
\end{corollary}
In the following corollary, we write \(\realcycle_\infty\) for the composition
\begin{equation}
  \label{eq:real-cycle-infty}
  \begin{aligned}
    \op{H}^i_Z(X, \colim\limits \I^j(\mathcal{L})) &\xrightarrow{\sign_\infty} \op{H}^i_Z(X, \support_*\Z(\mathcal{L})) \overset{\support^*}{\text{---------}}\\
    \longrightarrow
    &\op{H}^i_{Z_r}(X_r, \Z(\mathcal{L}))  \xrightarrow{\iota^*} \op{H}^i_{Z(\R)}(X(\R), \Z(\mathcal{L})).
  \end{aligned}
\end{equation}
The real cycle class map \(\op{H}^i_Z(X,\I^j)\to\op{H}^i_{Z(\R)}(X(\R), \Z(\mathcal{L}))\) is obtained by composing \(\realcycle_\infty\) with the canonical map \(\op{H}^i(X,\colim\limits \I^j)\to \op{H}^i_Z(X,\I^j)\) (see diagram~\eqref{big-diagram}).
\begin{corollary}\label{cor:cycle_infty-iso} For any a smooth scheme \(X\) over $\mathbb{R}$, and any closed subset $Z$ of $X$, the map \(\realcycle_\infty\colon \op{H}^i_Z(X, \colim\limits \I^j(\mathcal{L})) \to \op{H}^i_{Z(\R)}(X(\R), \Z(\mathcal{L}))\) is an isomorphism.
\end{corollary}
\begin{proof}
  The first map in \eqref{eq:real-cycle-infty} is an isomorphism by \prettyref{cor:isosignature}.
  The last map is an isomorphism by 
  Lemmas~\ref{lem:iota-on-constant-sheaves} and \ref{lem:general-twist-of-structure-morphism}: \(\iota_*\Z(\mathcal L) \cong \Z(\mathcal L)\).
  It remains to see that the central map $\support^*$ is an isomorphism.  This is proved for unrestricted support in \cite[Lemma~4.6]{jacobson}, and the argument given there can be generalized to arbitrary supports.  Indeed, recall that the cohomology groups with support in \(Z\) are defined via the right derived functors of \(\Gamma_Z\). We thus have a Grothendieck spectral sequence
  \[
    \op{H}^p_Z(X, \op{R}^q\support_*\Z({\mathcal L})) \Rightarrow \op{R}^{p+q}(\Gamma_Z\support_*)(\Z({\mathcal L})).
  \]
  By \prettyref{lem:pullback-with-support}, we can identify \(\Gamma_Z\support_*\) with \(\Gamma_{Z_r}\), hence we may identify the target of the spectral sequence with \(\op{H}^{p+q}_{Z_r}(X_r,\Z({\mathcal L}))\).
  On the other hand, by \cite[Theorem~19.2]{scheiderer} the higher direct images $\op{R}^q\support_*$ for \(q>0\) vanish, so the spectral sequence collapses. As in the case without supports (see Section \ref{sec:prelim-ret}), one may check that the edge map is given by $\support^*$.
\end{proof}
\subsection{Compatibility with pushforwards: an alternative proof}

\subsubsection{The real Thom class and an alternative proof of \prettyref{thm:pushforwards-compatible}}
\begin{lemma}\label{lem:real-Thom-class}
  Let $p\colon V\to X$ be a vector bundle of rank $d$ over a smooth variety $X$ over $\R$.
  The image of the algebraic Thom class $\op{th}(V)$ of \prettyref{def:algebraic-thom-class} under the real cycle class map
  \[
    \op{H}^d_X(V, \I^d(p^*\det V^\vee)) \longrightarrow \op{H}^d_{X(\R)}(V(\R), \Z(p^*\det V^\vee))
  \]
  is a Thom class of the continuous vector bundle \(p(\R)\colon V(\R)\to X(\R)\) in the sense of \prettyref{def:topological-thom-class}.
\end{lemma}
\begin{proof}
  Let $\op{th}^{\op{cl}}(V)$ be the image in $\op{H}^d_{X(\R)}(V(\R), \Z(p^*\det V^\vee))$ of the algebraic Thom class under the real cycle class map. Consider the following square:
  \[
    \xymatrix@C=5em{
      \op{H}^i(X, \colim \I^j ) \ar[r]^-{\op{th}(V) \cup p^*(-) }
      \ar[d]^{\realcycle_\infty}
      &
      \op{H}^{i+d}_X(V, \colim \I^{j+d}(p^*\det V^\vee) )
      \ar[d]^{\realcycle_\infty}
      \\
      \op{H}^i(X(\R), \Z )
      \ar[r]^-{\op{th}^{\op{cl}}(V) \cup p^*(-)}
      &
      \op{H}^{i+d}_{X(\R)}(V(\R), \Z(p^*\det V^\vee) )
    }
  \]
  The square commutes as the real cycle class map is compatible with products (\prettyref{prop:cup-products-compatible}) and pullbacks (\prettyref{prop:pullbacks-compatible}). The vertical maps are isomorphisms by \prettyref{cor:cycle_infty-iso}.
  Multiplication with the Thom class \(\op{th}(V)\) can be identified with the d\'evissage isomorphism, as noted in the proof of \prettyref{lem:pushforwards-agree}, so by \prettyref{lem:cohomology-of-colimits} the top horizontal map is a colimit of isomorphisms.  It follows that the top horizontal map is itself an isomorphism, hence the lower horizontal map is an isomorphism, too.

  This applies, in particular, to the restricted bundle \(V_x\to \Spec(\R)\) over any real point $x\colon \Spec(\R) \rightarrow X$. So all maps in the following commutative square are isomorphisms:
  \[
    \xymatrix@C=5em{
      \op{H}^i(x, \colim \I^j )
      \ar[r]^-{\op{th}(V_x) \cup p^*(-) }
      \ar[d]^-{\realcycle_\infty}
      &
      \op{H}^{i+d}_x(V_x, \colim \I^{j+d}(p^*\det V_x^\vee) )
      \ar[d]^-{\realcycle_\infty}
      \\
      \op{H}^i(x(\R), \Z )
      \ar[r]^-{\op{th}^{\realcycle}(V_x) \cup p^*(-)}
      &
      \op{H}^{i+d}_{x(\R)}(V_x(\R), \Z(p^*\det V_x^\vee ))
    }
  \]
  As the Thom class \(\op{th}(V_x)\) of the restricted bundle is the pullback of \(\op{th}(V)\) along the inclusion \(V_x\hookrightarrow V\), and as the real cycle class map is compatible with pullbacks, we find that \(\op{th}^{\realcycle}(V_x)\) is likewise the pullback of \(\op{th}^{\realcycle}(V)\) along the inclusion \(V_x(\R)\hookrightarrow V(\R)\).  By taking $i=0$ in the above diagram, we see that the lower horizontal map sends $1$ to a generator. So,   \(\op{th}^{\realcycle}(V)\) is a Thom class in the sense of
  \prettyref{def:topological-thom-class}, as claimed.
\end{proof}
\begin{proof}[Alternative proof of \prettyref{prop:thom-class-comparison}]
\prettyref{prop:thom-class-comparison} follows from \prettyref{lem:real-Thom-class} and the compatibility of the canonical map
\(
  \op{H}^i_Z(X,\I^j(\mathcal{L})) \rightarrow \op{H}^i_Z(X, \colim \I^j(\mathcal{L}))
\)
with the cup product.
\end{proof}
\prettyref{thm:pushforwards-compatible} now follows exactly as before.

 \begin{remark}
Given \prettyref{prop:thom-class-comparison}, which we have just reproved, the main remaining point is to see that the real cycle class map is compatible with deformation to the normal bundle. Using the same notation as in the first proof of \prettyref{prop:comparedefconetubnbhd}, we have the following commutative diagram:
\[
\xymatrix{
\op{H}^i_Z(\Nb_ZX, \colim \I^j(\mathcal{L}))
\ar[r]^-{i_0^*}
\ar[d]^-{\sign}
&
\op{H}^i_{Z\times \mathbb{A}^1}(D(Z,X), \colim \I^j(\mathcal{L}))
\ar[d]^-{\sign}
&
\ar[l]_-{i_1^*}
\op{H}^i_Z(X, \colim \I^j(\mathcal{L}))
\ar[d]^-{\sign}
\\
\op{H}^i_{Z(\mathbb{R})}(\Nb_ZX (\mathbb{R}), \Z({\mathcal{L}}))
\ar[r]^-{i_0^*}
&
\op{H}^i_{(Z\times \mathbb{A}^1) (\mathbb{R})}(D(Z,X)(\mathbb{R}), \Z({\mathcal{L}}))
&
\ar[l]_-{i_1^*}
\op{H}^i_{Z(\mathbb{R})}(X(\mathbb{R}), \Z({\mathcal{L}}))
}
      \]
     As in the proof of \prettyref{lem:real-Thom-class}, the horizontal maps are isomorphisms by the well-known construction of the deformation to the normal bundle (see the discussion above \prettyref{prop:comparedefconetubnbhd}) and by \prettyref{lem:cohomology-of-colimits}, and the vertical maps are isomorphisms by \prettyref{cor:cycle_infty-iso}.
    We conclude that the lower horizontal maps are also isomorphisms.
   This is independent of any comparison of the deformation to the normal bundle with tubular neighborhood constructions.

  We could have \emph{defined} the pushforward in topology as the cup product with a Thom class followed by the deformation to the normal bundle, and this would have been enough for our purposes since the two localization sequences in topology (with values in the locally constant sheaf $\Z(\mathcal{L})$) and in algebraic geometry (with values in the sheaf $\colim \I^j (\mathcal{L})$) can be identified via the cycle class map.  However, the identification of the pushforward defined in this way with more common constructions of pushforwards in singular cohomology requires arguments of the kind presented in the the proof of \prettyref{prop:comparedefconetubnbhd} below.
 \end{remark}


\begin{thebibliography}{BG+02}

\bibitem[Ada74]{adams}
  J.\ F.\ Adams.
  Stable Homotopy and Generalised Homology.
  Chicago Lectures in Mathematics, 1974.

\bibitem[ABF17]{ABF:models}
  A.\ Asok, B.\ Doran, J.\ Fasel: Smooth models of motivic spheres and the clutching construction, IMRN 2017, no.6, 1890--1925.

\bibitem[AF14]{AsokFaselSpheres}
  A.\ Asok and J.\ Fasel.
  Algebraic vector bundles on spheres.
  J.\ Topology 7 (2014), 894--926.

\bibitem[AF16]{asok-fasel:euler}
  A.\ Asok and J.\ Fasel.
  Comparing Euler classes.
  Q.\ J.\ Math.\ 67 (2016), no.~4, 603--635.

\bibitem[Bac18]{bachmann}
  T.\ Bachmann.
  Motivic and real \'etale stable homotopy theory.
  Compos.\ Math.\ 154 (2018), 883--917.

\bibitem[Ba00]{Ba00}
  P.\ Balmer.
  Triangular Witt groups. Part I: The 12-term localization exact sequence.
  K-Theory 19, no 4 (2000), 311--363.
  
\bibitem[BH61]{borel-haefliger}
  A.\ Borel and A.\ Haefliger.
  La classe d'homologie fondamentale d'un espace analytique.
  Bull.\ Soc.\ Math.\ France 89 (1961), 461--513.

\bibitem[BW02]{BW02}
  P.\ Balmer and C.\ Walter.
  A Gersten--Witt spectral sequence for regular schemes.
  Ann.\ Scient.\ ENS 35 (2002), 127--152.

\bibitem[BG+02]{bgpw}
  P.\ Balmer, S.\ Gille, I.\ Panin, C.\ Walter.
  The Gersten conjecture for Witt groups in the equicharacteristic case.
  Doc.\ Math.\ 7 (2002), 203--217.

\bibitem[BW18]{benoist:wittenberg}
  O.\ Benoist and O.\ Wittenberg.
  On the integral Hodge conjecture for real varieties I.
  Preprint, arXiv:1801.00872v1.
  
\bibitem[Bre97]{bredon}
  G.\ E.\ Bredon.
  Sheaf Theory.
  GTM vol.~170, Second Edition.
  Springer, 1997.

\bibitem[Bru84]{brumfiel}
  G.\ Brumfiel.
  Witt rings and K-theory.
  Rocky Mtn.\ J.\ Math.\ 14 (1984), 733--765.
  
\bibitem[CH11]{calmeshornbostel:pushforward}
  B.\ Calmès and J.\ Hornbostel.
  Push-forwards for Witt groups of schemes.
  Comment.\ Math.\ Helv.\ 86 (2011), 437--468.
  
\bibitem[DK01]{daviskirk} J.\ Davis  and P.\ Kirk.  Lecture notes in algebraic topology. Graduate Studies in Mathematics, 35, AMS, 2001.
\bibitem[Del91]{delfs}
  H.\ Delfs.
  Homology of locally semialgebraic spaces.
  Lecture Notes in Mathematics 1484.
  Springer, 1991.

\bibitem[EKM08]{elman-karpenko-merkurjev}
  R.\ Elman, N.\ Karpenko, A.\ Merkurjev.
  The algebraic and geometric theory of quadratic forms.
  AMS Colloquium Publications, 56. American Mathematical Society, Providence, RI, 2008.

\bibitem[Fas07]{fasel:chowwittring}
  J.\ Fasel.
  The Chow--Witt ring.
  Doc.\ Math.\ 12 (2007), 275--312.

\bibitem[Fas08]{fasel:memoir}
  J.\ Fasel.
  Groupes de Chow--Witt.
  M{\'e}m.\ Soc.\ Math.\ Fr.\ (N.S.) no.~113, 2008.

\bibitem[Fas09]{fasel:excess}
  J.\ Fasel.
  The excess intersection formula for Grothendieck--Witt groups.
  Manuscripta Math.\ 130 (2009), no.~4, 411--423.

\bibitem[Fas11]{fasel:orbits}
  J.\ Fasel.
  Some remarks on orbit sets of unimodular rows.
  Comment.\ Math.\ Helv.\ 86 (2011), 13--39.

\bibitem[Fas13]{fasel:ij}
  J.\ Fasel.
  The projective bundle theorem for $\mathbf{I}^j$-cohomology.
  J.\ K-theory 11 (2013), 413--464.


\bibitem[Ful98]{fultonbook}
  W.\ Fulton.
  Intersection theory (2nd edition).
  Ergebnisse der Mathematik und ihrer Grenzgebiete 2.
  Springer, Berlin, 1998.

\bibitem[Gil02]{gille:support}
  S.\ Gille.
  On Witt groups with support.
  Math.\ Ann.\ 322 (2002), 103--137.
  
\bibitem[Gil07a]{gille}
  S.\ Gille.
  A graded Gersten--Witt complex for schemes with a dualizing complex and the Chow group.
  J.\ Pure Appl.\ Algebra 208 (2) (2007), 391--419.

\bibitem[Gil07b]{Gil07b}
  S.\ Gille.
  The general d\'evissage theorem for Witt groups of schemes.
  Arch.\ Math.\ 88 (2007), 333--343.

\bibitem[GN01]{GN:pairings}
  S.\ Gille and A.\ Nenashev.
  Pairings in triangular Witt theory.
  Journal of Algebra 261 (2003), 292--309.
  
\bibitem[God73]{godement}
  R.\ Godement.
  Topologie alg\'{e}brique et th\'{e}orie des faisceaux.
  Hermann, Paris 1964.

\bibitem[Ha77]{hartshorne}
  R.\ Hartshorne.
  Algebraic geometry.
  Graduate Texts in Mathematics, no.~52. Springer, 1977.

\bibitem[Ha01]{hatcher}
  A.\ Hatcher.
  Algebraic Topology.
  Cambridge University Press, 2001.

\bibitem[Hi76]{hirsch}
  M.\ W.\ Hirsch.
  Differential Topology.
  GTM vol.~33. Springer, 1976.

\bibitem[HW19]{hornbostelwendt}
  J.\ Hornbostel and M.\ Wendt.
  Chow--Witt rings of classifying spaces for symplectic and special linear groups.
  J.\ Topol.\ 12 (3) (2019), 915--965. 

\bibitem[HXZ18]{HXZ:quadrics}
  J.\ Hornbostel, H.\ Xie and M.\ Zibrowius.
  Chow--Witt rings of split quadrics.
  Preprint, arXiv:1811.12685.
  To appear in \textit{Motivic homotopy theory and refined enumerative geometry}, Proceedings of a conference in Essen 2018, Contemporary Mathematics.

\bibitem[Hu94]{husemoller}
  D.\ Husemoller.
  Fibre bundles.
  Third edition. Graduate Texts in Mathematics, 20.
  Springer-Verlag, New York, 1994.

\bibitem[Iv86]{iversen}
  B.\ Iversen.
  Cohomology of sheaves.
  Universitext. Springer-Verlag, Berlin, 1986.

\bibitem[Jac15]{jacobson-pre2}
  J.\ A.\ Jacobson.
  From the global signatures to higher signatures.
  Preprint, arXiv:1411.0993v2 (2015).

\bibitem[Jac17]{jacobson}
  J.\ A.\ Jacobson.
  Real cohomology and the powers of the fundamental ideal in the Witt ring.
  Ann.\ K-Theory 2 (2017), 357--385.

\bibitem[Joh73]{johnstone}
  P.\ T.\ Johnstone.
  Topos theory.
  London Mathematical Society Monographs, vol.~10.
  Academic Press [Harcourt Brace Jovanovich, Publishers], London-New York, 1977.

\bibitem[KSW16]{karoubi-schlichting-weibel}
  M.\ Karoubi, M.\ Schlichting and C.\ Weibel.
  The Witt group of real algebraic varieties.
  J.\ Topol.\ 9 (2016), no.~4, 1257--1302. 
  
\bibitem[Kra91]{krasnov}
  V.\ A.\ Krasnov.
  Characteristic classes of vector bundles on a real algebraic variety. Izv. Akad. Nauk SSSR Ser. Mat. 55 (1991), no.~4, 716--746.

\bibitem[Lei04]{leinster}
  T.\ Leinster.
  Higher operads, higher categories.
  London Mathematical Society Lecture Note Series, 298. Cambridge University Press, Cambridge, 2004.

\bibitem[Mac71]{maclane}
  S.\ Mac Lane.
  Categories for the working mathematician.
  Graduate Texts in Mathematics, vol.~5.
  Springer-Verlag, New York-Berlin, 1971.

\bibitem[MS74]{milnorstasheff}
  J.\ Milnor and J.\ Stasheff.
  Characteristic classes.
  Annals of Math.\ Studies, vol.~76, 1974.

\bibitem[Mor04]{morel:puissances}
  F.\ Morel.
  Sur les puissances de l'id{\'e}al fondamental de l'anneau de Witt.
  Comment.\ Math.\ Helv.\ 79 (2004), no.~4, 689--703.

\bibitem[Mor12]{MField}
  F.\ Morel.
  $\mathbb{A}^1$-algebraic topology over a field.
  Lecture Notes in Mathematics 2052. Springer, Heidelberg, 2012.

\bibitem[Nen07]{nenashev}
  A.\ Nenashev.
  Gysin maps in Balmer--Witt theory.
  J.\ Pure Appl.\ Algebra 211 (2007), no.~1, 203--221.

\bibitem[stacks]{stacks}
  The stacks project.
  \href{https://stacks.math.columbia.edu/}{https://stacks.math.columbia.edu/}

\bibitem[Scha85]{Schar}
  W.\ Scharlau.
  Quadratic and Hermitian forms.
  Springer 1982.

\bibitem[Schm98]{Schmidt}
M.\ Schmid.
Wittringhomologie,
thesis available https://www.math.uni-bielefeld.de/~rost/schmid.html (1998).

\bibitem[Sch94]{scheiderer}
  C.\ Scheiderer.
  Real and \'etale cohomology.
  Lecture Notes in Mathematics 1588, Springer, 1994.

\bibitem[Sch95]{scheiderer:purity}
  C.\ Scheiderer.
  Purity theorems for real spectra and applications.
  Real analytic and algebraic geometry (Trento, 1992), 229--250, de Gruyter, Berlin, 1995.

\bibitem[SGA4]{SGA4.I}
  Th\'{e}orie des topos et cohomologie \'{e}tale des sch\'{e}mas.
  Tome 1:
  Th\'{e}orie des topos.
  (S\'{e}minaire de G\'{e}om\'{e}trie Alg\'{e}brique du Bois-Marie 1963--1964 (SGA 4);
  Dirig\'{e} par M.\ Artin, A.\ Grothendieck, et J.\ L.\ Verdier. Avec la collaboration de N.\ Bourbaki, P.\ Deligne et B.\ Saint-Donat.)
  Lecture Notes in Mathematics 269.
  Springer, 1972.

\bibitem[Spa66]{spanierbook}
  E.\ H.\ Spanier.
  Algebraic topology.
  Corrected reprint, Springer, Berlin 1981.

\bibitem[Sw75]{switzer}
  R.\ Switzer.
  Algebraic Topology -- Homotopy and Homology.
  Grundlehren der Mathematischen Wissenschaften 212.
  Springer-Verlag, 1975.

\bibitem[tD08]{tomdieck}
  T.\ tom Dieck.
  Algebraic Topology,
  EMS textbooks in Mathematics, 2008.

\bibitem[Tot97]{totaro:mu}
  B.\ Totaro.
  Torsion algebraic cycles and complex cobordism.
  J.\ Amer.\ Math.\ Soc.\ 10 (1997), 467--493.

\bibitem[Wa83]{warner}
  F.\ Warner.
  Foundations of Differentiable Manifolds and Lie Groups.
  GTM vol.~94. Springer, 1983.

\bibitem[Wei94]{weibel}
  C.\ A.\ Weibel.
  An introduction to homological algebra.
  Cambridge Studies in Advanced Mathematics, 38. Cambridge University Press, Cambridge, 1994.

\bibitem[Wen10]{wendt:m-cells}
  M.\ Wendt.
  More examples of motivic cell structures.
  Preprint, arXiv:1012.0454v2.

\bibitem[Wen18]{real-grassmannian}
  M.\ Wendt.
  Chow--Witt rings of Grassmannians.
  Preprint, arXiv:1805.06142v1.

\bibitem[Zib11]{zib:WCCV}
  M.\ Zibrowius.
  Witt groups of complex cellular varieties.
  Doc.\ Math.\ 16 (2011), 465--511.

\end{thebibliography}
\end{document}